\documentclass[a4paper,11pt]{amsart}

\allowdisplaybreaks
\usepackage[latin1]{inputenc}
\usepackage[cyr]{aeguill}
\usepackage[english]{babel}
\usepackage{tikz,hyperref}

\usepackage[T1]{fontenc}
\usepackage{amsmath,bbold}
\usepackage{amsfonts}
\usepackage{amssymb}
\usepackage{amsthm,eepic}
\usepackage{mathrsfs}
\usepackage{enumitem}
\usepackage{graphicx}
\usepackage{footnote}
\usepackage{hyperref,tikz}
\usepackage{newtxtext}
\usetikzlibrary{arrows.meta,
                chains,
                positioning}
\hypersetup{colorlinks=true,    linkcolor=blue,
citecolor=red,       filecolor=BrickRed,       urlcolor=darkgreen
}

\renewcommand{\geq}{\geqslant}
\renewcommand{\leq}{\leqslant}
\newtheorem{thm}{Theorem}
\newtheorem{prop}{Proposition}

\newtheorem{cor}[prop]{Corollary}
\newtheorem{lem}[prop]{Lemma}

\newtheorem{Assumption}{Assumption}

\newcommand{\be}{\begin{equation}}
\newcommand{\ee}{\end{equation}}

\usepackage{pgfplots}

\pgfplotsset{compat=newest}

\newenvironment{Assumptionbis}
  {\addtocounter{Assumption}{-1}%
   \begin{Assumption}}
  {\end{Assumption}}
  
  \newenvironment{Assumptionbisbis}
  {\addtocounter{Assumption}{-1}%
   \begin{Assumption}}
  {\end{Assumption}}
  
   \newenvironment{Assumptionbisbisbis}
  {\addtocounter{Assumption}{-1}%
   \begin{Assumption}}
  {\end{Assumption}}

\colorlet{darkgreen}{green!50!black}
\definecolor{darkseagreen}{rgb}{0.56, 0.74, 0.56}
\definecolor{lightcyan}{rgb}{0.88, 1.0, 1.0}
\definecolor{lightblue}{rgb}{0.68, 0.85, 0.9}
\definecolor{palecerulean}{rgb}{0.61, 0.77, 0.89}
\definecolor{lgreen} {RGB}{180,210,100}
\definecolor{dblue}  {RGB}{20,66,129}
\definecolor{ddblue} {RGB}{11,36,69}
\definecolor{lred}   {RGB}{220,0,0}
\definecolor{nred}   {RGB}{224,0,0}
\definecolor{norange}{RGB}{230,120,20}
\definecolor{nyellow}{RGB}{255,221,0}
\definecolor{ngreen} {RGB}{98,158,31}
\definecolor{dgreen} {RGB}{78,138,21}
\definecolor{nblue}  {RGB}{28,130,185}
\definecolor{jblue}  {RGB}{20,50,100}
\definecolor{Apricot} {RGB}{255, 170, 123} 
\definecolor{dpurple}  {RGB}{53,21,93}

\def\R{{\mathbb R}}

\def\N{{\mathbb N}}
\def\C{{\mathbb C}}
\def\Z{{\mathbb Z}}
\def\1{\mathbb{1}}

\def\E{{\mathbb E}}
\def\eps{{\epsilon}}
\def\P{{\mathbb P}}

\def\cal{\mathcal}

\DeclareMathOperator{\Res}{Res}
\def\Z{\mathbb{Z}}
\def\C{\mathbb{C}}
\def\R{\mathbb{R}}

\renewcommand{\epsilon}{\varepsilon}
\let\e=\epsilon

\usepackage{color}
\definecolor{darkgreen}{rgb}{0,0.4,0}
\definecolor{MyDarkBlue}{rgb}{0,0.08,0.50}
\definecolor{BrickRed}{rgb}{0.65,0.08,0}

\title{Reflected random walks and unstable Martin boundary}

\author{Irina Ignatiouk-Robert}
\address{Universit\'e de Cergy-Pontoise, D\'epartement de math\'ematiques, 2 Avenue Adolphe Chauvin,
95302 Cergy-Pontoise Cedex, France}
\email{irina.ignatiouk@u-cergy.fr}

\author{Irina Kurkova}
\address{Sorbonne Universit\'e, Laboratoire de Probabilit\'es, Statistique et Mod\'elisation, 4 Place Jussieu, 75005 Paris, France}
\email{irina.kourkova@sorbonne-universite.fr}

\author{Kilian Raschel}
\address{Universit\'e d'Angers \& CNRS, Laboratoire Angevin de Recherche en Math\'ematiques, SFR MATHSTIC, 49000 Angers, France}
\email{raschel@math.cnrs.fr}
\thanks{This project has received funding from the European Research Council (ERC) under the European Union's Horizon 2020 research and innovation programme under the Grant Agreement No.\ 759702.}


\date{\today}

\begin{document}

\begin{abstract}
We introduce a family of two-dimensional reflected random walks in the positive quadrant and study their Martin boundary. While the minimal boundary is systematically equal to a union of two points, the full Martin boundary exhibits an instability phenomenon, in the following sense: if some parameter associated to the model is rational (resp.\ non-rational), then the Martin boundary is countable, homeomorphic to $\mathbb Z\cup\{\pm\infty\}$ (resp.\ uncountable, homeomorphic to $\mathbb R\cup\{\pm\infty\}$). Such instability phenomena are very rare in the literature. Along the way of proving this result, we obtain several precise estimates for the Green functions of reflected random walks with escape probabilities along the boundary axes and an arbitrarily large number of inhomogeneity domains. Our methods mix probabilistic techniques and an analytic approach for random walks with large jumps in dimension two.

\medskip

\noindent\textsc{R\'esum\'e.} Nous introduisons une famille de marches al\'eatoires en dimension deux, r\'efl\'echies au bord du quart de plan positif, et \'etudions leur fronti\`ere de Martin. Tandis que leur fronti\`ere minimale est syst\'ematiquement une union de deux points, nous montrons que la fronti\`ere de Martin compl\`ete est intrins\`equement instable, au sens suivant : lorsqu'un certain param\`etre associ\'e au mod\`ele s'av\`ere rationnel (respectivement non rationnel), la fronti\`ere de Martin est alors d\'enombrable et hom\'eomorphe \`a $\mathbb Z\cup\{\pm\infty\}$ (respectivement non d\'enombrable et hom\'eomorphe \`a $\mathbb R\cup\{\pm\infty\}$). De tels ph\'enom\`enes d'instabilit\'e sont rares dans la litt\'erature. Les d\'emonstrations contiennent plusieurs estim\'ees pr\'ecises pour des fonctions de Green de marches al\'eatoires r\'efl\'echies avec probabilit\'e de fuite le long des axes, poss\'edant en outre un nombre infini de domaines d'inhomog\'en\'eit\'e. Nos m\'ethodes m\'elangent des techniques probabilistes avec une approche analytique pour des marches al\'eatoires avec grands pas en dimension deux.
\end{abstract}

\maketitle

\section{Introduction}

\subsection{Martin boundary and instability}
Before formulating our main results, we recall the  definition of Martin boundary and some associated key results in this field.

\subsubsection*{A brief account of Martin boundary theory}
The concept of Martin compactification for countable Markov chains is introduced by Doob~\cite{Do-59}. Its name originates from the work~\cite{Ma-41}, in which Martin studied integral representations of harmonic functions for the classical Laplace operator on Euclidean domains.
Consider a transient, irreducible, sub-stochastic Markov chain $Z=\{Z(n)\}_{n\geq 0}$ on a state space $E\subset \Z^d$, $d\geq1$, with transition probabilities $\{p(x,y)\}_{x,y\in E}$. Given $x,y\in E$, the associated Green function $g(x,y)$ and Martin kernel $k(x,y)$ are respectively defined by
\begin{equation*}
   g\bigl(x,y\bigr) = \sum_{n=0}^\infty \P_x\bigl(Z(n) = y\bigl)\qquad \text{and} \qquad k\bigl(x,y\bigr) = \frac{g\bigl(x,y\bigr)}{g\bigl(x_0,y\bigr)},
\end{equation*}
where $\P_x$ denotes the probability measure on the set of trajectories of $Z$ corresponding to the initial state $Z(0)=x$, and $x_0$ is a given reference point in $E$.

For irreducible Markov chains, the family of functions $\{k(\cdot, y) \}_{y\in E}$ is relatively compact with respect to the topology of pointwise convergence; in other words, for any sequence $\{y_n\}$ of points in $E$, there exists a subsequence $\{y_{n_k}\}$ along which the sequence $k(\cdot, y_{n_k})$ converges pointwise on $E$. The Martin  compactification $E_M$ is defined as the (unique) smallest compactification of $E$ such that the Martin kernels $k(x,\cdot)$  extend continuously; a sequence $\{y_n\}_{n\geq 0}$ converges to a point of the Martin boundary $\partial_M E = E_M\setminus E$ if it leaves every finite subset of $E$ and if the sequence of functions $k(\cdot,y_n)$ converges pointwise.  

Recall that a function $h : E\to\R$ is harmonic for $Z$ if, for all $x\in E$, $\E_x\bigl(h(Z(1))\bigr) = h(x)$. By the Poisson-Martin representation theorem, for every non-negative harmonic function $h$, there exists a positive Borel measure $\nu$ on $\partial_M E$ such that
\begin{equation*}
   h(x) = \int_{\partial_M E} k\bigl(x,\eta\bigr) d\nu(\eta).
\end{equation*}
The convergence theorem says that for any $x\in E$, the sequence $\{Z(n)\}_{n\geq 0}$ converges $\P_x$-almost surely to a $\partial_M E$-valued random variable. The Martin boundary therefore provides all non-negative harmonic functions and shows how the Markov chain $Z$ goes to infinity.  

In order to identify the Martin boundary, one has to investigate all possible limits of the Martin kernel $k\bigl(x,y_n\bigr)$ when $\vert y_n\vert\to\infty$. As a consequence, the identification of the Martin boundary often reduces to the asymptotic computation of the Martin kernel or of the Green function. Such results are now well established for spatially homogeneous random walks, see, e.g., Ney and Spitzer~\cite{NeSp-66} for $E=\Z^d$, see also the book of Woess~\cite{Wo-00} for many relevant references. On the other hand, as we shall see later in this introduction, the case of inhomogeneous random walks is much more involved and still largely open.

\subsubsection*{Instability of the Martin boundary}
Once the notion of Martin boundary has been settled, it is natural to ask how stable it is with respect to the parameters of the model. Roughly speaking, throughout this paper, the Martin boundary will be called stable if the Martin compactification does not depend on small modifications of the transition probabilities.

Although we shall not use it in the present paper, let us recall the historically first way to measure the stability of the Martin boundary, which has been introduced by Picardello and Woess \cite{PiWo-92}. Define the $\rho$-Green function by 
\begin{equation*}
   g\bigl(x,y;\rho\bigr) = \sum_{n=0}^\infty \P_x\bigl(Z(n) = y\bigl)\rho^{-n},
\end{equation*}
for $\rho$ in the spectral interval. Then one may define the associated $\rho$-Martin kernel and $\rho$-Martin compactification. According to \cite[Def.~2.4]{PiWo-92}, the Martin boundary is stable if the $\rho$-Martin compactification does not depend on the eigenvalue $\rho$ (with a possible exception at the critical value) and if the Martin kernels are jointly continuous w.r.t.\ space variable and eigenvalue.

The majority of known examples of Martin boundary is stable, see for example \cite{PiWo-90}. More precisely, to the best of our knowledge, the only known Markov process with unstable Martin boundary is a model of reflected random walk \cite{IR-10a}, as worked out by the first author of this paper.

\subsubsection*{Contribution}

Our main objective is to 
shed light on a rare instability phenomenon of the Martin boundary, in the framework of two-dimensional reflected random walks. 
More specifically, we propose a family of probabilistic models for which the arithmetic nature (rational vs.\ non-rational) of some parameter has a strong influence on the topology of the Martin boundary. 
Precisely, there is a directional angle $\gamma_0$ associated with the random walk, see \eqref{eq:def_gamma_0} below, such that for rational $t_0 = \tan\gamma_0$, the boundary is countable, while it is uncountable when $t_0$ is non-rational.
As the quantity $t_0$ is locally analytic in the transition probabilities (viewed as variables), one immediately deduces that the Martin boundary is unstable (in our definition). See Theorem~\ref{thm:main-4} for a precise statement. We would like to emphasize the following two differences w.r.t.\ the only known other example of instability:
\begin{itemize}
   \item The instability in \cite{IR-10a} follows from an interplay between boundary and interior parameters. Our example only concerns interior parameters and is, from that point of view, more intrinsic. 
   \item Being analytic and non-constant, the function $t_0$ in \eqref{eq:def_gamma_0} takes infinitely many rational and non-rational values, and so the Martin boundary ``jumps'' infinitely many times from $\mathbb Z\cup\{\pm\infty\}$ to $\mathbb R\cup\{\pm\infty\}$. On the other hand, in the example found in \cite{IR-10a}, there is somehow only one jump, meaning that the spectral interval my be divided into two subintervals, and within each of them the Martin boundary is stable.
\end{itemize}

\subsection{Multidimensional reflected random walks in complexes}

This section is split into three parts. We shall first introduce the class of reflected random walks in the positive quadrant, on which we shall work throughout the paper and prove the instability phenomenon described above. 
We will next view this family of models within
a larger class of multidimensional reflected random walks in complexes, in relation with many probabilistic questions 
of the last fifty years. Finally, we will review the literature from the viewpoint of Martin boundary and Green functions of reflected random walks (and more generally of non-homogeneous Markov chains).

\subsubsection*{Our model}
Without entering into full details (more are to come in Section~\ref{sec:model}, where the model will be carefully introduced), we now define the main discrete stochastic process studied throughout this paper: it is a random walk in dimension two, reflected at the boundary of the quarter plane, with a finite (but arbitrarily large) number of homogeneity domains, see Figure \ref{fig:model}. These domains are either points, half-lines, or a (translated) quadrant. Within each homogeneity domain, the model admits 
spatially homogeneous
transition probabilities. Jumps into positive (North-East) directions may be unbounded, while we will place some restrictions on the size of the negative (South-West) jumps. See Assumptions~\ref{as1}, \ref{as2} and \ref{as3} in Section~\ref{sec:model} for more details. We shall fix the parameters so as to have a transient Markov process, with escape probabilities along the axes, see Figure~\ref{fig:paths}.

\subsubsection*{A general model of reflected random walks}
Our probabilistic model belongs to a more general class of piecewise homogeneous models in multidimensional domains (the domains being usually simple cones, such as half-spaces or orthants, or union of cones, called complexes). This class of models is of great interest, as it is much richer than its fully homogeneous analogue, but still admits structured inhomogeneities, opening the way for a detailed analysis. Reflected random walks in the half-line (in dimension $1$) are particularly studied in the literature, and so is the two-dimensional model of random walks in the quarter plane; for historical references, see Malyshev \cite{Ma-70,Ma-72,Ma-73}, Fayolle and Iasnogorodski \cite{FaIa-79}, see also \cite{FaIaMa-17}, Cohen and Boxma \cite{CoBo-83}.

Many probabilistic features of these models are investigated in the literature, for instance in relation with the classification of Markov chains (recurrence and transience), see \cite{FaMaMe-95}. Another strong motivation to their study comes from their links with queueing systems \cite{CoBo-83,Co-92}. Furthermore, these models offer the opportunity to develop remarkable tools, 
for example from
complex analysis \cite{CoBo-83,Co-92,KuMa-98,FaIaMa-17} and asymptotic analysis \cite{KuMa-98}. Many combinatorial objects (maps, permutations, trees, Young tableaux, partitions, walks in Weyl chambers, queues, etc.)\ can be encoded by lattice walks, see \cite{BMMi-10} and references therein, so that understanding the latter, we will better understand the 
objects. 
Let us also mention the article \cite{DeWa-15} by Denisov and Wachtel, which contains several fine estimates of exit times and local probabilities in cones. Finally, as it will be clear in the next paragraph, many interesting questions related to potential theory arise in the study 
of these random walk models which are not spatially homogeneous.

\subsubsection*{Martin boundary theory for non-homogeneous Markov chains}
For such processes, the problem of explicitly describing the Martin compactification, or the Green functions asymptotics, is usually highly complex, and only few results are available in the literature.
For random walks on non-homogeneous trees, the Martin boundary is described by Cartier~\cite{Ca-71}; for random walks on trees with bounded jumps, see \cite{PiWo-87}; for random walks on hyperbolic graphs with spectral radius less than one, see \cite{An-88}.

Alili and Doney~\cite{AlDo-01} identify the Martin boundary for a space-time random walk $S(n)=(Z(n),n)$, when $Z$ is a homogeneous random walk  on $\Z$ killed when hitting the negative half-line. Doob~\cite{Do-59} identifies the Martin boundary for Brownian motion on a half-space, by using an explicit form of the Green function. 
The results in \cite{Do-59,AlDo-01} are obtained by using the one-dimensional structure of the process.

In dimension two, a complex analysis method (based on the study of elliptic curves) is proposed by Kurkova and  Malyshev~\cite{KuMa-98} to identify and classify the Martin boundary for reflected nearest neighbor random walks with drift in $\N\times\Z$ and $\N^2$ ($\N$ denoting the set of non-negative integers $\{0,1,2,\ldots\}$). In particular, although not put forward explicitly as such, \cite[Thm.~2.6]{KuMa-98} contains an instability result of the Martin boundary (similar to the one we prove in our paper). This analytic approach was actually initially introduced by Malyshev \cite{Ma-70,Ma-72,Ma-73} in his study of stationary distributions of ergodic random walks in the quarter plane. Let us also mention the work \cite{KuRa-11}, where the second and third authors of the present article compute the Martin boundary of killed random walks in the quadrant, developing Malyshev's approach to that context. We emphasize that in \cite{KuMa-98,KuRa-11}, exact asymptotics of the Green function (not only of the Martin kernel) are derived.

In  order to identify the Martin boundary of a piecewise homogeneous (killed or reflected)\ random walk on a half-space $\Z^{d}\times\N$, a large deviation approach combined with Choquet-Deny theory and the ratio limit theorem of Markov-additive processes is proposed by Ignatiouk-Robert~\cite{IR-08,IR-10b,IR-20}, and Ignatiouk-Robert and Loree~\cite{IRLo-10}. It should be mentioned that some key arguments in \cite{IR-08,IR-10b,IRLo-10,IR-20} are valid only for Markov-additive processes, meaning that the transition probabilities are invariant w.r.t.\ translations in some directions. In the previously cited articles, no exact asymptotics for Green functions are derived: only the Martin kernel asymptotics is considered. Building on the estimates of the local probabilities in cones derived in \cite{DeWa-15}, the paper \cite{DuRaTaWa-22} obtains the asymptotics of the Green functions in this context.

\subsection{Advances in the analytic and probabilistic approaches to random walks with large jumps}

\subsubsection*{Progress on the analytical method...}

We shall introduce the generating functions of the Green functions and prove that they satisfy various functional equations, starting from which we will deduce contour integral formulas for the Green functions. Applying asymptotic techniques to these integrals will finally lead to our main results.

Historically, the first techniques developed to study the above-mentioned problems were analytic, see in particular the 
pioneering works by Malyshev \cite{Ma-70,Ma-72,Ma-73}, Fayolle and Iasnogorodski \cite{FaIa-79}. In brief, the main idea consists in working on a Riemann surface naturally associated with the random walk, via the transition probability generating function. In turn, the fine study of the Riemann surface allows to observe and prove various probabilistic 
behavior of the model.

This Riemann surface has genus one in the case of random walks with jumps to the eight nearest neighbors (sometimes called small step random walks). Such Riemann surfaces may be fully studied (e.g., using their parametrization with elliptic functions). This is the case considered in the early works \cite{Ma-70,Ma-72,Ma-73,FaIa-79} as well as in subsequent papers on different models such as \cite{KuMa-98} or \cite{KuSu-03}.

Larger jumps lead to higher genus Riemann surfaces, which become much harder to fully analyze, not to say impossible in general, as explained in the note \cite{FaRa-15}. 
We mention the (combinatorial) paper \cite{BaFl-02} in dimension one,
and \cite{FaRa-15,BoBMMe-21} in dimension two, where some particular cases are studied. In the framework of random walks in the quarter plane with 
arbitrarily big jumps, the books \cite{CoBo-83,Co-92} propose a theoretical study of relevant functional equations, concluding with the same difficulty that in general, one cannot expect a sufficiently precise study of the associated Riemann surface to deduce really explicit results (for instance on the stationary probabilities or Green functions). 

For similar reasons, studying reflected or killed random walks in dimension three seems highly challenging, because of the need of describing the associated algebraic curve.

The main progress 
achieved here is that we are able to dispense with the study of the Riemann surface in its entirety, and to replace it by a local study at only two points, 
which turn out to contain
all the information 
(at least from the asymptotic point of view).
This local study being not at all sensitive to the size of the jumps (or to the genus of the surface), we are able to treat very general random walks in dimension two. 

Our paper is therefore the first step in solving similar problems for random walks with big jumps in dimension two (for instance, join-the-shortest-queue like problems) and in higher dimensions as well, in which the complete description of the algebraic curve is not possible in general.

We also emphasize that the number of homogeneity domains is arbitrarily large.

\subsubsection*{...using probabilistic techniques}

In order to perform this reduction of the whole Riemann surface to two points, we combine the analytic method with several probabilistic estimates. More precisely, we will introduce simpler models, such as reflected random walks in half-spaces, and we will compare the Green functions of these models to those of our main random walk. In particular, we will prove that asymptotically, the main Green function appears a sum of two terms, each of them can be interpreted in the light of half-space Green functions and further quantities related to one-dimensional random walks. 

The combination of analytical and probabilistic aspects becomes apparent on our results:
we derive Theorem~\ref{thm:main-2} by probabilistic methods, then Theorem~\ref{thm:main-3} by analytical methods; the union of the two results gives our main Theorem~\ref{thm:main-4} on the instability of the Martin boundary.


\subsubsection*{Related research}

To conclude this introduction, we highlight future projects in relation with the present work. First, the progress we did in the understanding of the techniques could be applied to various other cases of two-dimensional reflected random walks. In this paper, we choose to focus on a model with escape probabilities along the two axes, because of our initial motivation related to the instability phenomenon. As a second extension of our techniques and results, we would like to look at higher dimensional reflected random walk models. A third project is to study the precise link between our definition of stability and that of Picardello and Woess \cite{PiWo-90}.

\section{The model}
\label{sec:model}

\subsection{The main model}

Consider a random walk $\{Z(n)\}_{n\geq 0}=\{(X(n),Y(n))\}_{n\geq0}$ on $\N^2$ with transition probabilities
\be\label{trans_probabilities}
     p\bigl((i,j)\to (i',j')\bigr) =
     \begin{cases} 
          \mu(i'-i, j'-j) &\text{if $i,j \geq k_0$,}\\
          \mu'_j(i'-i,j'-j) &\text{if $i \geq k_0$ and $0\leq j < k_0$,}\\
          \mu''_i(i'-i,j'-j) &\text{if $0\leq i  < k_0$ and $j \geq k_0$,}\\
          \mu_{ij}(i'-i,j'-j) &\text{if $0\leq i  < k_0$ and $0\leq j < k_0$,}
     \end{cases} 
\ee
where $k_0 > 0$ is a given constant and $\mu, \mu'_j, \mu''_i, \mu_{ij}$ are probability measures on $\Z^2$; see Figure \ref{fig:model}. 
The transition probabilities \eqref{trans_probabilities} are translation invariant in the ``interior'' of the quarter plane, which consists of all points at distance at least $k_0$ from the two half-axes $\mathbb N\times \{0\}$ and $\{0\}\times \mathbb N$. In the ``boundary strips'' $\mathbb N\times \{0,\ldots,k_0-1\}$ and $\{0,\ldots,k_0-1\}\times\mathbb N$, the transition probabilities are locally homogeneous.
We will assume that the following conditions are satisfied. 

\begin{figure}
\begin{tikzpicture}
\begin{scope}[scale=0.7]
\draw[dblue!100, fill=dblue!100] (0,0) rectangle (0.5,0.5);
\draw[dblue!70, fill=dblue!70] (0,0.5) rectangle (0.5,1.5);
\draw[dblue!50, fill=dblue!50] (0,1.5) rectangle (0.5,2.5);
\draw[dblue!70, fill=dblue!70] (0.5,0) rectangle (1.5,0.5);
\draw[dblue!50, fill=dblue!50] (1.5,0) rectangle (2.5,0.5);
\draw[dblue!50, fill=dblue!50] (0.5,0.5) rectangle (1.5,1.5);
\draw[dblue!20, fill=dblue!20] (1.5,1.5) rectangle (2.5,2.5);
\draw[dblue!30, fill=dblue!30] (0.5,1.5) rectangle (1.5,2.5);
\draw[dblue!30, fill=dblue!30] (1.5,0.5) rectangle (2.5,1.5);
\draw[white, fill=ngreen!10] (2.5,2.5) rectangle (9.5,9.5);
\draw[Apricot!100, fill=Apricot!100] (0,2.5) rectangle (0.5,9.5);
\draw[Apricot!55, fill=Apricot!55] (0.5,2.5) rectangle (1.5,9.5);
\draw[Apricot!25, fill=Apricot!25] (1.5,2.5) rectangle (2.5,9.5);
\draw[palecerulean!100, fill=palecerulean!100] (2.5,0) rectangle (9.5,0.5);
\draw[palecerulean!55, fill=palecerulean!55] (2.5,0.5) rectangle (9.5,1.5);
\draw[palecerulean!25, fill=palecerulean!25] (2.5,1.5) rectangle (9.5,2.5);
\draw[white, thick] (-1.5,-1.5) grid (9.5,9.5);
\draw[->] (0,-1.5) -- (0,9.5);
\draw[->] (-1.5,0) -- (9.5,0);
\draw[->, ngreen, semithick] (6,6) -- (8,5);
\draw[->, ngreen, semithick] (6,6) -- (7,3);
\draw[->, ngreen, semithick] (6,6) -- (9,3);
\draw[->, ngreen, semithick] (6,6) -- (9,7);
\draw[->, ngreen, semithick] (6,6) -- (4,9);
\draw[->, ngreen, semithick] (6,6) -- (3,3);
\draw[->, ngreen, semithick] (6,6) -- (4,5);
\draw[->, ngreen, semithick] (6,6) -- (8,8);
\draw[->, ngreen, semithick] (6,6) -- (5,6);
\draw[->, ngreen, semithick] (6,6) -- (6,8);
\draw[->, ngreen, semithick] (6,6) -- (6,5);
\draw[->, black, semithick] (6,1) -- (7,2);
\draw[->, black, semithick] (6,1) -- (6,2);
\draw[->, black, semithick] (6,1) -- (5,4);
\draw[->, black, semithick] (6,1) -- (8,0);
\draw[->, black, semithick] (6,1) -- (4,0);
\draw[->, black, semithick] (1,6) -- (2,7);
\draw[->, black, semithick] (1,6) -- (2,6);
\draw[->, black, semithick] (1,6) -- (3,5);
\draw[->, black, semithick] (1,6) -- (0,7);
\draw[->, black, semithick] (1,6) -- (0,3);
\node[above right, ngreen] at (4.2,6.04) {\tiny{$\mu(i,j)$}};
\node[above right, black] at (4.2,1.04) {\tiny{$\mu'_{1}(i,j)$}};
\node[above right, black] at (0.7,4.3) {\tiny{$\mu''_{1}(i,j)$}};
\end{scope}
\end{tikzpicture}
\caption{Description of the model in the case $k_0=3$. Each colored strip is a homogeneity domain for the transition probabilities \eqref{trans_probabilities}.}
\label{fig:model}
\end{figure}

\begin{Assumption}
\label{as1}
One has
\begin{enumerate}
     \item $\mu(i,j) = 0$ if either $i < -k_0$ or $j< - k_0$;
     \item for any $j\in\{0,\ldots,k_0-1\}$, $\mu'_j(i',j') = 0$ if either $i' < - k_0$ or $j' < -j$;
     \item for any $i\in\{0,\ldots,k_0-1\}$, $\mu''_i(i',j') = 0$ if either $j' < - k_0$ or $i' < -i$;
     \item for any $i,j\in\{0,\ldots,k_0-1\}$, $\mu_{ij}(i',j') = 0$  if either  $i' < -i$ or $j' < -j$.
\end{enumerate}
\end{Assumption}

 \begin{Assumption}
  \label{as2}
For any $i,j\in\{0,\ldots, k_0\}$, $\E_{i,j}\bigl(\vert Z(1)\vert\bigr) <\infty$.
  \end{Assumption}


\subsection{Auxiliary models}
In addition to our main model $\{Z(n)\}$, we introduce three local random walks, $\{Z_0(n)\}$, $\{Z_1(n)\}$ and $\{Z_2(n)\}$, which correspond to the local behavior of the process $\{Z(n)\}$ far away from the boundaries $\N\times\{0,\ldots,k_0-1\}$ and $\{0,\ldots,k_0-1\}\times\N$; see Figure~\ref{fig:model_Z0_Z1_Z2}. These secondary processes will be used both in the statements and in the proofs of the main results.

We first introduce
the classical random walk $\{Z_0(n)\}$ on $\Z^2$ with homogeneous probabilities of transition
\begin{equation}
\label{eq:transitions_Z0}
     p_0\bigl((i,j)\to(i',j')\bigr) =\mu(i'-i,j'-j). 
\end{equation} 
The mean step (or drift vector) of the random walk $\{Z_0(n)\}$ is $m=(m_1,m_2)$, where
\be\label{eq:drift}
m_1 =\sum_{(i,j)\in\Z^2} i \mu(i,j) \quad \text{and} \quad m_2 =\sum_{(i,j)\in\Z^2} j\mu(i,j).
\ee
We further define two Markov-additive processes as follows: first, we denote by $Z_1 = \{Z_1(n)\} = \{(X_1(n),Y_1(n))\}$ a random walk on $\N\times\Z$ with transitions 
\be\label{local_trans_probabilities_1} 
p_1\bigl((i,j)\to(i',j')\bigr) =\begin{cases}
\mu(i'-i, j'-j) &\text{if $i \geq k_0$,}\\
\mu''_i(i'-i,j'-j) &\text{if $0\leq i < k_0$.}
\end{cases} 
\ee
Similarly, we construct the random walk $Z_2 = \{Z_2(n)\}=\{(X_2(n),Y_2(n))\}$ on $\Z\times\N$ with transition probabilities 
\be\label{local_trans_probabilities_2} 
p_2\bigl((i,j)\to(i',j')\bigr) =\begin{cases}
\mu(i'-i, j'-j) &\text{if $j \geq k_0$,}\\
\mu'_j(i'-i,j'-j) &\text{if $0\leq j < k_0$.}
\end{cases} 
\ee
The local random walk $Z_1$ (resp.\ $Z_2$) describes the behavior of the original walk $Z$ far from the boundary $\{(i,j)\in\N^2: 0\leq j< k_0\}$ (resp.\ $\{(i,j)\in\N^2: 0\leq i< k_0\}$). We shall assume the following:
\begin{Assumption}
\label{as3}
The random walks $Z_0$, $Z_1$, $Z_2$ and $Z$ are irreducible on their respective state spaces. 
\end{Assumption}

Observe that for any $j\in\Z$, the transition probabilities of $Z_1$ (resp.\ $Z_2$) are invariant with respect to translations by the vector $(0,j)$ (resp.\ $(j,0)$). According to the classical terminology, $Z_1$ (resp.\ $Z_2$) is a Markov-additive process with Markovian part $\{X_1(n)\}$ (resp.\ $\{Y_2(n)\}$) and additive part $\{Y_1(n)\}$ (resp.\ $\{X_2(n)\}$). The Markovian parts $\{X_1(n)\}$ and $\{Y_2(n)\}$ are Markov chains on $\N$ with respective transition probabilities
\begin{equation}
\label{eq:distrib_p_1}
p_1(i,i') = \sum_{j'\in\Z} p_1\bigl((i,0)\to(i',j')\bigr) = \begin{cases} \sum_{j'\in\Z} \mu(i'-i, j') &\text{if $i \geq k_0$,}\\
\sum_{j'\in\Z} \mu''_i(i'-i,j') &\text{if $0\leq i < k_0$,}
\end{cases} 
\end{equation}
and  
\begin{equation*}
p_2(j,j') = \sum_{i'\in\Z} p_2\bigl((0,j)\to(i',j')\bigr) = \begin{cases} \sum_{i'\in\Z} \mu(i', j'-j) &\text{if $j \geq k_0$,}\\
\sum_{i'\in\Z} \mu'_j(i',j'-j) &\text{if $0\leq j < k_0$.}
\end{cases} 
\end{equation*}
In the book \cite{FaMaMe-95} of Fayolle, Malyshev and Menshikov, the Markovian parts $\{X_1(n)\}$ and $\{Y_2(n)\}$  of the corresponding local processes are called induced Markov chains relative to the boundaries $\{(i,j)\in\N^2 : 0\leq i < k_0\}$ and $\{(i,j)\in\N^2 : 0\leq j < k_0\}$.
We refer to \cite{FaMaMe-95} for further definitions and properties of induced Markov chains. The processes $X_1$ and $Y_2$ inherit an irreducibility property from Assumption \ref{as3}. We will assume moreover that:

\begin{Assumption}
\label{as4} The Markov chains  $X_1$ and $Y_2$ are aperiodic. 
\end{Assumption}

Let us mention the following straightforward result:
\begin{lem}
\label{positive_recurrence_lemma}
Under Assumptions \ref{as1}--\ref{as4}, if $m_1 < 0$ (resp.\ $m_2 < 0$) in \eqref{eq:drift}, the Markov chain $\{X_1(n)\}$ (resp.\ $\{Y_2(n)\}$) is positive recurrent.
\end{lem}

\begin{figure}
\begin{tikzpicture}
\begin{scope}[scale=0.4]
\draw[white, fill=ngreen!10] (-4.5,-4.5) rectangle (4.5,4.5);
\draw[white, thick] (-4.5,-4.5) grid (4.5,4.5);
\draw[->, ngreen, semithick] (0,0) -- (-1,2);
\draw[->, ngreen, semithick] (0,0) -- (-3,1);
\draw[->, ngreen, semithick] (0,0) -- (-3,3);
\draw[->, ngreen, semithick] (0,0) -- (1,3);
\draw[->, ngreen, semithick] (0,0) -- (2,-2);
\draw[->, ngreen, semithick] (0,0) -- (-3,-3);
\draw[->, ngreen, semithick] (0,0) -- (-1,-2);
\draw[->, ngreen, semithick] (0,0) -- (2,2);
\draw[->, ngreen, semithick] (0,0) -- (0,-1);
\draw[->, ngreen, semithick] (0,0) -- (2,0);
\draw[->, ngreen, semithick] (0,0) -- (-1,0);
\node[above right, ngreen] at (1.8,1.8) {\tiny{$\mu(i,j)$}};
\draw[->] (-5,0) -- (5,0);
\draw[->] (0,-5) -- (0,5);
\end{scope}
\end{tikzpicture}
\qquad
\begin{tikzpicture}
\begin{scope}[scale=0.4]
\draw[white, fill=ngreen!10] (2.5,-4.5) rectangle (9.5,4.5);
\draw[Apricot!100, fill=Apricot!100] (0,-4.5) rectangle (0.5,4.5);
\draw[Apricot!55, fill=Apricot!55] (0.5,-4.5) rectangle (1.5,4.5);
\draw[Apricot!25, fill=Apricot!25] (1.5,-4.5) rectangle (2.5,4.5);
\draw[white, thick] (0,-4.5) grid (9.5,4.5);
\draw[->, ngreen, semithick] (6,0) -- (8,-1);
\draw[->, ngreen, semithick] (6,0) -- (7,-3);
\draw[->, ngreen, semithick] (6,0) -- (9,-3);
\draw[->, ngreen, semithick] (6,0) -- (9,1);
\draw[->, ngreen, semithick] (6,0) -- (4,2);
\draw[->, ngreen, semithick] (6,0) -- (3,-3);
\draw[->, ngreen, semithick] (6,0) -- (4,-1);
\draw[->, ngreen, semithick] (6,0) -- (8,2);
\draw[->, ngreen, semithick] (6,0) -- (5,0);
\draw[->, ngreen, semithick] (6,0) -- (6,2);
\draw[->, ngreen, semithick] (6,0) -- (6,-1);
\draw[->, black, semithick] (1,0) -- (2,1);
\draw[->, black, semithick] (1,0) -- (2,0);
\draw[->, black, semithick] (1,0) -- (4,1);
\draw[->, black, semithick] (1,0) -- (0,2);
\draw[->, black, semithick] (1,0) -- (0,-2);
\node[above right, ngreen] at (5.8,1.8) {\tiny{$\mu(i,j)$}};
\node[above right, black] at (-0.2,1.8) {\tiny{$\mu''(i,j)$}};
\draw[->] (0,-5) -- (0,5);
\draw[->] (-0.5,0) -- (10,0);
\end{scope}
\end{tikzpicture}\qquad
\begin{tikzpicture}
\begin{scope}[scale=0.4]
\draw[white, fill=ngreen!10] (-4.5,2.5) rectangle (4.5,9.5);
\draw[palecerulean!100, fill=palecerulean!100] (-4.5,0) rectangle (4.5,0.5);
\draw[palecerulean!55, fill=palecerulean!55] (-4.5,0.5) rectangle (4.5,1.5);
\draw[palecerulean!25, fill=palecerulean!25] (-4.5,1.5) rectangle (4.5,2.5);
\draw[white, thick] (-4.5,0) grid (4.5,9.5);
\draw[->, ngreen, semithick] (0,6) -- (-1,8);
\draw[->, ngreen, semithick] (0,6) -- (-3,7);
\draw[->, ngreen, semithick] (0,6) -- (-3,9);
\draw[->, ngreen, semithick] (0,6) -- (1,9);
\draw[->, ngreen, semithick] (0,6) -- (2,4);
\draw[->, ngreen, semithick] (0,6) -- (-3,3);
\draw[->, ngreen, semithick] (0,6) -- (-1,4);
\draw[->, ngreen, semithick] (0,6) -- (2,8);
\draw[->, ngreen, semithick] (0,6) -- (0,5);
\draw[->, ngreen, semithick] (0,6) -- (2,6);
\draw[->, ngreen, semithick] (0,6) -- (-1,6);
\draw[->, black, semithick] (0,1) -- (1,2);
\draw[->, black, semithick] (0,1) -- (-1,2);
\draw[->, black, semithick] (0,1) -- (1,4);
\draw[->, black, semithick] (0,1) -- (2,0);
\draw[->, black, semithick] (0,1) -- (-2,0);
\node[above right, ngreen] at (1.8,5.8) {\tiny{$\mu(i,j)$}};
\node[above right, black] at (1.8,-0.2) {\tiny{$\mu'(i,j)$}};
\draw[->] (-5,0) -- (5,0);
\draw[->] (0,-0.5) -- (0,10);
\end{scope}
\end{tikzpicture}
\caption{Description of the auxiliary models $Z_0$, $Z_1$ and $Z_2$.}
\label{fig:model_Z0_Z1_Z2}
\end{figure}

Assume $m_1 < 0$ and $m_2 < 0$, denote by $\{\pi_1(i)\}_{i\in\N}$ and $\{\pi_2(j)\}_{j\in\N}$ the invariant distributions of $\{X_1(n)\}$ and $\{Y_2(n)\}$, and introduce the quantities
\begin{equation}
\label{definition_vect_2}
     \left\{\begin{array}{lll}
     V_1 &=& \displaystyle \sum_{i=0}^\infty \pi_1(i) \E_{(i,0)} \bigl(Y_1(1)\bigr),\medskip\\
     V_2 &=& \displaystyle \sum_{j=0}^\infty \pi_2(j) \E_{(0,j)} \bigl(X_2(1)\bigr).
     \end{array}\right.
\end{equation}
According to the definition of the transition probabilities \eqref{eq:distrib_p_1} and thanks to Assumption~\ref{as2}, we have
\begin{equation*}
\sum_{i=0}^\infty \pi_1(i) \E_{(i,0)} \bigl(|Y_1(1)|\bigr) = \displaystyle \sum_{i=0}^{k_0-1} \pi_1(i) \E_{(i,0)} \bigl(|Y_1(1)|\bigr) + \left(1- \sum_{i=0}^{k_0-1} \pi_1(i)\right)\E_{(k_0,0)} \bigl(|Y_1(1)|\bigr)< \infty.
\end{equation*} 
The quantity $V_1$ is therefore well defined.  By  Theorem~12 of Prabhu, Tang and Zhu~\cite{Prabhu-Tang-Zhu},  $V_1$ represents the velocity of the fluid limit of the local random walk $Z_1$: for any $(k,\ell)\in\N\times\Z$, $\P_{(k,\ell)}$-a.s., 
\begin{equation}\label{fluid_limit_eq} 
   \lim_{n\to \infty} \frac{Y_1(n)}{n} = V_1. 
\end{equation}
By symmetry, analogous results hold for $X_2(n)$ and $V_2$.

\section{Results}
\label{sec:main_results}

\subsection{Statements}
Introduce two further assumptions:
\begin{Assumption}
\label{as5}
The coordinates $m_1,m_2$ of the drift vector \eqref{eq:drift} are both negative.
\end{Assumption}
\begin{Assumption}
\label{as6}
The quantities $V_1,V_2$ in \eqref{definition_vect_2} are positive.
\end{Assumption}

In particular, Assumptions \ref{as5} and \ref{as6} imply that the reflected random walk $\{Z(n)\}$ is transient, with possible escape at infinity along each axis, see Figure \ref{fig:paths} and Proposition \ref{sec4_lem1}.
\begin{figure}[ht!]
\includegraphics[width=0.4\textwidth]{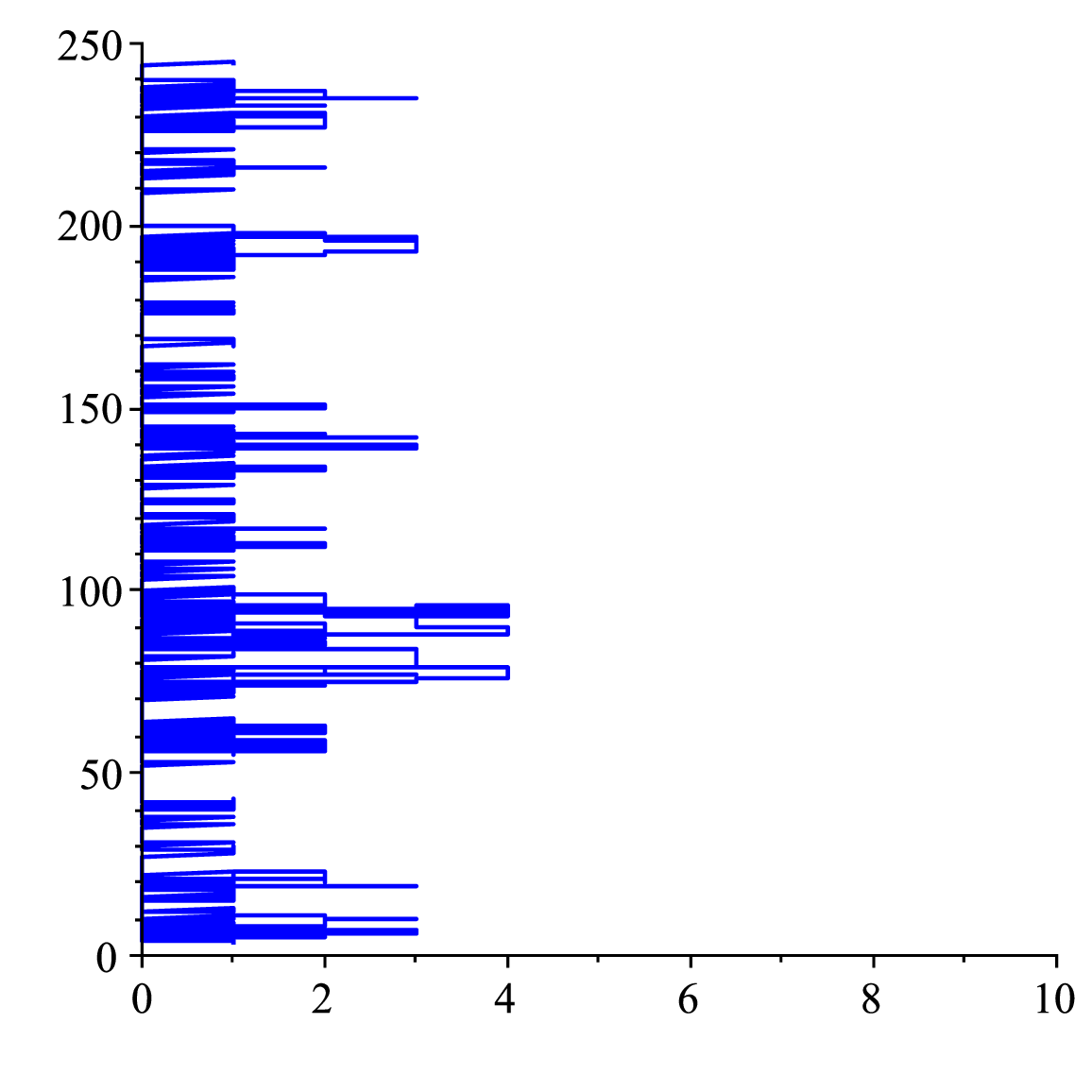}\qquad
\includegraphics[width=0.4\textwidth]{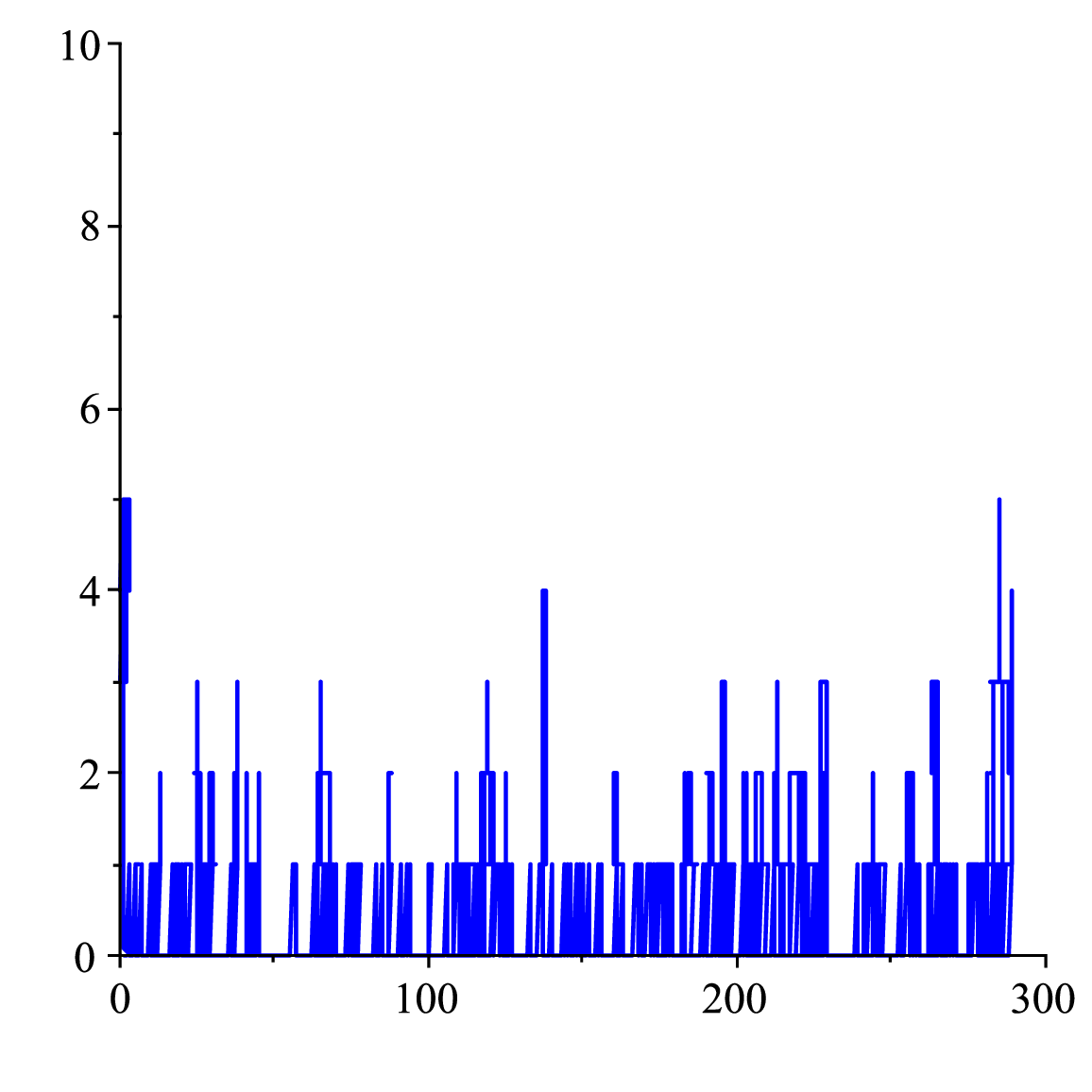}
\caption{Two typical paths, with escape along the vertical and horizontal axes, respectively.}
\label{fig:paths}
\end{figure}

We are now ready to state our results. Recall that the Green function at $(i,j)$ starting from $(i_0,j_0)$ is
\begin{equation}
\label{eq:def_Green_function}
     g\bigl((i_0,j_0) \to (i,j)\bigr)=\sum_{n=0}^\infty \P_{(i_0,j_0)} \bigl(Z(n) = (i,j)\bigr)=\sum_{n=0}^\infty  p^{(n)}\bigl((i_0,j_0) \to (i,j)\bigr).
\end{equation}
Moreover, ${\cal N}_1$ and ${\cal N}_2$ will denote the numbers of visits of $\{Z(n)\}$ to the 
boundary strips
$\N\times\{0,\ldots,k_0-1\}$ and $\{0,\ldots,k_0-1\}\times\N$:
\begin{equation}
\label{eq:def_N1_N2}
{\cal N}_1 = \sum_{n=0}^\infty \1_{\N\times\{0,\ldots,k_0-1\}}(Z(n))
\quad \text{and} \quad
 {\cal N}_1 = \sum_{n=0}^\infty \1_{\{0,\ldots,k_0-1\}\times\N}(Z(n)).
\end{equation}
The numbers ${\cal N}_1$ and ${\cal N}_2$ may be infinite. Let finally $\pi_1,\pi_2$ be the invariant distributions of $\{X_1(n)\}$ and $\{Y_2(n)\}$, see Lemma \ref {positive_recurrence_lemma} and below.

Our first preliminary result is the following: 

\begin{thm}[Boundary asymptotics of the Green function]
\label{thm:main-1}
Under Assumptions \ref{as1}--\ref{as6},
for any  $i\in\N$ and $(i_0,j_0)\in\N^2$, 
\begin{equation}
\label{eq:thm:main-1a}
   \lim_{j\to\infty} g\bigl((i_0,j_0)\to(i,j)\bigr)  =   \P_{(i_0,j_0)} \bigl({\cal N}_1 < \infty\bigr)\frac{\pi_1(i)}{V_1}
\end{equation}
and
\begin{equation}
\label{eq:thm:main-1b}
   \lim_{j\to\infty} g\bigl((i_0,j_0)\to(j,i)\bigr) =  \P_{(i_0,j_0)} \bigl({\cal N}_2 < \infty\bigr)\frac{\pi_2(i)}{V_2}.
\end{equation}
\end{thm}

Theorem~\ref{thm:main-1} shows that when moving to infinity in horizontal (or vertical) direction, the Green function decomposes asymptotically as a product of two terms, see \eqref{eq:thm:main-1a}:
the first one, $\P_{(i_0,j_0)}\bigl({\cal N}_1 < \infty\bigr)$, is the probability of escape along the horizontal axis (it is a harmonic function); the second part, namely, $\pi_1(i)/V_1$, is the asymptotics of the Green function of the half-space random walk $Z_1$. See \eqref{MAP_eq4} applied for the Markov-additive process $Z_1$.

To get the asymptotics of the Green functions as $i+j$ goes to infinity along an angular direction $j/i\to \tan\gamma$, $\gamma\in[0,\frac{\pi}{2}]$, which we will derive in Theorem~\ref{thm:main-3}, we need more restrictive assumptions on the steps of the process $Z$.

For $(i, j) \in  \mathbb N^2$ and $(u,v)\in \mathbb R^2$, let 
\begin{equation*}
   \Phi_{i,j} (u,v)= \E_{i, j} \left( \exp \bigl\langle (u,v), (Z(1)-(i,j) ) \bigr\rangle \right).
\end{equation*}
There are only $(k_0+1)^2$ different functions, those with $i,j\in\{0,\ldots, k_0\}$. The following assumption is clearly stronger than Assumption~\ref{as2}.
 
\medskip

\setcounter{Assumption}{2}
\begin{Assumptionbis}
\label{as2p}
For any $i,j\in\{0,\ldots, k_0\}$, $\Phi_{i,j}(u,v) <\infty$ in a neighborhood of $(0,0)$.
  \end{Assumptionbis}
  
This assumption is needed to prove an exponential rate of convergence in \eqref{eq:thm:main-1a} and \eqref{eq:thm:main-1b}.
  
  \begin{thm}[Theorem~\ref{thm:main-1} with exponential convergence]
\label{thm:main-2}
Under Assumptions \ref{as1}--\ref{as6} and \ref{as2p}, 
for any $i\in\N$, there exist positive constants $C_i$ and $\delta_i$ such that for any $(i_0,j_0)\in\N^2$ and $j\in\mathbb N$,
\begin{equation}
\label{eq:thm:main-1}
   \left\vert g\bigl((i_0,j_0)\to(i,j)\bigr) -    \P_{(i_0,j_0)} \bigl({\cal N}_1 < \infty\bigr)\frac{\pi_1(i)}{V_1}\right\vert \leq C_i \exp( - \delta_i  (j-j_0)).
\end{equation}
Analogously,
for any $j\in\N$, there exist positive constants $C_j'$ and $\delta_j'$ such that for any $(i_0,j_0)\in\N^2$ and $i\in\mathbb N$,
\begin{equation*}
   \left\vert g\bigl((i_0,j_0)\to(i,j)\bigr) -    \P_{(i_0,j_0)} \bigl({\cal N}_2 < \infty\bigr)\frac{\pi_2(j)}{V_2}\right\vert \leq C_j' \exp( - \delta_j'  (i-i_0)).
\end{equation*}
\end{thm}

Theorem~\ref{thm:main-2} is the first step in order to get  the asymptotics of the Green functions along angular directions. Due to the exponential convergence in \eqref{eq:thm:main-1a} and \eqref{eq:thm:main-1b}, the generating functions of the Green functions along horizontal and vertical axes admit a pole at $x=1$ and $y=1$. This result will be essential for further analytic analysis.

In order to state our main results, we need to state two new assumptions, which are stronger than Assumption~\ref{as2p}. We introduce
\begin{equation*}
   {\cal Q}=\{(u,v) \in {\mathbb R}^2 : \Phi_{k_0,k_0} (u,v) \leq 1\}.
\end{equation*}
\setcounter{Assumption}{2}
\begin{Assumptionbisbis}
\label{as2pp}
Any point of $\partial {\mathcal Q}$ has a neighborhood in ${\mathbb R}^2$ in which $\Phi_{k_0,k_0}$ is finite.
\end{Assumptionbisbis} 

This assumption corresponds to Ney and Spitzer's assumption (1.4) in the paper \cite{NeSp-66}, where the Martin boundary for homogeneous random walks is constructed via a Cramer's change of measure. Under it, the set
\begin{equation*}
   \partial {\cal Q} = \{(u,v) \in {\mathbb R}^2 : \Phi_{k_0,k_0}(u,v) =1\}
\end{equation*}
 is a closed curve homeomorphic to a circle, and the equation  
 $\Phi_{k_0,k_0}(u,0)=1$ (resp.\ $\Phi_{k_0,k_0}(0,v)=1$)  on ${\mathbb R}$  has two roots 
 $u=0$ and $u=u_1>0$ (resp.\ $v=0$ and $v=v_1>0$).
 
We also assume that:
\setcounter{Assumption}{2}
\begin{Assumptionbisbisbis}
\label{as2ppp}
The functions $\Phi_{i,k_0}$ (resp.\ $\Phi_{k_0,j}$) are finite in a neighborhood of $(u_1,0)$ (resp.\ $(0, v_1)$), for any $i\in\{0,\ldots, k_0-1\}$ (resp.\ $j\in\{0,\ldots, k_0-1\}$). The functions $\Phi_{i,j}$ are finite in both of these neighborhoods, for any $i,j\in\{0,\ldots, k_0-1\}$.
\end{Assumptionbisbisbis} 

Under Assumptions \ref{as2pp} and \ref{as2ppp}, the points
 \begin{equation}
     \label{u1v1}
 x_1=\exp(u_1) \quad \text{and}\quad  y_1=\exp(v_1)
 \end{equation}
 do exist and determine the asymptotics of the invariant measures $\pi_1(i)$
 and $\pi_2(j)$ above. Namely, there exist constants $A_1>0$ and $A_2>0$ such that as $i,j\to\infty$,
 \begin{equation}
 \label{asinvm}
    \pi_1(i) = \frac{A_1}{x_1^i} +o(x_1^{-i}) \quad\text{and}\quad \pi_2(j)
      =\frac{A_2}{y_1^j} + o(y_1^{-j}).
 \end{equation}
This classical result is a part of the elementary Lemma \ref{lem:inv_mes}.
We are now ready to formulate our main result.
\begin{thm}[Interior asymptotics of the Green function]
\label{thm:main-3}
Under Assumptions \ref{as1}--\ref{as6}, \ref{as2pp} and \ref{as2ppp},
as $i+j$ goes to infinity along an angular direction ($j/i$ going to a constant in $[0,\infty]$), one has
\begin{equation}
\label{eq:mainres}
     g\bigl((i_0,j_0) \to (i,j)\bigr) = \mathbb P_{(i_0,j_0)}\bigl(\mathcal N_1<\infty\bigr)\frac{\pi_1(i)}{V_1}  + \mathbb P_{(i_0,j_0)}\bigl(\mathcal N_2<\infty\bigr)\frac{\pi_2(j)}{V_2} + o(x_1^{-i} + y_1^{-j}).
\end{equation} 
\end{thm}

 Let now
\begin{equation}
\label{eq:def_gamma_0}
     t_0=\tan\gamma_0=\frac{u_1}{v_1}=\frac{\log x_1}{\log y_1}.
\end{equation}
\begin{thm}[Martin boundary]
\label{thm:main-4}
Under Assumptions \ref{as1}--\ref{as6}, \ref{as2pp} and \ref{as2ppp},
the minimal Martin boundary is a union of two points. Moreover, if $t_0\in\mathbb Q$ (resp.\ $t_0\notin \mathbb Q$), then the full Martin boundary is homeomorphic to $\mathbb Z\cup\{\pm\infty\}$ (resp.\ $\mathbb R\cup\{\pm\infty\}$).
\end{thm}
The angular direction $\gamma_0$ subdivides the space into a lower and an upper part, see Figure~\ref{fig:subdivision}, where the limits are the two respective minimal harmonic functions, plus the critical direction $\gamma_0$, along which one gets either countably or continuum many Martin limits, according to the nature of $\tan \gamma_0$. When parametrized by $\mathbb Z$ or $\mathbb R$, those Martin limits will accumulate at one of the two minimal harmonic functions when the parameters tend to $\pm\infty$. This captures the compactness of the boundary, which then can be thought of as $\mathbb Z\cup\{\pm\infty\}$ or $\mathbb R\cup\{\pm\infty\}$.
\begin{figure}
\centering
\begin{tikzpicture}
\begin{scope}[scale=0.5]
\draw[nred,fill=nred!30,domain=0:31.5] plot ({6.5*cos(\x)}, {6.5*sin(\x)});
\draw[dblue,fill=dblue!30,domain=33.5:90] plot ({6.5*cos(\x)}, {6.5*sin(\x)});
\draw[nred!30,fill=nred!30] (0.1,0) -- (6.5,0) -- (5.54,3.38);
\draw[dblue!30,fill=dblue!30] (0,0.1) -- (0,6.5) -- (5.5,3.55);
\draw[white, thick] (-1.5,-1.5) grid (6.5,6.5);
\draw[nred,thick,domain=0:31.5] plot ({6.5*cos(\x)}, {6.5*sin(\x)});
\draw[dblue,thick,domain=33.5:90] plot ({6.5*cos(\x)}, {6.5*sin(\x)});
\draw[green,thick,->] (0,0) -- (5.5,3.47);
\draw[->] (0,-1.5) -- (0,6.5);
\draw[->] (-1.5,0) -- (6.5,0);
\draw[nred,thick,->] (0,0) -- (6.2,2);
\draw[dblue,thick,->] (0,0) -- (0.4,6.5);
\draw[thick,domain=0:32] plot ({cos(\x)}, {sin(\x)});
\node[above right] at (0.8,-0.1) {\tiny{$\gamma_0$}};
\end{scope}
\end{tikzpicture}
\caption{The angular direction $\gamma_0$ subdivides the quarter plane into two wedges, each of which corresponding to one minimal harmonic function. Martin limits accumulate along the critical direction.}
\label{fig:subdivision}
\end{figure}

Remark that the angle $\gamma_0$ does not coincide in general with the angle made by the drift vector, see Figure~\ref{fig:examples_gamma0}. We further notice that it is easy to construct elementary examples with rational (resp.\ non-rational) $t_0$, see again Figure~\ref{fig:examples_gamma0}.



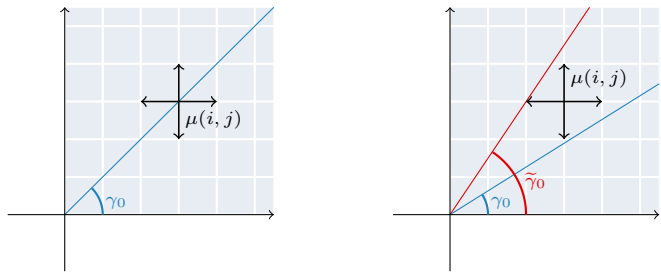
\begin{figure}
\begin{tikzpicture}
\begin{scope}[scale=0.5]
\draw[white, fill=dblue!10] (0,0) rectangle (5.5,5.5);
\draw[white, thick] (-1.5,-1.5) grid (5.5,5.5);
\draw[nblue] (0,0) -- (5.5,5.5);
\draw[->] (0,-1.5) -- (0,5.5);
\draw[->] (-1.5,0) -- (5.5,0);
\draw[->, black, semithick] (3,3) -- (4,3);
\draw[->, black, semithick] (3,3) -- (2,3);
\draw[->, black, semithick] (3,3) -- (3,4);
\draw[->, black, semithick] (3,3) -- (3,2);
\node[above right, black] at (2.9,2.0) {\tiny{$\mu(i,j)$}};
\draw [nblue,thick,domain=0:45] plot ({cos(\x)}, {sin(\x)});
\node[above right, nblue] at (0.8,-0.1) {\tiny{$\gamma_0$}};
\end{scope}
\end{tikzpicture}\qquad\qquad
\begin{tikzpicture}
\begin{scope}[scale=0.5]
\draw[white, fill=dblue!10] (0,0) rectangle (5.5,5.5);
\draw[white, thick] (-1.5,-1.5) grid (5.5,5.5);
\draw[nred] (0,0) -- (3.67,5.5);
\draw[nblue] (0,0) -- (5.5,3.47);
\draw[->] (0,-1.5) -- (0,5.5);
\draw[->] (-1.5,0) -- (5.5,0);
\draw[->, black, semithick] (3,3) -- (4,3);
\draw[->, black, semithick] (3,3) -- (2,3);
\draw[->, black, semithick] (3,3) -- (3,4);
\draw[->, black, semithick] (3,3) -- (3,2);
\node[above right, black] at (2.9,3.1) {\tiny{$\mu(i,j)$}};
\draw [nblue,thick,domain=0:32] plot ({cos(\x)}, {sin(\x)});
\node[above right, nblue] at (0.8,-0.1) {\tiny{$\gamma_0$}};
\node[above right, nred] at (1.7,0.4) {\tiny{$\widetilde\gamma_0$}};
\draw [nred,thick,domain=0:56] plot ({2*cos(\x)}, {2*sin(\x)});
\end{scope}
\end{tikzpicture}
\caption{Left: for a symmetric model (meaning $\mu(i,j)=\mu(j,i)$), the angle $\gamma_0=\frac{\pi}{4}$ coincides with the angle $\widetilde\gamma_0$ made by the drift vector; in this case $t_0=1$. Right: for the example $\mu(1,0)=\frac{1}{6}$, $\mu(0,-1)=\frac{3}{8}$, $\mu(-1,0)=\frac{1}{3}$, $\mu(0,1)=\frac{1}{8}$ with (non-symmetric) transition probabilities, these two angles are different ($\gamma_0=\arctan\frac{\log 2}{\log 3}$ and $\widetilde\gamma_0=\arctan\frac{3}{2}$).}
\label{fig:examples_gamma0}
\end{figure}

\subsection{Proof of Theorem~\ref{thm:main-4}}
\begin{proof}
Plugging estimates (\ref{asinvm})  into \eqref{eq:mainres} and using the definition \eqref{eq:def_gamma_0} of $t_0$, one readily obtains
\begin{equation}
\label{eq:asymptotics_Martin_kernel}
     \frac{g\bigl((i_0,j_0) \to (i,j)\bigr)}{g\bigl((0,0) \to (i,j)\bigr)}\sim \frac{\bigl(\mathbb P_{(i_0,j_0)}\bigl(\mathcal N_1<\infty\bigr)A_1/V_1\bigr)y_{1}^{j-it_0}  + \mathbb P_{(i_0,j_0)}\bigl(\mathcal N_2<\infty\bigr)A_2/V_2}{\bigl(\mathbb P_{(0,0)}\bigl(\mathcal N_1<\infty\bigr)A_1/V_1\bigr)y_{1}^{j-it_0}  + \mathbb P_{(0,0)}\bigl(\mathcal N_2<\infty\bigr)A_2/V_2}.
\end{equation}
Let us examine three different regimes when $i+j$ goes to infinity along an angular direction $j/i\to \tan\gamma$; see Figures~\ref{fig:subdivision} and \ref{fig:examples_gamma0}. First, if $\gamma>\gamma_0$, then $j-it_0\to\infty$ and with \eqref{eq:asymptotics_Martin_kernel}, the limit Martin kernel equals
\begin{equation*}
     k\bigl(i_0,j_0\bigr)=\lim\limits_{\substack{i+j\to\infty\\j/i\to\tan\gamma}}\frac{g\bigl((i_0,j_0) \to (i,j)\bigr)}{g\bigl((0,0) \to (i,j)\bigr)}= \frac{\mathbb P_{(i_0,j_0)}\bigl(\mathcal N_1<\infty\bigr)}{\mathbb P_{(0,0)}\bigl(\mathcal N_1<\infty\bigr)}.
\end{equation*}
For 
the same reasons, if now $\gamma<\gamma_0$, then $j-it_0\to-\infty$ and
\begin{equation*}
     k\bigl(i_0,j_0\bigr)=\frac{\mathbb P_{(i_0,j_0)}\bigl(\mathcal N_2<\infty\bigr)}{\mathbb P_{(0,0)}\bigl(\mathcal N_2<\infty\bigr)}.
\end{equation*}
The limit case $\gamma=\gamma_0$ is the most interesting. The set of points $j-it_0$ in \eqref{eq:asymptotics_Martin_kernel} lies on an additive subgroup of $\mathbb R$, namely, $\mathbb Z+t_0\mathbb Z$. If $t_0=\frac{n_0}{m_0}$ is rational, then the 
limits for the Martin kernel are ($n\in\mathbb Z \cup\{\pm\infty\}$)
\begin{equation*}
      \frac{\bigl(\mathbb P_{(i_0,j_0)}\bigl(\mathcal N_1<\infty\bigr)A_1/V_1\bigr)y_{1}^{n/m_0}  + \mathbb P_{(i_0,j_0)}\bigl(\mathcal N_2<\infty\bigr)A_2/V_2}{\bigl(\mathbb P_{(0,0)}\bigl(\mathcal N_1<\infty\bigr)A_1/V_1\bigr)y_{1}^{n/m_0}  + \mathbb P_{(0,0)}\bigl(\mathcal N_2<\infty\bigl)A_2/V_2}.
\end{equation*}
In particular, the full Martin boundary is 
countable. If $t_0\notin \mathbb Q$, then the subgroup $\mathbb Z+t_0\mathbb Z$ is dense in $\mathbb R$, and any combination
\begin{equation*}
      \frac{\bigl(\mathbb P_{(i_0,j_0)}\bigl(\mathcal N_1<\infty\bigr)A_1/V_1\bigr)u  + \mathbb P_{(i_0,j_0)}\bigl(\mathcal N_2<\infty\bigr)A_2/V_2}{\bigl(\mathbb P_{(0,0)}\bigl(\mathcal N_1<\infty\bigr)A_1/V_1\bigr)u  + \mathbb P_{(0,0)}\bigl(\mathcal N_2<\infty\bigr)A_2/V_2}
\end{equation*}
appears
in the limit, for any $u\in[0,\infty]$. The statement on the minimal Martin boundary is clear.
\end{proof}

\subsection{Structure of the paper and main ideas of the proofs}
The remaining part of the paper is organized as follows. Preliminary results needed for the proof of Theorems~\ref{thm:main-1}  and~\ref{thm:main-2} are obtained in Sections~\ref{sec:rough_estimates}--\ref{sec:4}. The proof of Theorem~\ref{thm:main-1} is given in Section~\ref{sec:5p}, and  Theorem~\ref{thm:main-2} is proved in Section~\ref{sec:5}. In Sections~\ref{sec:7}--\ref{sec:proofs_2}, we prove Theorem~\ref{thm:main-3}.

In order to prove Theorems~\ref{thm:main-1} and~\ref{thm:main-2}, we consider the stopping times 
\begin{equation*}
     {T}_{k} =\inf\{n > 0: X(n) \leq (k_0-1)\vee k\}, \quad k\geq k_0,
\end{equation*}
and 
\begin{equation*}
\tau = \inf\{n>0: Y(n) < k_0\}
\end{equation*}
(see also Section~\ref{sec:hitting_times}) and we use the following discrete equation:
\begin{multline}\label{sec:structure_eq1}
     {g}\bigl((i,j)\to (k,\ell)\bigr) =\E_{(i,j)}\left( \sum_{n=0}^{\tau-1} \1_{\{X(n) = k, Y(n) = \ell\}}\right)   + \sum_{\substack{0\leq j'< k_0\\ i'\geq 0}} {g}\bigl((i,j)\to(i',j')\bigr) \times \\ \times \sum_{\substack{0\leq i''\leq( k_0-1)\vee k\\ j'' \geq k_0}} \P_{(i',j')}\bigl(Z({T}_{k}) = (i'',j''), {T}_{k}< \tau\bigr) \E_{(i'',j'')}\left( \sum_{n=0}^{\tau-1} \1_{\{X(n) = k, Y(n) = \ell\}}\right),
\end{multline} 
see \eqref{sec4_eq1} in Lemma~\ref{sec4_lem3}. For any $i\geq 0$ and $j\geq k_0$, before time $\tau$, the transition probabilities of the process $Z=(X,Y)$ starting at $(i,j)$ are identical to those of the local Markov additive process $Z_1=(X_1,Y_1)$. Accordingly, the quantities
\begin{equation*}
   \E_{(i,j)}\left( \sum_{n=0}^{\tau-1} \1_{\{X(n) = k, Y(n) = \ell\}}\right)
\end{equation*}
may be investigated through the properties of the Markov additive processes.
We show that for any $i,j,k\in\N$
\be\label{sec:structure_eq2} 
\lim_{\ell\to\infty} \E_{(i,j)}\left( \sum_{n=0}^{\tau-1} \1_{\{X(n) = k, Y(n) = \ell\}}\right) = \frac{\pi_1(k) \P_{i,j)}\bigl(\tau = \infty\bigr)}{V_1}, 
\ee
and we apply then the dominated convergence theorem to get \eqref{eq:thm:main-1a} in Theorem~\ref{thm:main-1}.

To prove that the rate of convergence in \eqref{eq:thm:main-1a} is exponential, i.e., to get \eqref{eq:thm:main-1} in Theorem~\ref{thm:main-2}, we obtain exponential estimates for the quantities $\P_{(i',j')}\bigl(Z({T}_{k}) = (i'',j''), {T}_{k}< \tau\bigr)$ in \eqref{sec:structure_eq1}, see Lemma~\ref{sec4_lem2}, and we prove an exponential rate of convergence in \eqref{sec:structure_eq1}, see Propositions~\ref{prop3_1} and \ref{prop3_1p}.

Theorem~\ref{thm:main-3} is proved in Section~\ref{sec:proofs_2}
by using analytic methods and the result of Theorem~\ref{thm:main-2}. In Section~\ref{sec:7}, we introduce the generating functions of the Green functions and prove that they satisfy the functional equation \eqref{eq:main_func_eq}, which is valid in the region where $\vert x \vert <1$ and $\vert y \vert <1$. The Green functions  
$g\big((i_0, j_0) \to (i,j) \big)$  can be deduced from it by the two-dimensional integral Cauchy formula \eqref{eq:Green_function_double_integral},
and their asymptotics may be found by shifting the integration contours and taking into account singularities of the integrand. For this purpose, the
 generating functions of the Green functions along horizontal and vertical directions $g_\ell(x), \widetilde g_k(y)$ 
   need to be continued outside their unit disc. Due to Theorem~\ref{thm:main-2}, these functions can be continued  meromorphically   
and admit a unique pole at $x=1$ and $y=1$  in 
  $\vert x \vert <1+\epsilon$ and $\vert y \vert <1+\epsilon$ respectively, which is essential for our further analysis.  
  Applying asymptotic techniques to the integral \eqref{eq:Green_function_double_integral} will finally lead to the conclusion of  Theorem~\ref{thm:main-3}.

\section{Rough uniform estimates of the Green functions}
\label{sec:rough_estimates}

In this section, we prove the uniform boundedness of the Green function \eqref{eq:def_Green_function} along the axes. This result proves in particular that under Assumptions~\ref{as1}--\ref{as6} the Markov chain $Z$ is transient and is needed for the proof of  Theorems~\ref{thm:main-2} and \ref{thm:main-3}. Let $m_1,m_2$ be defined in \eqref{eq:drift} and $V_1,V_2$ in \eqref{definition_vect_2}.

\begin{prop} 
\label{sec4_lem1} 
Under Assumptions~\ref{as1}--\ref{as4} 
and if in addition $m_1 < 0$ and $V_1 > 0$, then the Markov chain $\{Z(n)\}$ is transient and for any $k\in\N$, 
\begin{equation*}
     \sup_{i,j,\ell\in\N} g\bigl((i,j)\to(k,\ell)\bigr)  < \infty. 
\end{equation*}
In the same way, under Assumptions~\ref{as1}--\ref{as4}, and 
if $m_2<0$ and $V_2 > 0$, the Markov chain $\{Z(n)\}$ is transient and for any $\ell\in\N$, 
\begin{equation*}
     \sup_{i,j,k\in\N} g\bigl((i,j)\to(k,\ell)\bigr)  < \infty. 
\end{equation*}
\end{prop}

\begin{proof}
We give a proof of the first assertion of this lemma;
the second assertion follows by symmetry, exchanging the two coordinates.
It is well known that for any $(i,j), (k,\ell)\in\N^2$, 
\begin{multline*}
   g\bigl((i,j)\to(k,\ell)\bigr)  = \P_{(i,j)}\bigl(Z(n) = (k,\ell)  \text{ for some }  n\geq 0\bigr)  g\bigl((k,\ell)\to(k,\ell)\bigr) \\ \leq g\bigl((k,\ell)\to(k,\ell)\bigr), 
\end{multline*} 
and 
\begin{equation*}
g\bigl((k,\ell)\to(k,\ell)\bigr) = \frac{1}{1-\P_{(k,\ell)}\bigl(Z(n) = (k,\ell) \text{ for some } n\geq 1\bigr) }.
\end{equation*} 
It is therefore sufficient to show that for any $k\in\N$,
\be\label{sec4_lem1_eq1} 
\sup_{\ell\in\N} \P_{(k,\ell)}\bigl(Z(n) = (k,\ell)  \text{ for some } n\geq 1\bigr) <1.
\ee 
To prove the inequality \eqref{sec4_lem1_eq1}, we consider the stopping times 
\begin{align}
   \tau(k,\ell) &= \inf\{n>0: Z(n) = (k,\ell)\},\label{eq:def_tau(k,l)}\\
   \tau &= \inf\{n>0: Y(n) < k_0\},\label{eq:def_tau} \\
   \tau^{\textnormal{loc}}(k,\ell) &= \inf\{n>0: Z_1(n) = (k,\ell)\},\label{eq:def_tau_1(k,l)}\\
   \tau_1^{\textnormal{loc}} &= \inf\{n>0: Y_1(n) < k_0\},\label{eq:def_tau_1}
\end{align}
see also Section~\ref{sec:hitting_times}. Then, for any $(k,\ell)\in\N^2$, 
\begin{align*} 
 \P_{(k,\ell)}\bigl(Z(n) = (k,\ell)  & \text{ for some } n\geq 1\bigr) = \P_{(k,\ell)}\bigl( \tau(k,\ell) <  \infty\bigr) \\
 &= \P_{(k,\ell)}\bigl( \tau \leq \tau(k,\ell) <  \infty\bigr) + \P_{(k,\ell)}\bigl( \tau(k,\ell) <  \infty  \text{ and }  \tau(k,\ell) < \tau\bigr) \\
 &\leq \P_{(k,\ell)}\bigl( \tau <  \infty\bigr) + \P_{(k,\ell)}\bigl( \tau(k,\ell) <  \infty  \text{ and } \tau(k,\ell) < \tau\bigr),
\end{align*} 
where, according to the definition of the local process $Z_1$, 
\begin{equation*}
     \P_{(k,\ell)}\bigl( \tau <  \infty\bigr)  = \P_{(k,\ell)}\bigl( \tau_1^{\textnormal{loc}} <  \infty\bigr) 
\end{equation*} 
and 
\begin{align*}
\P_{(k,\ell)}\bigl( \tau(k,\ell) < \infty   \text{ and }  \tau(k,\ell) < \tau\bigr)  &= \P_{(k,\ell)}\bigl( \tau^{\textnormal{loc}}(k,\ell) <  \infty \text{ and } \tau^{\textnormal{loc}}(k,\ell) < \tau_1^{\textnormal{loc}}\bigr) \\
&\leq \P_{(k,\ell)}\bigl( \tau^{\textnormal{loc}}(k,\ell) <  \infty\bigr) =   \P_{(k,0)}\bigl( \tau^{\textnormal{loc}}(k,0) <  \infty\bigr). 
\end{align*} 
Hence, for any $(k,\ell)\in\N^2$,  
\be\label{sec4_lem1_eq2}
 \P_{(k,\ell)}\bigl(Z(n) = (k,\ell)   \text{ for some } n\geq 1\bigr) \leq  \P_{(k,\ell)}\bigl( \tau_1^{\textnormal{loc}} <  \infty\bigr)   + \P_{(k,0)}\bigl( \tau^{\textnormal{loc}}(k,0) <  \infty\bigr). 
\ee
Recall now that by \eqref{fluid_limit_eq}, for any $(k,\ell)\in\N^2$, $\P_{(k,\ell)}$-a.s.,  $\lim_{n\to\infty} Y_1(n) = \infty$. For any $k\in\N$, the quantity
$\inf_{n\in\mathbb N} Y_1(n)$ is therefore $\P_{(k,0)}$-a.s.\ finite, and 
\be\label{sec4_lem1_eq3} 
\P_{(k,0)} \bigl( \tau^{\textnormal{loc}}(k,0) <  \infty\bigr) < 1. 
\ee
Indeed, the Markov chain  $Z_1$ is transient, so its first return time to $(k,0)$ is infinite with positive probability. 
It follows from the ($\P_{(k,0)}$-a.s.)\ finiteness
of $\inf_{n\in\mathbb N} Y_1(n)$ that
\begin{equation*}
   \lim_{\ell\to \infty}   \P_{(k,\ell)}\bigl( \tau_1^{\textnormal{loc}} <  \infty\bigr)  =     \lim_{\ell\to\infty} \P_{(k,0)}\bigl( Y_1(n) < k_0 - \ell \text{ for some } n\geq1\bigr)  =0 ,
\end{equation*}
and using \eqref{sec4_lem1_eq2}, we conclude that for any $k\in\N$, 
\begin{equation*}
\limsup_{\ell\to\infty}  \P_{(k,\ell)}\bigl(Z(n) = (k,\ell) \text{ for some } n\geq 1\bigr)  \leq \P_{(k,0)}\bigl( \tau^{\textnormal{loc}}(k,0) <  \infty\bigr). 
\end{equation*} 
When combined with \eqref{sec4_lem1_eq3}, the last relation proves that for any $k\in\N$ and $N\in\N$ large enough,
\be\label{sec4_lem1_eq4} 
\sup_{\ell\geq N} \P_{(k,\ell)}\bigl(Z(n) = (k,\ell)  \text{ for some } n\geq 1\bigr) <1. 
\ee
The Markov chain $\{Z(n)\}$ being irreducible (Assumption \ref{as3}), the last relation proves  that   $\{Z(n)\}$  is transient, and consequently for all $\ell\in\N$
\be\label{sec4_lem1_eq5} 
\P_{(k,\ell)}\bigl(Z(n) = (k,\ell)  \text{ for some } n\geq 1\bigr) <1. 
\ee
When combined,
\eqref{sec4_lem1_eq4} and \eqref{sec4_lem1_eq5} imply \eqref{sec4_lem1_eq1}, and thus Proposition~\ref{sec4_lem1}. 
\end{proof}

\section{Useful results for Markov-additive processes and their consequences for the Markov chain $Z$} 

In this section, we obtain the results needed for the proof of Theorem~\ref{thm:main-2}.  

\subsection{General statements for Markov-additive processes} Let ${\cal E}$ be a countable set. Recall that a Markov chain $\{({\cal M}(n), {\cal A}(n))\}_{n\geq0}$ on the (countable) set ${\cal E}\times \Z$  is called Markov-additive if for any $(i,j),(i',j')\in {\cal E}\times \Z$, 
\begin{equation*}
 \P_{(i,j)}\bigl(({\cal M}(1),{\cal A}(1)) = (i',j')\bigr) = \P_{(i,0)}\bigl(({\cal M}(1),{\cal A}(1)) = (i',j'-j)\bigr).
\end{equation*} 
The first component $\{{\cal M}(n)\}$ is the Markovian part of the Markov-additive chain $\{({\cal M}(n), {\cal A}(n))\}$, and $\{{\cal A}(n)\}$ is its additive part.

We first recall some properties of the Markov-additive processes. 
The Markovian part $\{{\cal M}(n)\}$ is a Markov chain on ${\cal E}$ with transition probabilities
\begin{equation*}
   \P\bigl({\cal M}(n+1) = i'  \vert  {\cal M}(n) = i\bigr) = \P_{(i,0)}\bigl({\cal M}(1)=i'\bigr), \quad \forall i,i'\in{\cal E}. 
\end{equation*}
If the Markovian part $\{{\cal M}(n)\}$ is irreducible and recurrent, then letting for $i\in{\cal E}$
\begin{equation*}
   ({\cal M}(0),{\cal A}(0)) = (i,0)
\end{equation*} 
and then
\begin{equation}
\label{eq:def_t_n}
   \left\{\begin{array}{rcl}
   t(i)=t_1(i) &=& \inf\{p > 0 : {\cal M}(p) = i\},\medskip\\
   t_{n+1}(i) &=& \inf\{p > t_n(i) : {\cal M}(p) = i\}, \quad \forall n\in\N^*, 
   \end{array}\right.
\end{equation}
one gets a homogeneous random walk $\{S_i(n)\} = \{{\cal A}(t_n(i))\}$ on $\Z$. If moreover the Markov chain  $\{{\cal M}(n)\}$ is positive recurrent with stationary distribution $\{\pi(i)\}_{i\in{\cal E}}$, and if the series 
\be\label{MAP_eq2}
V = \sum_{i'\in{\cal E}} \E_{(i',0)}\bigl({\cal A}(1)\bigr) \pi(i') 
\ee
absolutely converges, then the steps of the random walk $\{S_i(n)\}$ are integrable and  the following generalization of Wald's identity holds: 
\be\label{MAP_eq3}
\E_{(i,0)}\bigl({\cal A}(t_1(i))\bigr) = \frac{V}{\pi(i)},
\ee
see Theorem~11  of Prabhu, Tang and Zhu~\cite{Prabhu-Tang-Zhu}. Moreover, in this case, by \cite[Thm.~12]{Prabhu-Tang-Zhu}, for any $(i,j)\in{\cal E} \times\Z$, $\P_{(i,j)}$-a.s.,
\begin{equation*}
   \frac{{\cal A}(n)}{n} \to  V \quad \text{as} \quad  n\to\infty, 
\end{equation*}
and if $V > 0$, then by Theorem 2.1 of Alsmeyer~\cite{Al-94}, for any $k\in{\cal E}$, 
\be\label{MAP_eq4}
\lim_{\ell\to \infty} \E_{(i,0)}\left( \sum_{n=0}^\infty \1_{\{{\cal M}(n) = k, {\cal A}(n) = \ell\}}\right) = \frac{\pi(k)}{V}. 
\ee
Consider moreover some $k_0\in\Z$ and let 
\begin{equation}
\label{eq:def_cal_T}
   {\mathcal  T} = \inf\{n \geq 0: {\cal A}(n) < k_0\}.  
\end{equation}

Using the Markov property, \eqref{MAP_eq4} and the dominated convergence theorem, one easily obtains:
\begin{prop}
\label{MAP_prop1a}
If the Markov chain $\{{\cal M}(n)\}$ is positive recurrent with stationary distribution $\{\pi(i)\}_{i\in{\cal E}}$, if the series \eqref{MAP_eq2} absolutely converges and if $V > 0$, then for any $(i,j)\in {\cal E}\times\Z$ and $k\in{\cal E}$, 
\be\label{MAP_eq5a}
\lim_{\ell\to \infty} \E_{(i,j)}\left( \sum_{n=0}^{{\mathcal  T} - 1} \1_{\{{\cal M}(n) = k, {\cal A}(n) = \ell\}}\right) = \frac{\pi(k)\P_{(i,j)}({\mathcal  T} =  \infty)}{V}. 
\ee
\end{prop} 
\begin{proof}
Indeed, using the Markov property applied with the stopping time ${\mathcal T}$, for any $(i,j), (k,\ell)\in {\cal E}\times\Z$, one gets 
\begin{multline}\label{MAP_eq5b}
\E_{(i,0)}\left( \sum_{n=0}^\infty \1_{\{{\cal M}(n) = k, {\cal A}(n) = \ell\}}\right)  =  \E_{(i,j)}\left( \sum_{n=0}^{{\mathcal  T} - 1} \1_{\{{\cal M}(n) = k, {\cal A}(n) = \ell\}}\right)  \\ + 
\sum_{\substack{(i',j')\in{\cal E}\times\Z\\j < k_0}} \P_{(i,j)}\bigl({\cal M}({\mathcal T})= i', {\cal A}({\mathcal T}) = j'\bigr) \,\E_{(i',j')}\!\!\left( \sum_{n=0}^\infty \1_{\{{\cal M}(n) = k, {\cal A}(n) = \ell\}}\right).
\end{multline} 
Since for any $(i,j)\in{\cal E}\times\Z$,
\begin{multline*}
\E_{(i,j)}\left( \sum_{n=0}^\infty \1_{\{{\cal M}(n) = k, {\cal A}(n) = \ell\}}\right) \\ \leq \E_{(k,\ell)}\left( \sum_{n=0}^\infty \1_{\{{\cal M}(n) = k, {\cal A}(n) = \ell\}}\right) = \E_{(k,0)}\left( \sum_{n=0}^\infty \1_{\{{\cal M}(n) = k, {\cal A}(n) = 0\}}\right), 
\end{multline*} 
by the dominated convergence theorem and using \eqref{MAP_eq4}, it follows that 
\begin{multline*}
 \lim_{\ell\to \infty}  \sum_{\substack{(i',j')\in{\cal E}\times\Z\\ j < k_0}} \P_{(i,j)}\bigl({\cal M}({\mathcal T})= i', {\cal A}({\mathcal T}) = j'\bigr) \E_{(i',j')}\left( \sum_{n=0}^\infty \1_{\{{\cal M}(n) = k, {\cal A}(n) = \ell\}}\right) \\
 =  \sum_{\substack{(i',j')\in{\cal E}\times\Z\\j < k_0}} \P_{(i,j)}\bigl({\cal M}({\mathcal T})= i', {\cal A}({\mathcal T}) = j'\bigr) \frac{\pi(k)}{V}. 
\end{multline*} 
Since clearly,
\begin{equation}
\label{MAP_eq5c} 
   \sum_{\substack{(i',j')\in{\cal E}\times\Z\\ j < k_0}} \P_{(i,j)}\bigl({\cal M}({\mathcal T})= i', {\cal A}({\mathcal T}) = j'\bigr) = \P_{(i,j)}\bigl({\mathcal  T} <  \infty\bigr),
\end{equation} 
the last relation combined with \eqref{MAP_eq5b}  proves that 
\begin{equation*}
\lim_{\ell\to \infty} \E_{(i,j)}\left( \sum_{n=0}^{{\mathcal  T} - 1} \1_{\{{\cal M}(n) = k, {\cal A}(n) = \ell\}}\right)  = \frac{\pi(k)}{V}  \left(1 - \P_{(i,j)}\bigl({\mathcal  T} <  \infty\bigr)  \right) = \frac{\pi(k)\P_{(i,j)}({\mathcal  T} =  \infty)}{V} .\qedhere
\end{equation*}
\end{proof} 
This result will be used to prove Theorem~\ref{thm:main-1}. To prove Theorem~\ref{thm:main-2}, we need an exponential rate of convergence in  \eqref{MAP_eq5a}.

We will now assume the five following hypotheses: 
\setcounter{Assumption}{6}
\begin{Assumption}
\label{MAPas3}
The Markov chain $\{{\cal M}(n)\}$ is aperiodic and irreducible on ${\cal E}$.
\end{Assumption} 
\begin{Assumption}
\label{MAPas4}
There exist a positive function $f: {\cal E} \to [1,\infty)$ and a finite subset ${\cal E}_0\subset{\cal E}$ such that 
\begin{enumerate}
\item for any $i\in{\cal E}_0$, 
\begin{equation*}
   \E_i\bigl(f({\cal M}(1))\bigr) < \infty;
\end{equation*}
\item there is $\delta > 0$ such that for any $i\in{\cal E}\setminus{\cal E}_0$, 
\begin{equation*}
   \E_i\bigl(f({\cal M}(1))\bigr) \leq \exp(-\delta) f(i). 
\end{equation*}
\end{enumerate} 
\end{Assumption} 
\begin{Assumption}
\label{MAPas5}
The function 
\begin{equation*}
   \alpha \mapsto \sup_{i\in{\cal E}} \E_{(i,0)}\bigl(\exp(\alpha {\cal A}(1))\bigr)  
\end{equation*}
is finite in some neighborhood of $\alpha=0$. 
\end{Assumption}
Observe that thanks to Assumptions~\ref{MAPas3} and \ref{MAPas4}, the Markov chain $\{{\cal M}(n)\}$ is geometrically ergodic and therefore positive recurrent (see Chapter~15 in the book of Meyn and Tweedie~\cite{MeTw-93}). As above, we will denote by 
 $\{\pi(i)\}_{i\in{\cal E}}$ its stationary distribution, and we define the quantity $V$ by \eqref{MAP_eq2}. To see that under the above assumptions, the series \eqref{MAP_eq2} absolutely converges, it is sufficient to notice that  for $\alpha > 0$ small enough, using Assumption~\ref{MAPas5}, 
 \begin{equation*}
    \sup_{i\in{\cal E}} \E_{(i,0)}\bigl(\vert{\cal A}(1)\vert\bigr)   \leq \frac{1}{\alpha} \sup_{i\in{\cal E}} \max\left\{ \E_{(i,0)}\bigl(\exp(\alpha {\cal A}(1))\bigr), \E_{(i,0)}\bigl(\exp(- \alpha {\cal A}(1))\bigr)\right\} <  \infty.  
\end{equation*}

 \begin{Assumption}
 \label{MAPas6}
 $V > 0$.
 \end{Assumption}

\begin{Assumption}
\label{MAPas7}
There exist $\alpha > 0$  and $\eps > 0$ such that for any $(i,j)\in{\cal E}\times\Z$ with $i\not\in{\cal E}_0$,  
\begin{equation*}
   \E_{(i,j)}\bigl(\exp(\alpha ({\cal A}(1)-j)) \bigr) < \exp(-\eps).
\end{equation*}
\end{Assumption} 
\noindent
Without any restriction of generality, we will assume that the set ${\mathcal E}_0$ in Assumption~\ref{MAPas7} is the same as in Assumption~\ref{MAPas4}.

In order to prove that the convergence in \eqref{MAP_eq5a} is exponential, we first prove that the convergence in  \eqref{MAP_eq4} is exponential and uniform with respect to the initial position $(i,0)$.
Assumption~\ref{MAPas5} is the classical Cramer condition, it is needed to get an exponential rate of convergence in \eqref{MAP_eq4}. Assumption~\ref{MAPas7} is needed for the exponential convergence in \eqref{MAP_eq4} to be uniform with respect to the initial position $(i,0)$.

We consider $t(k)$ as in \eqref{eq:def_t_n}, and begin our analysis with the following preliminary results. 

\begin{lem}
\label{MAP_lem1}
Under the hypotheses \ref{MAPas3}--\ref{MAPas5}, for any $k\in{\cal E}$, there exist two positive constants $c_k,\theta_k$ such that for any $i\in{\cal E}$ and $n\in\N$, 
\begin{equation*}
   \P_{(i,0)} \bigl(t(k) > n\bigr) \leq c_k f(i) \exp(-\theta_k n). 
\end{equation*}
\end{lem}

\begin{proof}
This lemma is a consequence of Theorem~15.2.6 in the book of Meyn and Tweedie~\cite{MeTw-93}. 
\end{proof}

\begin{lem}
\label{MAP_lem2}
Under the hypotheses \ref{MAPas3}--\ref{MAPas5}, for any $k\in{\cal E}$, there exist two positive constants $c_k',\theta_k'$ such that for any $i \in{\cal E}$ and $n\in\N$, 
\be\label{estimates_lem_2pp_eq1}
   \P_{(i,0)} \bigl( \vert {\cal A}(t(k))\vert \geq n \bigr) \leq c_k' f(i) \exp(-\theta_k' n).
\ee
\end{lem}

\begin{proof}
For any $i\in\N$, $n\in\N$ and $\varkappa > 0$, 
\begin{align*}
   \P_{(i,0)}\bigl(\vert {\cal A}(t(k))\vert \geq n\bigr) &\leq \P_{(i,0)}\bigl(\vert {\cal A}(t(k)) \vert  \geq n,  t(k) \leq \varkappa n\bigr) + \P_{(i,0)}\bigl(t(k) >  \varkappa n\bigr) \\
&\leq \sum_{1 \leq s \leq \varkappa n} \P_{(i,0)}\bigl( \vert {\cal A}(s) \vert  \geq n\bigr) + \P_{(i,0)}\bigl(t(k) >  \varkappa n\bigr).
\end{align*} 
Hence, for any $\delta > 0$, using the Markov inequality, one gets 
\begin{equation*}
   \P_{(i,0)}\bigl(\vert {\cal A}(t(k))\vert  \geq n\bigr) \leq  \exp(-\delta n) \sum_{1 \leq s \leq \varkappa n} \E_{(i,0)}\bigl(\exp(\delta  \vert {\cal A}(s) \vert )\bigr) + \P_{(i,0)}\bigl(t(k) >  \varkappa n\bigr).
\end{equation*}
On the one hand, using Assumption~\ref{MAPas3}, for $\delta > 0$ small enough, 
\begin{equation*}
   \E_{(i,0)}\bigl(\exp(\delta  \vert {\cal A}(s) \vert )\bigr) \leq C^s,
\end{equation*} 
with some $C > 1$ not depending on $i\in\N$ and $s\in\N$. On the other hand, by  Lemma~\ref{MAP_lem1}, for any  $k\in\N$, there exist $c_k> 0$ and $\theta_k>0$ such that for any $i\in\N$, $n\in\N$ and $\varkappa>0$, 
\begin{equation*}
   \P_{(i,0)}\bigl(t(k) >  \varkappa n\bigr)  \leq c_k f(i) \exp(\theta_k \varkappa n).
\end{equation*}
When combined, these inequalities imply that 
\begin{equation*}
   \P_{(i,0)}\bigl(\vert {\cal A}(t(k)) \vert  \geq n\bigr) \leq \exp(-\delta n) \frac{C^{\varkappa n}}{C-1}  + c_k f(i) \exp(- \theta_k \varkappa n).
\end{equation*}
Since $f(i)\geq 1$ by Assumption~\ref{MAPas4},  letting $\varkappa = \delta /(2\log C)$, one gets \eqref{estimates_lem_2pp_eq1} with some $c_k'> 0$ and $\theta_k' = \min\{\delta, \theta_k\varkappa\}/2$.
\end{proof}

We are now ready to prove an exponential rate of convergence in \eqref{MAP_eq4}. This is a subject of the following lemma. 
\begin{lem}
\label{MAP_lem4}
Under the hypotheses \ref{MAPas3}--\ref{MAPas6}, for any $k\in{\cal E}$, there exists $\theta''_k > 0$ such that for any $i\in{\cal E}$ and $\ell \in\N$, 
\be
\label{prop2_2_eq1} 
     \left\vert  \E_{(i,0)}\left( \sum_{n=0}^{\infty} \1_{\{{\cal M}(n) = k, {\cal A}(n) = \ell\}}\right) -  \frac{\pi_1(k)}{V_1} \right\vert  \leq C_{i,k} \exp\bigl( - \theta''_k \ell\bigr),
\ee
with some $C_{i,k} > 0$ not depending on $\ell\in\N$. 
\end{lem}

\begin{proof}
To prove this lemma, we investigate the generating functions (for all $i,k\in{\cal E}$)
\begin{equation}
\label{eq:def_Green_MAP}
     {\cal G}_{(i,0)\to (k,\cdot)}(y) = \sum_{\ell= -\infty}^\infty  \E_{(i,0)}\left( \sum_{n=0}^{\infty} \1_{\{{\cal M}(n) = k, {\cal A}(n) = \ell\}}\right)  y^\ell
\end{equation}
and 
\begin{equation}\label{def_func_Phi}
   \Phi_{(i,0)\to (k,\cdot)}(y)  =   \sum_{\ell= -\infty}^\infty \P_{(i,0)}\bigl({\cal A}(t(k)) = \ell\bigr) y^\ell. 
\end{equation}
Remark that  by Lemma~\ref{MAP_lem2}, for any $k\in{\cal E}$, the Laurent series \eqref{def_func_Phi} converges in the annulus
\begin{equation*}
   {\cal C}_{\theta_k'} = \{ y\in\C: e^{-\theta_k'} < \vert y\vert < e^{\theta_k'}\},
\end{equation*}
for some $\theta_k' > 0$ depending on $k$. Moreover, using the Markov property and classical arguments from renewal theory, for any $k\in{\cal E}$ and $y\in{\cal C}_{\theta_k'}$, one gets 
\begin{equation*}
   {\cal G}_{(i,0)\to (k,\cdot)}(y)  = \delta_{i,k} + \Phi_{(i,0)\to (k,\cdot)}(y) {\cal G}_{(k,0)\to (k,\cdot)}(y)
\end{equation*}
with $\delta_{i,k}$ denoting the Kronecker symbol
\begin{equation*}
   \delta_{i,k} = \begin{cases} 1 &\text{if $i=k$},\\
0 &\text{otherwise.} 
\end{cases} 
\end{equation*}
Hence, for all those $y\in {\cal C}_{\theta_k'}$ for which 
$
   \Phi_{(k,0)\to (k,\cdot)}(y) \not= 1
$,
one gets 
\begin{equation*}
    {\cal G}_{(k,0)\to (k,\cdot)}(y) =  \frac{1}{ 1 - \Phi_{(k,0)\to (k,\cdot)}(y)}  
\end{equation*}
and for any $i\not=k$, 
\begin{equation*}
   {\cal G}_{(i,0)\to (k,\cdot)}(y) =   \frac{\Phi_{(i,0)\to (k,\cdot)}(y)}{1-  \Phi_{(k,0)\to (k,\cdot)}(y)}. 
 \end{equation*}
Remark furthermore that 
$
   \Phi_{(k,j)\to (k,\cdot)}(1) = 1
$
and using \eqref{MAP_eq3}, 
\be\label{prop2_1_eq2}
\frac{d}{dy} \Phi_{(k,0)\to (k,\cdot)}(1) = \E_{(k,0)}\bigl({\cal A}(t_1(k))\bigr) =  \frac{V}{\pi(k)} > 0.
\ee
Recall moreover that the Markov-additive process $\{({\cal M}(n),{\cal A}(n))\}$ is irreducible in its state space ${\cal E}\times\Z$, and hence, for any $k\in\N$, the positive matrix $\bigl(\P_{(k,j)}\bigl({\cal A}(t_1(k)) = \ell\bigr)\bigr)_{j,\ell\in\Z}$ is also irreducible. This proves that the sequence of integers $\ell\in\Z$ for which $\P_{(k,j)}\bigl({\cal A}(t_1(k)) = \ell\bigr) > 0$ is aperiodic and consequently, for any point $y$ in the unit circle $\vert y\vert = 1$ such that $y\not= 1$,
\begin{equation*}
   \vert \Phi_{(k,0)\to (k,\cdot)}(y) \vert < \Phi_{(k,0)\to (k,\cdot)}(1) = 1.
\end{equation*}
 Hence, for some $\theta''_k  > 0$, the functions \eqref{eq:def_Green_MAP} are meromorphic in a neighborhood of the annulus ${\cal C}_{\theta''_k}$  and has there a unique, simple pole at the point $y=1$, with residue $\pi(k)/V$. This proves that for any $i\in{\cal E}$ and $\ell \in\N$,  \eqref{prop2_2_eq1}  holds with some $C_{i,k} > 0$ not depending on $\ell\in\N$. 
\end{proof} 

We just proved an exponential rate of convergence in \eqref{MAP_eq4}, and would like to show now that this convergence is uniform with respect to the initial state $(i,0)$. 
To that purpose, for a given $k\in{\cal E}$, we define a stopping time $T(k)$ by letting 
\begin{equation}
\label{eq:def_T(k)}
   T(k) = \inf\bigl\{n\geq 1: {\cal M}(n)\in{\cal E}_0\cup\{k\}\bigr\} . 
\end{equation}
\begin{lem}
\label{MAP_lem5}
Under Assumption~\ref{MAPas7}, for any $k\in{\cal E}$, $i\in{\cal E}\setminus{\cal E}_0$ and $\ell\in\Z$, 
\begin{equation*}
   \P_{(i,0)}\bigl( {\cal A}(T(k)) = \ell \bigr) \leq \frac{\exp(-\alpha \ell )}{1-\exp(-\eps)}. 
\end{equation*}
\end{lem}
\begin{proof} For any $n\in\N$, using Assumption~\ref{MAPas7} one gets 
\begin{align*}
\P_{(i,0)}\bigl( {\cal A}(T(k)) = \ell, T(k) = n \bigr)  &\leq \P_{(i,0)}\bigl( {\cal A}(n) = \ell, T(k) \geq n \bigr)\\
&\leq \exp(-\alpha \ell)  \E_{(i,0)}\bigl( \exp(\alpha{\cal A}(n));  T(k) \geq n \bigr)\\
&\leq \exp(-\alpha \ell -\eps n), 
\end{align*} 
and consequently, 
\begin{equation*}
     \P_{(i,j)}\bigl({\cal A}(T(k)) = \ell\bigr) 
     =  \sum_{n =1}^\infty \P_{(i,j)}\bigl({\cal A}(T(k)) = \ell, T(k) = n\bigr)
     \leq   \frac{\exp(-\alpha \ell )}{1-\exp(-\eps)}.\qedhere
\end{equation*} 
\end{proof} 
Now we are ready to prove that the exponential convergence in \eqref{MAP_eq4} is uniform with respect to the initial state $(i,0)$. 

\begin{lem}
\label{MAP_lem6}
Under the hypotheses \ref{MAPas3}--\ref{MAPas7}, for any $k\in{\cal E}$, there exist positive constants $C_k,\delta_k$ such that for any $i\in{\cal E}$ and $\ell\in\Z$, 
\be\label{MAP_eq6}
\left|  \E_{(i,j)}\left( \sum_{n=0}^{\infty} \1_{\{{\cal M}(n) = k, {\cal A}(n) = \ell\}}\right) -   \frac{\pi(k)}{V} \right|  \leq C_k \exp(-\delta_k (\ell -j)).
\ee
\end{lem} 

\begin{proof}
It is sufficient to prove Lemma~\ref{MAP_lem6} for $j=0$, since clearly
\begin{equation*}
  \E_{(i,j)}\left( \sum_{n=0}^{\infty} \1_{\{{\cal M}(n) = k, {\cal A}(n) = \ell\}}\right)  =   \E_{(i,0)}\left( \sum_{n=0}^{\infty} \1_{\{{\cal M}(n) = k, {\cal A}(n) = \ell-j\}}\right).
\end{equation*}
Let $i\in{\cal E}$ and $(k,\ell)\in {\cal E}\times\Z$ with $i \notin{\cal E}_0$. It comes from the Markov property that
\begin{align*}
     \E_{(i,0)}&\left( \sum_{n=0}^{\infty} \1_{\{{\cal M}(n) = k, {\cal A}(n) = \ell\}}\right)  \\
     &= \sum_{\substack{i'\in{\cal E}_0\cup\{k\}\\ j'\in\Z} }
     \P_{(i,0)}\bigl({\cal M}(T(k)) = i', {\cal A}(T(k))= j' \bigr)  \E_{(i',j')}\left( \sum_{n=0}^{\infty} \1_{\{{\cal M}(n) = k, {\cal A}(n) = \ell\}}\right) \\
     &= \sum_{\substack{i'\in{\cal E}_0\cup\{k\}\\ j'\in\Z} }
     \P_{(i,0)}\bigl({\cal M}(T(k)) = i', {\cal A}(T(k))= j' \bigr)  \E_{(i',0)}\left( \sum_{n=0}^{\infty} \1_{\{{\cal M}(n) = k, {\cal A}(n) = \ell-j'\}}\right).
\end{align*} 
Moreover, since the Markov chain $\{{\cal M}(n)\}$ is recurrent, 
\begin{equation*}
      \sum_{\substack{i'\in{\cal E}_0\cup\{k\}\\ j'\in\Z} }    \P_{(i,0)}\bigl({\cal M}(T(k)) = i', {\cal A}(T(k))= j' \bigr) = \P_{(i,0)}\bigl(T(k)< \infty\bigr) = 1,
\end{equation*}
and consequently
\begin{multline*} 
\left\vert  \E_{(i,0)}\left( \sum_{n=0}^{\infty} \1_{\{{\cal M}(n) = k, {\cal A}(n) = \ell\}}\right)  -  \frac{\pi(k)}{V} \right\vert  \leq  \sum_{\substack{i'\in{\cal E}_0\cup\{k\}\\ j'\in\Z} } \P_{(i,0)}\bigl({\cal M}(T(k)) = i', {\cal A}(T(k))= j' \bigr) \\ 
\times \left\vert \E_{(i',0)}\left( \sum_{n=0}^{\infty} \1_{\{{\cal M}(n) = k, {\cal A}(n) = \ell - j'\}}\right)  -  \frac{\pi(k)}{V} \right\vert .  
\end{multline*} 
In order to derive \eqref{MAP_eq6}, we now split the right-hand side of the above inequality into two parts $\Sigma_1$ and $\Sigma_2$: in $\Sigma_1$ we consider the summation over $i'\in{\cal E}_0\cup\{k\}$ and $j' \geq \ell/2$, and in $\Sigma_2$ we consider the summation over $i'\in{\cal E}_0\cup\{k\}$ and $j' < \ell/2$.

To estimate $\Sigma_1$, we combine the straightforward relations  
\begin{align*} 
\left\vert \E_{(i',0)}\left( \sum_{n=0}^{\infty} \1_{\{{\cal M}(n) = k, {\cal A}(n) = \ell - j'\}}\right)  -  \frac{\pi(k)}{V} \right\vert   &\leq  \E_{(i',0)}\left( \sum_{n=0}^{\infty} \1_{\{{\cal M}(n) = k, {\cal A}(n) = \ell - j'\}}\right)  + \frac{1}{V} \\
&\hspace{-6cm} \leq  \E_{(k,\ell-j')}\left( \sum_{n=0}^{\infty} \1_{\{{\cal M}(n) = k, {\cal A}(n) = \ell - j'\}}\right)  + \frac{1}{V} = \E_{(k,0)}\left( \sum_{n=0}^{\infty} \1_{\{{\cal M}(n) = k, {\cal A}(n) = 0\}}\right) + \frac{1}{V} 
\end{align*}
with Lemma~\ref{MAP_lem5}:
\begin{align*}
     \Sigma_1  &\leq \sum_{\substack{i'\in{\cal E}_0\cup\{k\}\\j' \geq \ell/2} }\P_{(i,0)}\bigl({\cal M}(T(k)) = i', {\cal A}(T(k))= j' \bigr)  \left(\E_{(k,0)}\left( \sum_{n=0}^{\infty} \1_{\{{\cal M}(n) = k, {\cal A}(n) = 0\}}\right)  + \frac{1}{V}\right)  \\
     &\leq   \sum_{j'\geq \ell/2}  \P_{(i,0)}\bigl({\cal A}(T(k))= j' \bigr) \left(\E_{(k,0)}\left( \sum_{n=0}^{\infty} \1_{\{{\cal M}(n) = k, {\cal A}(n) = 0\}}\right)  + \frac{1}{V}\right) \\
     &\leq  \frac{1}{1-\exp(-\eps)}  \left(\E_{(k,0)}\left( \sum_{n=0}^{\infty} \1_{\{{\cal M}(n) = k, {\cal A}(n) = 0\}}\right)  + \frac{1}{V}\right) \sum_{j'\geq \ell/2} \exp(-\alpha j' ) \\
     &\leq  \frac{1}{1-\exp(-\eps)}  \left(\E_{(k,0)}\left( \sum_{n=0}^{\infty} \1_{\{{\cal M}(n) = k, {\cal A}(n) = 0\}}\right)  + \frac{1}{V}\right) \frac{ \exp(-\alpha \ell/2 )}{1-\exp(-\alpha)}  .
\end{align*}  
In order to estimate $\Sigma_2$, we use Lemma~\ref{MAP_lem4}: 
\begin{align*}
\Sigma_2 &\leq  \sum_{\substack{i'\in{\cal E}_0\cup\{k\}\\ j'  < \ell/2} } \P_{(i,0)}\bigl({\cal M}(T(k)) = i', {\cal A}(T(k))= j' \bigr)  C_{i',k} \exp\bigl( - \theta''_k (\ell-j')\bigr) \\
&\leq  \exp\bigl( - \theta''_k \ell/2\bigr)  \sum_{\substack{i'\in{\cal E}_0\cup\{k\}\\j'  < \ell/2} } \P_{(i,0)}\bigl({\cal M}(T(k)) = i', {\cal A}(T(k))= j' \bigr)  C_{i',k}  \\
&\leq  \exp\bigl( - \theta''_k \ell/2\bigr)  \max_{i'\in{\cal E}_0\cup\{ k\}} C_{i',k}.
\end{align*} 
When combined, these estimates prove \eqref{MAP_eq6} with $\delta_k = \min\{\alpha,\theta''_k\}/2$ and 
\begin{equation*}
     C_k = \frac{1}{1-\exp(-\eps)(1-\exp(-\alpha))}  \left(\E_{(k,0)}\left( \sum_{n=0}^{\infty} \1_{\{{\cal M}(n) = k, {\cal A}(n) = 0\}}\right)  + \frac{1}{V}\right)  +\max_{i'\in{\cal E}_0\cup\{ k\}} C_{i',k}.\qedhere
\end{equation*}
\end{proof}  

We are now ready to prove that the convergence in \eqref{MAP_eq5a} holds at an exponential rate. 

\begin{prop}
\label{MAP_prop1}
Under the hypotheses \ref{MAPas3}--\ref{MAPas7}, for any $k\in{\cal E}$, there exist positive constants $\widetilde{C}_k,\widetilde\delta_k$ such that for any $i\in{\cal E}$, $j\geq k_0$ and $\ell\geq k_0$, 
\be\label{MAP_eq7}
\left|  \E_{(i,j)}\left( \sum_{n=0}^{{\mathcal  T}-1} \1_{\{{\cal M}(n) = k, {\cal A}(n) = \ell\}}\right) -    \frac{\pi(k)\P_{(i,j)}({\mathcal  T} =  \infty)}{V} \right| \leq \widetilde{C}_k \exp(-\widetilde\delta_k (\ell - j)).
\ee
\end{prop} 
\begin{proof}
To prove this proposition, we combine Lemma~\ref{MAP_lem6}  with the  identities \eqref{MAP_eq5b} and \eqref{MAP_eq5c}. Using \eqref{MAP_eq5b} and \eqref{MAP_eq5c}, we get
\begin{multline*}
   \left\vert   \E_{(i,j)}\left( \sum_{n=0}^{{\mathcal  T}-1} \1_{\{{\cal M}(n) = k, {\cal A}(n) = \ell\}}\right)  -  \frac{\pi(k)}{V}\P_{(i,j)}\bigl({\mathcal  T} = \infty\bigr) \right\vert  \leq  \left\vert   \E_{(i,j)}\left( \sum_{n=0}^{\infty} \1_{\{{\cal M}(n) = k, {\cal A}(n) = \ell\}}\right)  -  \frac{\pi(k)}{V} \right\vert   \\
+\sum_{\substack{i'\in{\cal E}\\ j'\in\Z, j' < k_0} }
     \P_{(i,j)}\bigl({\cal M}({\mathcal  T}) = i', {\cal A}({\mathcal  T})= j' \bigr) \left| \E_{(i',j')}\left( \sum_{n=0}^{\infty} \1_{\{{\cal M}(n) = k, {\cal A}(n) = \ell\}}\right) -  \frac{\pi(k)}{V}  \right|
\end{multline*} 
and using then Lemma~\ref{MAP_lem6}, we conclude that for any $k\in\N$, there exist $\delta_k > 0$ and $C_k > 0$ such that for any $i\in{\cal E}$, $j\geq k_0$ and $\ell\geq k_0$, 
\begin{align*}
&\hspace{-0.5cm}\left\vert   \E_{(i,j)}\left( \sum_{n=0}^{{\mathcal  T}-1} \1_{\{{\cal M}(n) = k, {\cal A}(n) = \ell\}}\right)  -  \frac{\pi(k)}{V}\P_{(i,j)}\bigl({\mathcal  T} = \infty\bigr) \right\vert  \\ 
&\leq  C_{k} \exp\bigl( - \delta_k (\ell-j)\bigr) + \sum_{\substack{i'\in{\cal E}\\ j'\in\Z, j' < k_0} } \P_{(i,j)}\bigl({\cal M}({\mathcal  T}) = i', {\cal A}({\mathcal  T})= j' \bigr)  C_{k} \exp\bigl( - \delta_k (\ell-j')\bigr)  \\
&\leq  C_k \exp\bigl( - \delta_k (\ell-j)\bigr) \bigl(1 + \exp(\delta_k (k_0-1))\bigr).
\end{align*}
Proposition~\ref{MAP_prop1} is therefore proved. 
\end{proof} 

\subsection{Applications to the local process $Z_1$}
Recall that the process $Z_1=(X_1,Y_1)$ is Markov-additive with Markovian part $X_1$ on $\N$ and additive part $Y_1$ in $\Z$. It satisfies all hypotheses~\ref{MAPas3}--\ref{MAPas7}: 
\begin{itemize}
\item Assumption~\ref{MAPas3} is satisfied thanks to Assumptions~\ref{as3} and \ref{as4};
\item Assumption~\ref{MAPas4} is satisfied with ${\cal E}_0 = \{0,\ldots,k_0-1\}$ and the function $f:\N \to [1,\infty)$ defined by
\begin{equation*}
   f(i) = e^{\theta i}
\end{equation*}
with $\theta > 0$ small enough, because of Assumptions~\ref{as1}, \ref{as2} and \ref{as5}; 
\item Assumption~\ref{MAPas5} is satisfied thanks to Assumptions~\ref{as1} and \ref{as2p};
\item Assumption~\ref{MAPas6} is satisfied thanks to Assumption~\ref{as6};
\item Assumption~\ref{MAPas7} is satisfied with ${\cal E}_0 = \{0,\ldots,k_0-1\}$, thanks to Assumptions~\ref{as1}, \ref{as2} and \ref{as5}. 
\end{itemize} 
Hence, from Propositions~\ref{MAP_prop1a} and~\ref{MAP_prop1}, for $\tau_{1}^{\textnormal{loc}}$ as in \eqref{eq:def_tau_1},
one immediately gets: 
\begin{prop}
\label{prop3_1}
Under the hypotheses \ref{as1}--\ref{as6}, for any $k\in{\cal E}$, $i\in\N$ and $j\geq k_0$, 
\be\label{sec3_eq1a}
\lim_{\ell\to\infty}\E_{(i,j)}\left( \sum_{n=0}^{\tau_{1}^{\textnormal{loc}}-1} \1_{\{X_1(n) = k, Y_1(n) = \ell\}}\right)  =   \frac{\pi_1(k)\P_{(i,j)}\bigl(\tau_{1}^{\textnormal{loc}}  =  \infty\bigr)}{V_1}.
\ee
If moreover Assumption~\ref{as2p} is satisfied, then
for any $k\in{\cal E}$, 
there exist ${C}_k > 0$ and $\delta_k > 0$ such that for any $i\in\N$, $j\geq k_0$ and $\ell\geq k_0$, 
\be\label{sec3_eq1}
\left|  \E_{(i,j)}\left( \sum_{n=0}^{\tau_{1}^{\textnormal{loc}}-1} \1_{\{X_1(n) = k, Y_1(n) = \ell\}}\right) -    \frac{\pi_1(k)\P_{(i,j)}\bigl(\tau_{1}^{\textnormal{loc}}  =  \infty\bigr)}{V_1} \right| \leq {C}_k \exp(-\delta_k (\ell - j)) .
\ee
\end{prop} 
Obviously, a similar result holds for the local process $Z_2=(X_2,Y_2)$, which is Markov-additive with a Markovian part $Y_2$ on $\N$ and an additive part $X_2$ in $\Z$.

\subsection{Applications to the original  process $Z$}
Regarding the original process $Z=(X,Y)$ and the stopping time $\tau$ as in \eqref{eq:def_tau}, Proposition~\ref{prop3_1} yields the following result: 
\begin{prop}
\label{prop3_1p}
Under the hypotheses \ref{as1}--\ref{as6}, for any $i,j,k\in\N$, 
\be\label{sec3_eq1b}
\lim_{\ell\to\infty}\E_{(i,j)}\left( \sum_{n=0}^{\tau -1} \1_{\{X(n) = k, Y(n) = \ell\}}\right) =  \frac{\pi_1(k)\P_{(i,j)}\bigl(\tau  =  \infty\bigr)}{V_1}.
\ee
If moreover Assumption~\ref{as2p} is satisfied, then
 for any $k\in\N$, there exist $C'_k > 0$ and $\delta_k' > 0$ such that for any $i,j\in\N$  and $\ell\geq k_0$, 
\be\label{sec3_eq1p}
\left|  \E_{(i,j)}\left( \sum_{n=0}^{\tau -1} \1_{\{X(n) = k, Y(n) = \ell\}}\right) -    \frac{\pi_1(k)\P_{(i,j)}\bigl(\tau  =  \infty\bigr)}{V_1} \right| \leq C'_k \exp(-\delta_k'(\ell - j)).
\ee
\end{prop}

\begin{proof}
In the case $j\geq k_0$, this statement follows from Propositions~\ref{prop3_1} in a straightforward way. Indeed, before the stopping time $\tau $, the process $Z$ has the same transition probabilities as the local process $Z_1$ before the time $\tau_{1}^{\textnormal{loc}}$. 

Consider now  $i\in\N$, $0\leq j < k_0$ and $\ell\geq k_0$. 
Using the Markov property, one gets 
\begin{multline}\label{sec3_eq2a} 
\E_{(i,j)}\left( \sum_{n=0}^{\tau -1} \1_{\{X(n) = k, Y(n) = \ell\}}\right)  = \\ \sum_{(i',j')\in\N^2, j' \geq k_0} \P_{(i,j)}\bigl(Z(1) = (i',j')\bigr) \E_{(i',j')}\left( \sum_{n=0}^{\tau -1} \1_{\{X(n) = k, Y(n) = \ell\}}\right)
\end{multline}
and 
\begin{equation}\label{sec3_eq2b} 
   \P_{(i,j)}\bigl(\tau  =  \infty\bigr) = \sum_{(i',j')\in\N^2, j' \geq k_0} \P_{(i,j)}\bigl(Z(1) = (i',j')\bigr) \P_{(i',j')}\bigl(\tau  =  \infty\bigr).
\end{equation}
Observe moreover that for any $i',j',k,\ell\in\N$ such that $k\geq k_0$ and $\ell\geq k_0$, 
\begin{multline}\label{sec3_eq2c} 
\E_{(i',j')}\left( \sum_{n=0}^{\tau -1} \1_{\{X(n) = k, Y(n) = \ell\}}\right) \\ \leq \E_{(k,\ell)}\left( \sum_{n=0}^{\infty} \1_{\{X_1(n) = k, Y_1(n) = \ell\}}\right) 
= \E_{(k,0)}\left( \sum_{n=0}^{\infty} \1_{\{X_1(n) = k, Y_1(n) = 0\}}\right). 
\end{multline} 
Since for $j\geq k_0$, \eqref{sec3_eq1b} is already proved by Propositions~\ref{prop3_1}, then by the dominated convergence theorem it follows that 
\begin{align*}
\lim_{\ell\to\infty} \E_{(i,j)}\left( \sum_{n=0}^{\tau -1} \1_{\{X(n) = k, Y(n) = \ell\}}\right)  &=  \sum_{(i',j')\in\N^2, j' \geq k_0} \P_{(i,j)}\bigl(Z(1) = (i',j')\bigr) \frac{\pi_1(k)}{V_1} \\ &= \frac{\pi_1(k)\P_{(i,j)}\bigl(\tau  =  \infty\bigr)}{V_1}.
\end{align*}
Relation \eqref{sec3_eq1b} is proved therefore  for any $k\in\N$. Furthermore, when combined together, relations \eqref{sec3_eq2a} and \eqref{sec3_eq2b} imply that
\begin{multline*}
\left\vert  \E_{(i,j)}\left( \sum_{n=0}^{\tau -1} \1_{\{X(n) = k, Y(n) = \ell\}}\right) -  \frac{\pi_1(k)}{V_1}\P_{(i,j)}\bigl(\tau  = \infty\bigr) \right\vert   \\  \leq  \sum_{i'\geq 0, j' \geq k_0}  \P_{(i,j)}\bigl(Z(1) = (i',j')\bigr)  \left\vert  \E_{(i',j')}\left( \sum_{n=0}^{\tau -1} \1_{\{X(n) = k, Y(n) = \ell\}}\right) -  \frac{\pi_1(k)}{V_1}\P_{(i',j')}\bigl(\tau  = \infty\bigr)\right\vert .
\end{multline*} 
In order to derive \eqref{sec3_eq1p}, we now split the right-hand side above into two parts, $\Sigma_1$ and $\Sigma_2$. The term $\Sigma_1$ corresponds to the summation over $(i',j')\in\N^2$ such that  $j' > \ell/2$, and $\Sigma_2$ when $(i',j')\in\N^2$ such that  $k_0 \leq j' \leq \ell/2$. To estimate $\Sigma_1$, we combine the straightforward relations 
\begin{align*}
&\left\vert  \E_{(i',j')}\left( \sum_{n=0}^{\tau -1} \1_{\{X(n) = k, Y(n) = \ell\}}\right) -  \frac{\pi_1(k)}{V_1}\P_{(i',j')}\bigl(\tau  = \infty\bigr)\right\vert  \\
&\hspace{2cm}\leq \E_{(i',j')}\left( \sum_{n=0}^{\tau -1} \1_{\{X(n) = k, Y(n) = \ell\}}\right) +  \frac{\pi_1(k)}{V_1} \\
&\hspace{2cm}\leq \E_{(i',j')}\left( \sum_{n=0}^{\infty} \1_{\{X(n) = k, Y(n) = \ell\}}\right) +  \frac{\pi_1(k)}{V_1} = g\bigl((i',j')\to(k,\ell)\bigr) +  \frac{\pi_1(k)}{V_1}
\end{align*}
with Proposition \ref{sec4_lem1}. This proves that 
\begin{equation*}
   \Sigma_1\leq  \sum_{i'\geq 0, j' > \ell/2}  \P_{(i,j)}\bigl(Z(1) = (i',j')\bigr) \left( A_k + \frac{1}{V_1}\right) \leq   \P_{(i,j)}\bigl(Y(1) >  \ell/2\bigr) \left( A_k + \frac{1}{V_1}\right),
\end{equation*}
with 
\begin{equation*}
   A_k = \sup_{i,j,\ell\in\N} g\bigl((i,j)\to(k,\ell)\bigr) <  \infty. 
\end{equation*}
Moreover, applying the Markov inequality and using Assumption~\ref{as2}, we obtain for $0 \leq j < k_0$ and $\alpha = (0,\delta)$ with $\delta > 0$ small enough,  
\begin{equation*}
     \P_{(i,j)}\bigl(Y(1) >  \ell/2\bigr)  \leq \exp(-\delta \ell/2) \E_{(i,j)}\bigl(\exp\langle\alpha, Z(1)\rangle\bigr)  \leq  C \exp\bigl(-\delta \ell/2 + \delta (k_0-1)\bigr)
\end{equation*}
with some $C > 0$ not depending on $i$ and $\ell$. 
Therefore,
\begin{equation}
\label{eq:estimate_sigma1}
     \Sigma_1  \leq C \left( A_k +  \frac{1}{V_1}\right) \exp\bigl(-\delta \ell/2 + \delta (k_0-1)\bigr).
\end{equation}
To estimate $\Sigma_2$, we use  Proposition~\ref{prop3_1}:
\begin{equation}
\label{eq:estimate_sigma2}
     \Sigma_2  \leq  \sum_{i'\geq 0,  k_0 \leq j' \leq \ell/2}  \P_{(i,j)}\bigl(Z(1) = (i',j')\bigr){C}_{k} \exp\bigl( - \delta_k (\ell-j')\bigr) \leq {C}_{k} \exp\bigl( - \delta_k \ell/2\bigr).
\end{equation}  
When combined,
\eqref{eq:estimate_sigma1} and \eqref{eq:estimate_sigma2} prove \eqref{sec3_eq1p}, with $\delta_k' = \min\{\delta_k, \delta\}/2$ and 
\begin{equation*}
     C_k' = \widetilde{C}_k + C \left( A_k +  \frac{1}{V_1}\right) \exp\bigl(\delta (k_0-1)\bigr).\qedhere
\end{equation*}
\end{proof}

\section{Preliminary results to the proof of Theorem~\ref{thm:main-1} and Theorem~\ref{thm:main-2}}
\label{sec:4}

Consider for any $k\in\mathbb N$ the stopping times
\begin{equation}
\label{eq:T1_T2_k}
     {T}_{k} =\inf\{n > 0: X(n) \leq (k_0-1)\vee k\},
\end{equation}
and let $\tau $ be defined in \eqref{eq:def_tau} (see also Section~\ref{sec:hitting_times}).

The following lemma is needed for the proof of both Theorem~\ref{thm:main-1} and Theorem~\ref{thm:main-2}. 

\begin{lem}
\label{sec4_lem3}
Under Assumptions~\ref{as1}--\ref{as6}, for any $i,j,k\in\N$ and $\ell > k_0$, the following identities hold: 
\begin{multline}
\label{sec4_eq1}
     {g}\bigl((i,j)\to (k,\ell)\bigr) =\E_{(i,j)}\left( \sum_{n=0}^{\tau-1} \1_{\{X(n) = k, Y(n) = \ell\}}\right)   + \sum_{\substack{0\leq j'< k_0\\ i'\geq 0}} {g}\bigl((i,j)\to(i',j')\bigr) \times \\ \times \sum_{\substack{0\leq i''\leq( k_0-1)\vee k\\ j'' \geq k_0}} \P_{(i',j')}\bigl(Z({T}_{k}) = (i'',j''), {T}_{k}< \tau\bigr) \E_{(i'',j'')}\left( \sum_{n=0}^{\tau-1} \1_{\{X(n) = k, Y(n) = \ell\}}\right)  ,
\end{multline} 
\begin{align}
\label{sec4_eq2}
     \P_{(i,j)}\bigl(\tau = \infty\bigr) &= \P_{(i,j)}\bigl( {\cal N}_1 = 0\bigr),\\
\label{sec4_eq3}
     \P_{(i,j)}\bigr( \tau = \infty\bigr) &=
     \P_{(i,j)}\bigl( {T}_{k}< \infty,  \tau = \infty\bigr),
\end{align}
and
\begin{multline}
\label{sec4_eq4} 
     \sum_{\substack{0\leq j'< k_0\\ i'\geq 0}} {g}\bigl((i,j)\to(i',j')\bigr) \hspace{-3mm}\sum_{\substack{0\leq i''\leq (k_0-1)\vee k, \\j'' \geq k_0}}\!\!\!\P_{(i',j')}\bigl(Z({T}_{k}) = (i'',j''),  {T}_{k}< \tau\bigr) \P_{(i'',j'')}\bigl(\tau = \infty\bigr)  \\  =\P_{(i,j)}\bigl( 1 \leq {\cal N}_1 < \infty\bigr).  
\end{multline}
\end{lem} 
\begin{proof}
To prove the identity \eqref{sec4_eq1}, let us observe that for any $n > 0$, 
\begin{equation*}
\P_{(i,j)}\bigl(Z(n)=(k,\ell)\bigr) = \P_{(i,j)}\bigl(Z(n)=(k,\ell), \tau > n \bigr)  + \P_{(i,j)}\bigl(Z(n)=(k,\ell), \tau \leq n\bigr), 
\end{equation*}
with 
\begin{multline*}
     \P_{(i,j)}\bigl(Z(n)=(k,\ell), \tau \leq n\bigr) \\= 
     \sum_{s=1}^{n} \sum_{i'\geq 0, 0\leq j' < k_0} \P_{(i,j)}\bigl(Z(s)=(i',j')\bigr)\P_{(i',j')}\bigl(Z(n-s) = (k,\ell), \tau > n-s\bigr).
\end{multline*} 
Moreover, according to the definition of the stopping time ${T}_{k}$, on the event $\{Z(n)=(k,\ell)\}$ one has ${T}_{k}\leq n$, and consequently, 
\begin{multline*}
\P_{(i',j')}\bigl(Z(n) = (k,\ell), \tau > n\bigr) = \P_{(i',j')}\bigl(Z(n) = (k,\ell), \tau > n \geq {T}_{k}\bigr)\\
=\sum_{\substack{0\leq i''\leq (k_0-1)\vee k\\j'' \geq k_0}} \sum_{s=1}^{n} \P_{(i',j')}\bigl(Z(s) = (i'',j''), {T}_{k}= s\bigr) \P_{(i'',j'')}\bigl(Z(n-s) = (k,\ell), \tau > n-s \bigr).
\end{multline*} 
Hence, for any $i,j,k\in\N$ and $\ell \geq k_0$, 
\begin{multline*}
{g}\bigl((i,j)\to (k,\ell)\bigr) = \E_{(i,j)}\left( \sum_{n=0}^{\tau-1} \1_{\{X(n) = k, Y(n) = \ell\}}\right)  \\
+ \sum_{\substack{0\leq j'< k_0\\ i'\geq 0}} {g}\bigl((i,j)\to(i',j')\bigr) \E_{(i',j')}\left( \sum_{n=0}^{\tau-1} \1_{\{X(n) = k, Y(n) = \ell\}}\right) ,
\end{multline*}
with 
\begin{multline*}
\E_{(i',j')}\left( \sum_{n=0}^{\tau-1} \1_{\{X(n) = k, Y(n) = \ell\}}\right) = \sum_{\substack{0\leq i''\leq( k_0-1)\vee k\\ j'' \geq k_0}}    \P_{(i',j')}\bigl(Z({T}_{k}) = (i'',j''), {T}_{k}< \tau\bigr)\times \\
\times\E_{(i'',j'')}\left( \sum_{n=0}^{\tau-1} \1_{\{X(n) = k, Y(n) = \ell\}}\right).
\end{multline*}
When combined, the above relations prove \eqref{sec4_eq1}.

Equation \eqref{sec4_eq2} follows from  the definition of the stopping time $\tau$ and the variable ${\cal N}_1$, see \eqref{eq:def_N1_N2} and \eqref{eq:def_tau}.

The identity \eqref{sec4_eq3} will follow rather easily from the definition of the stopping times $\tau$, ${T}_{k}$, ${T}_{k}^{\textnormal{loc}}$ and $\tau_{1}^{\textnormal{loc}}$, where  ${T}_{k}^{\textnormal{loc}} $ and $\tau_{1}^{\textnormal{loc}}$ are defined in the same way as the stopping times $\tau$ and ${T}_{k}$ but for the local process $Z_1=(X_1,Y_1)$: see \eqref{eq:def_tau_1} and 
\be\label{eq:def_T_k_loc} 
T_k^{\textnormal{loc}} =\inf\{n > 0: X_1(n) \leq (k_0-1)\vee k\}.
\ee
We first consider the case of $(i,j)\in\N^2$ with $j \geq k_0$. We have
\begin{equation*}
     \P_{(i,j)}\bigl( {T}_{k}< \infty,  \tau = \infty\bigr) = \P_{(i,j)}\bigl( {T}_{k}^{\textnormal{loc}} < \infty,  \tau_{1}^{\textnormal{loc}} = \infty\bigr) = \P_{(i,j)}\bigl( \tau_{1}^{\textnormal{loc}} = \infty\bigr),
\end{equation*} 
where the last relation holds because $m_1 < 0$ and consequently $\P_{(i,j)}\bigl({T}_{k}^{\textnormal{loc}} < \infty\bigr) = 1$. Moreover, since for $(i,j)\in\N^2$ with $j\geq k_0$
\begin{equation*}
     \P_{(i,j)}\bigl(  \tau_{1}^{\textnormal{loc}}  = \infty\bigr) = \P_{(i,j)}\bigl(\tau = \infty\bigr),
\end{equation*}
we conclude that for $(i,j)\in\N^2$ with $j \geq k_0$, \eqref{sec4_eq3} holds. We now deal with pairs $(i,j)\in\N^2$ satisfying $0\leq j < k_0$. We then have
\begin{align*}
\P_{(i,j)}\bigl( {T}_{k}< \infty, \tau = \infty\bigr)  &= \P_{(i,j)}\bigl( {T}_{k}=1,  \tau = \infty\bigr) + \P_{(i,j)}\bigl( 1 < {T}_{k}< \infty,  \tau = \infty\bigr) \\
&= \sum_{\substack{0\leq i'\leq (k_0-1)\vee k  \\j' \geq k_0}} \P_{(i,j)}\bigl(Z(1)=(i',j')\bigr)\P_{(i',j')}\bigl(\tau = \infty\bigr)  \\&\hspace{1cm} + \sum_{\substack{ i' > (k_0-1)\vee k  \\j' \geq k_0}}\P_{(i,j)}\bigl(Z(1)=(i',j')\bigr) \P_{(i',j')}\bigl( {T}_{k}< \infty, \tau = \infty\bigr) \\
&= \sum_{\substack{0\leq i'\leq (k_0-1)\vee k  \\j' \geq k_0}}\P_{(i,j)}\bigl(Z(1)=(i',j')\bigr)\P_{(i',j')}\bigl(\tau = \infty\bigr)  \\&\hspace{1cm} + \sum_{\substack{ i' > (k_0-1)\vee k  \\j' \geq k_0}}\P_{(i,j)}\bigl(Z(1)=(i',j')\bigr) \P_{(i',j')}\bigl( \tau = \infty\bigr) \\
&= \P_{(i,j)}\bigl( {T}_{k}=1, \tau = \infty\bigr) + \P_{(i,j)}\bigl(  {T}_{k}> 1, \tau = \infty\bigr) \\ &= \P_{(i,j)}\bigl( \tau = \infty\bigr).
\end{align*} 
Relation \eqref{sec4_eq3} is therefore proved.

We conclude by providing the proof of \eqref{sec4_eq4}. We start by the right-hand side of \eqref{sec4_eq4} and we prove that it equals the probability $\P_{(i,j)}\bigl( 1 \leq {\cal N}_1 < \infty\bigr)$: 

\begin{align*}
\sum_{\substack{0\leq j'< k_0\\i'\geq 0}} &{g}\bigl((i,j)\to(i',j')\bigr)\sum_{\substack{0\leq i''\leq (k_0-1)\vee k\\j'' \geq k_0}} \P_{(i',j')}\bigl(Z({T}_{k}) = (i'',j''),  {T}_{k}< \tau\bigr) \P_{(i'',j'')}\bigl(\tau = \infty\bigr)\nonumber\\ 
&= \sum_{\substack{0\leq j'< k_0\\i'\geq 0}} {g}\bigl((i,j)\to(i',j')\bigr)  \sum_{\substack{0\leq i''\leq (k_0-1)\vee k\\j'' \geq k_0}}\P_{(i',j')}\bigl(Z({T}_{k}) = (i'',j''),  {T}_{k}< \infty,  \tau = \infty\bigr)\nonumber\\
&= \sum_{\substack{0\leq j'< k_0\\i'\geq 0}} {g}\bigl((i,j)\to(i',j')\bigr)  \P_{(i',j')}\bigl( {T}_{k}< \infty,  \tau = \infty\bigr)\nonumber\\
&= \sum_{\substack{0\leq j'< k_0\\i'\geq 0}} {g}\bigl((i,j)\to(i',j')\bigr)  \P_{(i',j')}\bigl(\tau = \infty\bigr)\\
&=\sum_{n=0}^\infty \sum_{\substack{0\leq j'< k_0\\ i'\geq 0}} \P_{(i,j)}\bigl( Z(n) = (i',j') \text{ and } \forall s\geq n,\,Y(s) \geq k_0\bigr) \\
&=  \sum_{n=0}^\infty  \P_{(i,j)}\bigl( Y(n) < k_0 \text{ and } \forall s\geq n,\, Y(s) \geq k_0\bigr) \\
&=   \P_{(i,j)}\bigl( \exists n\geq 1 \text{ such that } Y(n) < k_0  \text{ and } \forall s\geq n,\,Y(s) \geq k_0 \bigr) \\
&= \P_{(i,j)}\bigl( 1\leq {\cal N}_1 < \infty\bigr).
\end{align*} 
The proof is complete.
\end{proof} 

The following lemma is needed for the proof of Theorem~\ref{thm:main-2}:
\begin{lem}
\label{sec4_lem2}
Under Assumptions~\ref{as1}, \ref{as2p} and \ref{as5}, for any $k\geq 0$, there exist $\widetilde{C} >0$ and $\widetilde\delta > 0$ such that 
\be\label{sec4eq4} 
     \P_{(i, j)}\bigl( Y({T}_{k}) = \ell , {T}_{k}<  \tau\bigr) \leq \widetilde{C} \exp\bigl(-\widetilde\delta (i + \ell)\bigr), \quad \forall i \geq k_0,\,j\in\{0,\ldots,k_0-1\},\,\ell\geq 0. 
\ee
\end{lem}

\begin{proof} 
Consider the Laplace transform of the increments of the random walk $Z_0$, namely,
\begin{equation}
\label{eq:def_R}
     \Phi_{k_0,k_0}(\alpha) = \E_{(0,0)}\bigl(\exp\langle\alpha, Z_0(1)\rangle\bigr).
\end{equation}
Remark that $\nabla \Phi_{k_0,k_0}(0) = m = (m_1,m_2)$, and consequently,  for any $\alpha =(\alpha_1,\alpha_2)\in\R^2$ satisfying the inequality $\langle\alpha, m\rangle < 0$, the  function $s \mapsto \Phi_{k_0,k_0}(\alpha s)$ is decreasing in a neighborhood of $0$. Hence, letting $\alpha_1 = 0$ and $\alpha_2 = \theta > 0$, for $\theta > 0$ small enough, one gets
\begin{equation}
\label{eq:R<1}
     \Phi_{k_0,k_0}(\alpha) < 1.
\end{equation} 
Moreover, for all $\theta > 0$ small enough,  by Assumption \ref{as2p}, there is $C > 0$ such that 
\begin{equation}\label{C}
    \max_{ i \geq 0, 0 \leq j \leq k_0}  \E_{(i,j)}\bigl(\exp(\theta  Y(1))\bigr)  =  \max_{ i \geq 0, 0 \leq j \leq k_0}  \Phi_{(i,j)}\bigl((0,\theta)\bigr)  \leq C.
\end{equation}
To prove our lemma, we chose $\alpha=(0,\theta)$ with $\theta > 0$ small enough so that \eqref{eq:R<1} and \eqref{C} hold. 
Then, for any  $n\geq0$ and $k\geq 0$, the Markov inequality yields 
\begin{align*} 
     \P_{(i, j)}\bigl( Y({T}_{k}) = \ell, {T}_{k} = n <  \tau\bigr) &= \P_{(i, j)}\bigl( Y(n) = \ell ,  {T}_{k} = n <  \tau\bigr) \\
&\leq \exp(- \theta \ell) \E_{(i,j)}\bigl( \exp(\theta Y(n)); {T}_{k} = n <  \tau\bigr).
\end{align*} 
Furthermore, it follows from the definitions \eqref{eq:def_tau} and \eqref{eq:T1_T2_k} of the stopping times $\tau$ and ${T}_{k}$ that  
\begin{align*} 
     \E_{(i,j)}&\bigl( \exp(\theta Y(n)); {T}_{k}  = n <  \tau\bigr)  \\
&=\sum_{\substack{i' > (k_0 - 1)\vee k\\ j'> k_0}} \P_{(i,j)}\bigl(Z(1) = (i',j')\bigr) \E_{(i',j')}\bigl( \exp(\theta Y(n-1)); {T}_{k} = n - 1<  \tau\bigr) \\
&\leq \sum_{\substack{i' > (k_0-1)\vee k\\ j'> k_0}} \P_{(i,j)}\bigl(Z(1) = (i',j')\bigr) \E_{(i',j')}\bigl( \exp\langle\alpha, Z_0(n-1)\rangle\bigr). 
\end{align*}
Hence for $\alpha=(0,\theta)$, 
\begin{align*} 
     \E_{(i,j)}&\bigl( \exp(\theta Y(n)); {T}_{k}  = n <  \tau\bigr)  \\
&\leq \sum_{\substack{i' > (k_0-1)\vee k\\ j'> k_0}} \P_{(i,j)}\bigl(Z(1) = (i',j')\bigr) \exp\bigl(\theta j') \Phi_{k_0,k_0}(\alpha)^{n-1} \\
&\leq \E_{(i,j)}\bigl(\exp(\theta Y(1))\bigr) \Phi_{k_0,k_0}(\alpha)^{n-1},
\end{align*}
with $\Phi_{k_0,k_0}$ as in \eqref{eq:def_R}. Using next \eqref{eq:R<1} and \eqref{C}, we conclude that for all $n>1$, 
\begin{equation}\label{eq_sec4_lem2}
     \P_{(i, j)}\bigl( Y({T}_{k}) = \ell , {T}_{k}  = n <  \tau\bigr) \leq C \Phi_{k_0,k_0}(\alpha)^{n-1} \exp(- \theta \ell), 
\end{equation}
with $\Phi_{k_0,k_0}(\alpha) < 1$. Remark now that thanks to Assumption~\ref{as1}, the steps $X(n+1)-X(n)$ of the process $\{X(n)\}$ are bounded below by $-k_0$ and hence, for $n < (i-(k_0-1)\vee k)/k_0$, 
\begin{equation*}
   \P_{(i, j)}\bigl( Y({T}_{k}) = \ell , {T}_{k}  = n <  \tau\bigr)  = 0. 
\end{equation*}
This proves that 
\begin{equation*}
     \P_{(i, j)}\bigl( Y({T}_{k}) = \ell , {T}_{k} <  \tau\bigr) =\sum_{n \geq (i-(k_0-1)\vee k)/k_0} \P_{(i, j)}\bigl( Y({T}_{k}) = \ell ,  {T}_{k} = n <  \tau\bigr),
\end{equation*}
and hence, using \eqref{eq_sec4_lem2}, we obtain 
\begin{align*} 
     \P_{(i, j)}\bigl( Y({T}_{k}) = \ell , {T}_{k} <  \tau\bigr) 
     &\leq C  \exp(- \theta \ell)  \sum_{n \geq (i-(k_0-1)\vee k)/k_0} \Phi_{k_0,k_0}(\alpha)^{n-1} \\
     &\leq C \exp(- \theta \ell) (1-\Phi_{k_0,k_0}(\alpha))^{-1}  \Phi_{k_0,k_0}(\alpha)^{(i - (k_0-1)\vee k)/k_0 - 1}. 
\end{align*}
The last relation proves \eqref{sec4eq4} with $\widetilde\delta = \min\{\theta, \frac{-\log \Phi_{k_0,k_0}(\alpha)}{k_0}\}$ and 
\begin{equation*}
     \widetilde{C} = C_0 (1-\Phi_{k_0,k_0}(\alpha))^{-1} \Phi_{k_0,k_0}(\alpha)^{ - ((k_0-1)\vee k)/k_0 - 1}.\qedhere
\end{equation*}
\end{proof} 

\section{Proof of Theorem~\ref{thm:main-1}} 
\label{sec:5p}
To prove this theorem, we apply the dominated convergence theorem in \eqref{sec4_eq1}. Recall that by \eqref{sec4_eq1}, 
\begin{multline}\label{sec5p_eq1}
    \lim_{\ell\to\infty}  {g}\bigl((i,j)\to (k,\ell)\bigr) = \lim_{\ell\to\infty} \E_{(i,j)}\left( \sum_{n=0}^{\tau-1} \1_{\{X(n) = k, Y(n) = \ell\}}\right)   \\+  \lim_{\ell\to\infty} \sum_{\substack{0\leq i''\leq( k_0-1)\vee k\\ j'' \geq k_0}} C_{i,j}(i'',j'')\E_{(i'',j'')}\left( \sum_{n=0}^{\tau-1} \1_{\{X(n) = k, Y(n) = \ell\}}\right)  ,
\end{multline} 
where
\begin{equation*}
C_{i,j}(i'',j'') = \sum_{\substack{0\leq j'< k_0\\ i'\geq 0}} {g}\bigl((i,j)\to(i',j')\bigr) \P_{(i',j')}\bigl(Z({T}_{k}) = (i'',j''), {T}_{k}< \tau\bigr).
\end{equation*}
Moreover, for any $k\in\N$ and $(i,j)\in\N^2$, by Proposition~\ref{prop3_1p}, 
\begin{equation*}
\lim_{\ell\to\infty}  \E_{(i,j)}\left( \sum_{n=0}^{\tau-1} \1_{\{X(n) = k, Y(n) = \ell\}}\right) = \frac{\pi_1(k)}{V_1}\P_{(i,j)}\bigl(\tau = \infty\bigr) 
\end{equation*}
and by \eqref{sec4_eq4} combined with \eqref{sec4_eq2},
\begin{multline*}
\P_{(i,j)}\bigl(\tau = \infty\bigr) + \sum_{\substack{0\leq i''\leq( k_0-1)\vee k\\ j'' \geq k_0}} C_{i,j}(i'',j'') \P_{(i'',j'')}\bigl(\tau = \infty\bigr) \\ = \P_{(i,j)}\bigl({\cal N}_1 = 0\bigr) + \P_{(i,j)}\bigl(1\leq {\cal N}_1 < \infty\bigr) = \P_{(i,j)}\bigl({\cal N}_1 < \infty\bigr). 
\end{multline*} 
Hence to get \eqref{eq:thm:main-1a} it is sufficient to show that the order of the limit as $\ell\to\infty$ and the summation over $0\leq i''\leq( k_0-1)\vee k$ and  $j'' \geq k_0$ in the right-hand side of \eqref{sec5p_eq1} can be exchanged. 
Remark moreover that by Proposition~\ref{sec4_lem1}, there is $C>0$ such that for all $i'',j'',\ell\in\N$, 
\begin{equation*}
\E_{(i'',j'')}\left( \sum_{n=0}^{\tau-1} \1_{\{X(n) = k, Y(n) = \ell\}}\right) \leq \E_{(i'',j'')}\left( \sum_{n=0}^{\infty} \1_{\{X(n) = k, Y(n) = \ell\}}\right) = g\bigl(((i'',j'')\to(k,\ell)\bigr) \leq C.
\end{equation*} 
By the dominated convergence theorem, to get  \eqref{eq:thm:main-1a}, it is therefore sufficient to show that the series 
\begin{equation*}
\sum_{\substack{0\leq i''\leq( k_0-1)\vee k\\ j'' \geq k_0}} C_{i,j}(i'',j'')
\end{equation*}
converges. Since by \eqref{sec4_eq4}, 
\begin{equation*}
\sum_{\substack{0\leq i''\leq( k_0-1)\vee k\\ j'' \geq k_0}} C_{i,j}(i'',j'') \P_{(i'',j'')}\bigl(\tau<\infty\bigr) = \P_{(i,j)}\bigl(1\leq {\cal N}_1 < \infty) < \infty, 
\end{equation*}
it is sufficient to show that for some $N>0$ large enough,
\begin{equation*}
\inf_{\substack{0\leq i''\leq( k_0-1)\vee k\\ j'' \geq N}} \P_{(i'',j'')}\bigl(\tau<\infty\bigr)  > 0. 
\end{equation*}
To get the last relation, we notice that  for any $i''\in\N$,  the sequence $\bigl\{\P_{(i'',j'')}\bigl(\tau = \infty\bigr)\bigr\}_{j''\geq k_0}$ is increasing, and because of Assumption~\ref{as6}, from \eqref{fluid_limit_eq} it follows that for $j'' \geq k_0$ large enough, 
\begin{equation*}
   \P_{(i'',j'')}\bigl(\tau = \infty\bigr) > 0. 
\end{equation*}
Hence for $N\geq k_0$ large enough,
\begin{equation*}
   \inf_{\substack{0\leq i''\leq( k_0-1)\vee k\\ j'' \geq N}} \P_{(i'',j'')}\bigl(\tau<\infty\bigr) = \inf_{\substack{0\leq i''\leq( k_0-1)\vee k\\ j'' = N}} \P_{(i'',j'')}\bigl(\tau<\infty\bigr) > 0. 
\end{equation*}
Theorem~\ref{thm:main-1} is therefore proved. 

\section{Proof of Theorem~\ref{thm:main-2}} 
\label{sec:5}
Denote for $(i,j),(k,\ell)\in\N^2$, 
\begin{align*}
     \Delta_1\bigl((i,j)\to(k,\ell)\bigr) 
     &= \left\vert {g}\bigl((i,j)\to(k,\ell)\bigr) - \frac{\pi_1(k)}{V_1}\P_{(i,j)}\bigl({\cal N}_1 < \infty\bigr)\right\vert,\\  
     \Delta^+_1\bigl((i,j)\to(k,\ell)\bigr) 
     &= \left\vert \E_{(i,j)}\left( \sum_{n=0}^{\tau-1} \1_{\{X(n) = k, Y(n) = \ell\}}\right)   - \frac{\pi_1(k)}{V_1}\P_{(i,j)}\bigl(\tau = \infty\bigr)\right\vert.
\end{align*}
From Lemma~\ref{sec4_lem3}, it follows that for any $(i,j), (k,\ell)\in\N^2$ with $\ell \geq k_0$,
\begin{multline*}
     \Delta_1\bigl((i,j)\to(k,\ell)\bigr)   \leq  \Delta^+_1\bigl((i,j)\to(k,\ell)\bigr)  + \sum_{\substack{0\leq j'< k_0\\ i'\geq 0}} {g}\bigl((i,j)\to(i',j')\bigr) \\ \times 
  \sum_{\substack{0\leq i''\leq (k_0-1)\vee k \\j'' \geq k_0}} \P_{(i',j')}\bigl(Z({T}_{k}) = (i'',j''), {T}_{k}< \tau\bigr)  \Delta^+_1\bigl((i'',j'')\to(k,\ell)\bigr).
\end{multline*} 
Recall moreover that by Proposition~\ref{sec4_lem1},
\be
\label{proof_th1_eq2}
     \widetilde{C}_k = \sup_{i,j,\ell\in\N}  g\bigl((i,j)\to(k,\ell)\bigr)   <\infty,
\ee
and by Lemma~\ref{sec4_lem2}, for any $i'\geq0$, $0\leq j'< k_0$, $0\leq i''\leq (k_0-1)\vee k$ and $j''\geq k_0$, 
\begin{equation*}
     \P_{(i',j')}\bigl(Z({T}_{k}) = (i'',j''), {T}_{k}< \tau\bigr) \leq \P_{(i',j')}\bigl(Y({T}_{k}) = j'', {T}_{k}< \tau\bigr) 
\leq  \widetilde{C} \exp\bigl(-\widetilde\delta (i' + j'')\bigr).
\end{equation*} 
Hence,   
\begin{align}
     \Delta_1\bigl((i,j)\to(k,\ell)\bigr)   &\leq  \Delta^+_1\bigl((i,j)\to(k,\ell)\bigr) \nonumber\\ &\hspace{1cm} + \widetilde{C}_k\sum_{\substack{0\leq j'< k_0\\i'\geq 0}} \sum_{\substack{0\leq i''\leq (k_0-1)\vee k \\j'' \geq k_0}} \hspace{-0.4cm}\widetilde{C} \exp\bigl(-\widetilde\delta (i' + j'')\bigr)  \Delta^+_1\bigl((i'',j'')\to(k,\ell)\bigr) \nonumber\\
&\leq  \Delta^+_1\bigl((i,j)\to(k,\ell)\bigr) + A_k  \hspace{-0.4cm}\sum_{\substack{0\leq i''\leq (k_0-1)\vee k \\j'' \geq k_0}} \exp\bigl(- \widetilde\delta j''\bigr)  \Delta^+_1\bigl((i'',j'')\to(k,\ell)\bigr), \label{proof_th1_eq1} 
\end{align}
with $A_k = \widetilde{C}_k\widetilde{C} k_0 (1-e^{-\widetilde\delta})^{-1}$.   
Recall that by Proposition~\ref{prop3_1p}, for any $i,j,k\in\N$ and $\ell\geq k_0$, 
\be
\label{proof_th1_eq3} 
     \bigl\vert\Delta_1^+\bigl((i,j)\to(k,\ell)\bigr)\bigr\vert  \leq  C_{k} \exp\bigl( - \delta_k (\ell-j)\bigr), 
\ee
with some constants $C_k > 0$ and $\delta_k > 0$ not depending on $i,j,\ell\in\N$, and 
observe moreover that for any $k\in\N$, 
\begin{align}
\sup_{i,j,\ell\in\N}  \Delta_1^+\bigl((i,j)\to(k,\ell)\bigr)  &\leq  \sup_{i,j,\ell\in\N}  g\bigl((i,j)\to(k,\ell)\bigr)  + \sup_{i,j,\ell\in\N}  \pi_1(k)\P_{(i,j)}\bigl(\tau^+ = \infty\bigr)/V_1  \nonumber\\ 
&\leq  \sup_{i,j,\ell\in\N}  g\bigl((i,j)\to(k,\ell)\bigr)  +  1/V_1 \nonumber\\ 
&\leq  \widetilde{C}_k +  1/V_1  < \infty. \label{proof_th1_eq4}
\end{align} 
To prove Theorem~\ref{thm:main-2}, we split the right-hand side of \eqref{proof_th1_eq1} into two parts:
\begin{align*}
     \Sigma_1  &=  A_k \sum_{\substack{0\leq i''\leq (k_0-1)\vee k \\j'' > \ell/2}} \exp\bigl(- \widetilde\delta j''\bigr)  \Delta^+_1\bigl((i'',j'')\to(k,\ell)\bigr),\\
     \Sigma_2 &=\Delta^+_1\bigl((i,j)\to(k,\ell)\bigr)  +  A_k \sum_{\substack{0\leq i''\leq (k_0-1)\vee k\\k_0 < j'' \leq \ell/2}} \exp\bigl(- \widetilde\delta j''\bigr)  \Delta^+_1\bigl((i'',j'')\to(k,\ell)\bigr).
\end{align*}
In order to estimate $\Sigma_1$, we use the upper bound \eqref{proof_th1_eq4}:
\begin{equation}
\label{proof_th1_eq5}
     \Sigma_1\leq A_k (\widetilde{C}_k +  1/V_1) \sum_{\substack{0\leq i''\leq(k_0-1)\vee k \\j'' > \ell/2}} \exp\bigl(- \widetilde\delta j''\bigr)  \leq  A_k (\widetilde{C}_k + 1/V_1 )\frac{(k_0-1)\vee k + 1}{1- e^{-\widetilde\delta}}  \exp\bigl(-\widetilde\delta \ell/2\bigr),
\end{equation} 
and to estimate $\Sigma_2$, we use the inequality \eqref{proof_th1_eq3}:
\begin{align}
     \Sigma_2 &\leq  C_{k} \exp\bigl( - \delta_k (\ell-j)\bigr) +  A_k \sum_{\substack{0\leq i''\leq(k_0-1)\vee k\\k_0 < j'' \leq \ell/2}} \exp\bigl(- \widetilde\delta j''\bigr) C_{k} \exp\bigl( - \delta_k (\ell-j'')\bigr)  \nonumber \\ 
&\leq C_{k} \exp\bigl( - \delta_k (\ell-j)\bigr) + A_k C_k \frac{(k_0-1)\vee k + 1}{1- e^{-\widetilde\delta}} \exp\bigl( - \delta_k \ell/2\bigr).\label{proof_th1_eq6} 
\end{align} 
When combined,
\eqref{proof_th1_eq5} and \eqref{proof_th1_eq6} show that for any $k\in\N$, there exist $C_k'>0$ and $\delta_k'> 0$ such that 
\begin{multline*}
     \left\vert{g}\bigl((i,j)\to(k,\ell)\bigr) - \frac{\pi_1(k)}{V_1}\P_{(i,j)}\bigl({\cal N}_1 < \infty\bigr)\right\vert = \left\vert \Delta\bigl((i,j)\to(k,\ell)\bigr)\right\vert  \\ \leq  \Sigma_1 + \Sigma_2 
\leq  C_k' \exp\bigl( - \delta_k'  (\ell-j)\bigr), 
\end{multline*}
and consequently, the first assertion of Theorem~\ref{thm:main-2} holds. The proof of the second assertion of this theorem is entirely similar.

\section{Preliminary results to the proof of Theorem \ref{thm:main-3}}
\label{sec:7}

In this section, our main objective is to introduce (mostly analytic) tools for the proof of Theorem~\ref{thm:main-3} (and Theorem~\ref{thm:main-4}), which will be provided in the following section, Section~\ref{sec:proofs_2}. Contrary to the previous Sections \ref{sec:4} and \ref{sec:5}, where we use purely probabilistic arguments, we move here to an analytic framework: we introduce the generating functions of the Green functions and prove that they satisfy various functional equations, starting from which we will deduce contour integral formulas for the Green functions. Applying asymptotic techniques to these integrals will finally lead to our main results.

\subsection{Functional equations for the generating functions of the Green functions}

We first introduce the kernels:
\begin{equation}
\label{eq:kernels}
     \left\{\begin{array}{rcll}
     Q(x,y,z)&=&x^{k_0}y^{k_0} \bigl(z\sum_{i,j\geq -k_0} \mu(i,j)x^iy^j-1\bigr),&\medskip\\
     q_\ell'(x,y,z)&=&x^{k_0} y^\ell \bigl(z\sum_{i \geq -k_0, j\geq -\ell } \mu_\ell'(i,j)x^iy^j-1\bigr), 
&0\leq \ell\leq k_0-1,\medskip\\
     q_k''(x,y,z)&=&x^k y^{k_0}\bigl(z\sum_{i \geq -k,j \geq -k_0 } \mu_k''(i,j)x^iy^j-1\bigr),
&0\leq k\leq k_0-1,\medskip\\
     q_{k,\ell}(x,y,z)&=&x^k y^\ell  \bigl(z\sum_{i \geq -k, j \geq -\ell} \mu_{k,\ell}(i,j) x^i y^j-1\bigr), & 0\leq k,\ell \leq k_0-1.
     \end{array}\right.
\end{equation}
Each kernel corresponds to a homogeneity domain of Figure \ref{fig:model}. Let also the generating functions of the $z$-Green functions be
\begin{equation}
\label{eq:generating_functions}   
     \left\{\begin{array}{rcll}
     G(x,y,z)&=& \sum_{n \geq 0} \sum_{i,j\geq k_0} \mathbb P_{(i_0,j_0)}\bigl(Z(n) = (i,j)\bigr) x^{i-k_0}y^{j-k_0}  z^n,&\medskip\\
     g_\ell(x,z) &=& \sum_{n \geq 0} \sum_{i\geq k_0}  \mathbb P_{(i_0,j_0)}\bigl(Z(n) = (i,\ell)\bigr)x^{i-k_0}z^n, &\hspace{-30mm}0\leq \ell\leq k_0-1,\medskip\\
     \widetilde{g}_k(y,z) &=& \sum_{n \geq 0}  \sum_{j\geq k_0 }  \mathbb P_{(i_0,j_0)}\bigl(Z(n) = (k,j)\bigr) y^{j-k_0}z^n, &\hspace{-30mm}0\leq k\leq k_0-1,\medskip\\
     f_{i_0,j_0}(x,y,z) &=& \sum_{0 \leq k,\ell\leq k_0-1 } \bigl(\sum_{n \geq 0}  \mathbb P_{(i_0,j_0)}\bigl(Z(n) = (k,\ell)\bigr) z^n \bigr)  q_{k,\ell}(x,y,z)+x^{i_0}y^{j_0}. &
     \end{array}\right.
\end{equation}
They are all well defined when $\vert x\vert<1$, $\vert y\vert<1$ and $\vert z\vert<1$.

\begin{lem}
The following equation holds true, for any $\vert x\vert<1$, $\vert y\vert<1$ and $\vert z\vert<1$:
\begin{equation}
\label{eq:main}
     -Q(x,y,z)G(x,y,z)=\sum_{\ell=0}^{k_0-1}q_\ell'(x,y,z) g_\ell(x,z)+ \sum_{k=0}^{k_0-1} q_k''(x,y,z) \widetilde{g}_k(y,z)+f_{i_0,j_0}(x,y,z).
\end{equation}
\end{lem}

\begin{proof}
The generating functions are clearly convergent for $x,y,z$ less than $1$ in modulus, as the coefficients  $ \mathbb P_{(i_0,j_0)}\bigl(Z(n) = (i,j)\bigr)=p^{(n)} \bigl((i_0, j_0) \to (i,j)\bigr)$ are smaller than $1$. In the case of nearest-neighbor random walks, the functional equation \eqref{eq:main} has been obtained in \cite[Lem.~3.16]{KuMa-98}. All generating functions remain convergent in the case of larger steps, and the proof of \eqref{eq:main} follows exactly the same line in this generalized framework.
\end{proof} 

Let now 
\begin{equation*}
     Q(x,y)=Q(x,y,1),\ q'_\ell(x,y)=q'_\ell(x,y,1),\ q''_k(x,y)=q''_k(x,y,1),\ q_{k,\ell}(x,y)=q_{k,\ell}(x,y,1)
\end{equation*}
be the evaluations of the kernels \eqref{eq:kernels} at $z=1$. Clearly, with our notations of Assumptions \ref{as1} and \ref{as2}, for any $x>0$ and $y>0$,
\begin{equation*}
     \left\{\begin{array}{rcll}
     Q(x,y)&=&x^{k_0}y^{k_0}\bigl(\Phi_{k_0, k_0}(\log x, \log y)-1\bigr),&\medskip\\
     q'_\ell (x,y)&=&x^{k_0} y^{\ell}\bigl(\Phi_{k_0, \ell}(\log x, \log y)-1\bigr), 
&0\leq \ell\leq k_0-1,\medskip\\
     q''_k(x,y)&=&x^{k} y^{k_0}\bigl(\Phi_{k, k_0}(\log x, \log y)-1\bigr),
&0\leq k\leq k_0-1,\medskip\\
     q_{k, \ell}(x,y)&=&x^k y^\ell\bigl(\Phi_{k, \ell}(\log x, \log y)-1\bigr), & 0\leq k,\ell \leq k_0-1.
     \end{array}\right.
\end{equation*}
The generating functions of the Green functions are
\begin{equation}
\label{eq:generating_functions_Green}   
     \left\{\begin{array}{rcll}
     G(x,y) &=&  \sum_{i,j\geq k_0} g \bigl((i_0, j_0) \to (i,j)\bigr)x^{i-k_0}y^{j-k_0},&\medskip\\
     g_\ell(x) &=&   \sum_{i\geq k_0} g\bigl((i_0, j_0) \to (i,\ell)\bigr)x^{i-k_0}, & \hspace{-20mm}0\leq \ell\leq k_0-1,\medskip\\
     \widetilde g_k(y) &=&  \sum_{j\geq k_0 } g \bigl((i_0, j_0) \to (k,j)\bigr)y^{j-k_0},& \hspace{-20mm}0\leq k\leq k_0-1,\medskip\\
     f_{i_0,j_0}(x,y) &=&  \sum_{0 \leq k,\ell\leq k_0-1 }  g \bigl((i_0,j_0) \to (k,\ell)\bigr) q_{k,\ell}(x,y)+x^{i_0}y^{j_0}. &
\end{array}\right.
\end{equation}
The generating functions $g_\ell$ and $\widetilde g_k$ in \eqref{eq:generating_functions_Green} are strongly related to ${\cal G}_{(i,j)\to (\cdot,\ell)}$ and ${\cal G}_{(i,j)\to (k, \cdot)}$ in \eqref{eq:GF_kcdot} and \eqref{eq:GF_lcdot}. More specifically, they just differ by polynomial terms; for example,
\begin{equation}
\label{eq:relation_GG}
     {\cal G}_{(i,j)\to (k, \cdot)}(y)=y^{k_0}\widetilde g_k(y) + \sum_{0\leq j<k_0} g\bigl((i_0, j_0) \to (k,j)\bigr)y^{j}.
\end{equation}

\begin{lem}
\label{lem:eq_func_GF}
All generating functions \eqref{eq:generating_functions_Green} are absolutely convergent on the bidisk $\{(x,y)\in\mathbb C^2: \vert x\vert  <1, \vert y\vert<1\}$, where they satisfy the functional equation:
\begin{equation}
\label{eq:main_func_eq}
     -Q(x,y) G(x,y)=\sum_{\ell=0}^{k_0-1}q_\ell'(x,y) g_\ell(x)+ \sum_{k=0}^{k_0-1} q_k''(x,y) \widetilde g_k(y)+f_{i_0,j_0}(x,y).
\end{equation}
\end{lem}

Although Equation \eqref{eq:main_func_eq} appears formally as the evaluation of \eqref{eq:main} at $z=1$, its proof needs some care, as it is not at all clear a priori that the generating function $G(x,y)$ converges for $\vert x\vert<1$ and $\vert y\vert<1$. (This convergence actually constitutes a first non-obvious estimate.)

\begin{proof}[Proof of Lemma \ref{lem:eq_func_GF}]
The series $\sum_{n \geq 0} p^{(n)} \bigl((i_0,j_0) \to (k,\ell)\bigr) $ is convergent to $g\bigl((i_0,j_0) \to (k,\ell)\bigr)$, and by Proposition \ref{sec4_lem1},
\begin{equation*}
     \sup_{\substack{i \geq k_0\\\ell\in\{0,\ldots,k_0-1\}}}  g\bigl((i_0,j_0) \to (i,\ell)\bigr)<\infty 
     \quad\text{and}\quad
     \sup_{\substack{j \geq k_0\\k\in\{0,\ldots, k_0-1\}}}  g\bigl((i_0,j_0) \to (k,j)\bigr)<\infty.
\end{equation*}
Accordingly, the series $g_\ell(x,1)$ and $\widetilde g_k(1,y)$ are absolutely convergent respectively for any $x$ and $y$ with $\vert x\vert<1$ and $\vert y\vert<1$. Then, by Abel's theorem on power series, the limit of the right-hand side of \eqref{eq:main} as $z\to 1$ exists and equals the right-hand side of \eqref{eq:main_func_eq}.

As a consequence, the limit of the left-hand side of \eqref{eq:main} does exist as well, so that $\lim_{z\to 1} G(x,y,z)$ exists  for any pair $(x,y)$ in the set
\begin{equation}
\label{eq:set_CV}
     \{(x, y) \in \mathbb C^2: \vert x\vert <1,\vert y\vert<1,  Q(x,y,1) \neq  0\}.
\end{equation}
Furthermore, for any $(x,y) \in (0,1)^{2}$, the series $G(x, y,z)$ of $z$ has real non-negative coefficients. It follows that $G(x,y,1)$ converges for any real $(x, y)$ from the set \eqref{eq:set_CV}, which is dense in $(0,1)^{2}$. Therefore $G(x,y,1)$ is absolutely convergent for any complex $(x,y)$ with $\vert x\vert ,\vert y\vert <1$ and again by Abel's theorem, $\lim_{z\to 1} G(x,y,z)=G(x,y,1)$. Hence the limit of the left-hand side of \eqref{eq:main} as $z\to 1$ exists and equals \eqref{eq:main_func_eq}. Equation~\eqref{eq:main_func_eq} is proved.
\end{proof}


\subsection{Preliminary results on the zero-set $Q(x,y)=0$}
We gather in a single lemma several statements on the zeros of the two-variable polynomial $Q(x,y)$ and on the one-variable polynomials $Q(x,1)$ and $Q(1,y)$. The last ones may be interpreted as kernels of one-dimensional random walks (and a large literature exists around this type of models). These statements will be used in the proof of Theorem~\ref{thm:main-3}.

Although, under our standing assumptions, the equation $y\mapsto Q(x,y)=0$ ($x$ being fixed) has in general an infinite number of solutions, two roots play a very special role and carry  
most of the probabilistic information about the model. They will be denoted $Y_0(x)$ and $Y_1(x)$.

\begin{lem}
\label{lem:new:add}
Under Assumptions \ref{as1}--\ref{as6}, \ref{as2pp} and \ref{as2ppp},
the following assertions hold:
\begin{enumerate}
     \item\label{it1:lem:new:add}For $\vert x \vert =1$ and $\vert y \vert =1$ with $(x,y)\ne(1,1)$, we have $Q(x,y)\neq0$.

     \item\label{it2:lem:new:add}There exists a neighborhood $O_\delta(1)=\{x\in\mathbb C: \vert x-1\vert <\delta \} $ of $1$, with $\delta>0$ small enough, inside of which there exists a unique
function $Y_0$ analytic in $O_\delta(1)$, which satisfies $Y_0(1)=1$ and $Q(x, Y_0(x))=0$ for all $x\in O_\delta(1)$. The function $Y_0$ is one-to-one from $O_\delta(1)$ onto $U_\delta(1)=Y_0(O_\delta(1))$, which is a neighborhood of $1$.

     \item\label{it3:lem:new:add}For $\epsilon>0$ small enough, $Y_0(1+\epsilon)<1$.

     \item\label{it4:lem:new:add}There exists a domain (see Figure \ref{fig:neigh1})
\begin{equation}
\label{eq:def_V}
     V= \{(1-\epsilon)+\epsilon  e^{i \phi}:  \epsilon \in (0, \epsilon_0], \phi \in [-\phi_0, \phi_0] \}
\end{equation}
such that $V \subset O_\delta(1)$, and for any $x \in V$, we have $\vert x \vert< 1$ and $\vert Y_0(x) \vert <1$.

\item\label{it5:lem:new:add}On $(0,\infty)$, the function $Q(1,y)$ admits exactly two zeros, at $y_0=1$ and $y_1>1$. Moreover, for any $y$ with $1<\vert y \vert <y_1$, $Q(1,y)\ne 0$.
      
\item\label{it6:lem:new:add}For small $\epsilon>0$, the function $Q(1+\epsilon,y)$ has exactly two zeros on $(0,\infty)$. One of them is $Y_0(1+\epsilon)$, as defined in \ref{it2:lem:new:add} and \ref{it3:lem:new:add}. The other zero is called $Y_1(1+\epsilon)$. We have $Y_0(1+\epsilon)<Y_1(1+\epsilon)$. Furthermore, for any $y$ with $Y_0(1+\epsilon)<\vert y \vert <Y_1(1+\epsilon)$, we have $Q(1+\epsilon,y)\ne 0$.      
      
\item\label{it7:lem:new:add}For $\epsilon>0$ small enough, any $x$ with $\vert x \vert =1+\epsilon$ and any $y$ with $Y_0(1+\epsilon)<\vert y \vert <Y_1(1+\epsilon)$, we have $Q(1+\epsilon,y)\ne 0$.    
    
\item\label{it8:lem:new:add}For any $\epsilon>0$ small enough and any $t>0$,
\begin{equation*}
     (1+\epsilon) Y_1(1+\epsilon)^t >\min \{ x_1, y_1^t\},
\end{equation*}
where $x_1$ is defined symmetrically as $y_1$ in \ref{it5:lem:new:add}, but in the $x$-variable.
\end{enumerate}
\end{lem}

As stated, Lemma \ref{lem:new:add} concerns the functions $Q(1,y)$, $Y_0(x)$ and $Y_1(x)$. Obviously, symmetric statements hold for $Q(x,1)$, $X_0(y)$ and $X_1(y)$.

\begin{figure}[ht!]
\includegraphics[width=0.4\textwidth]{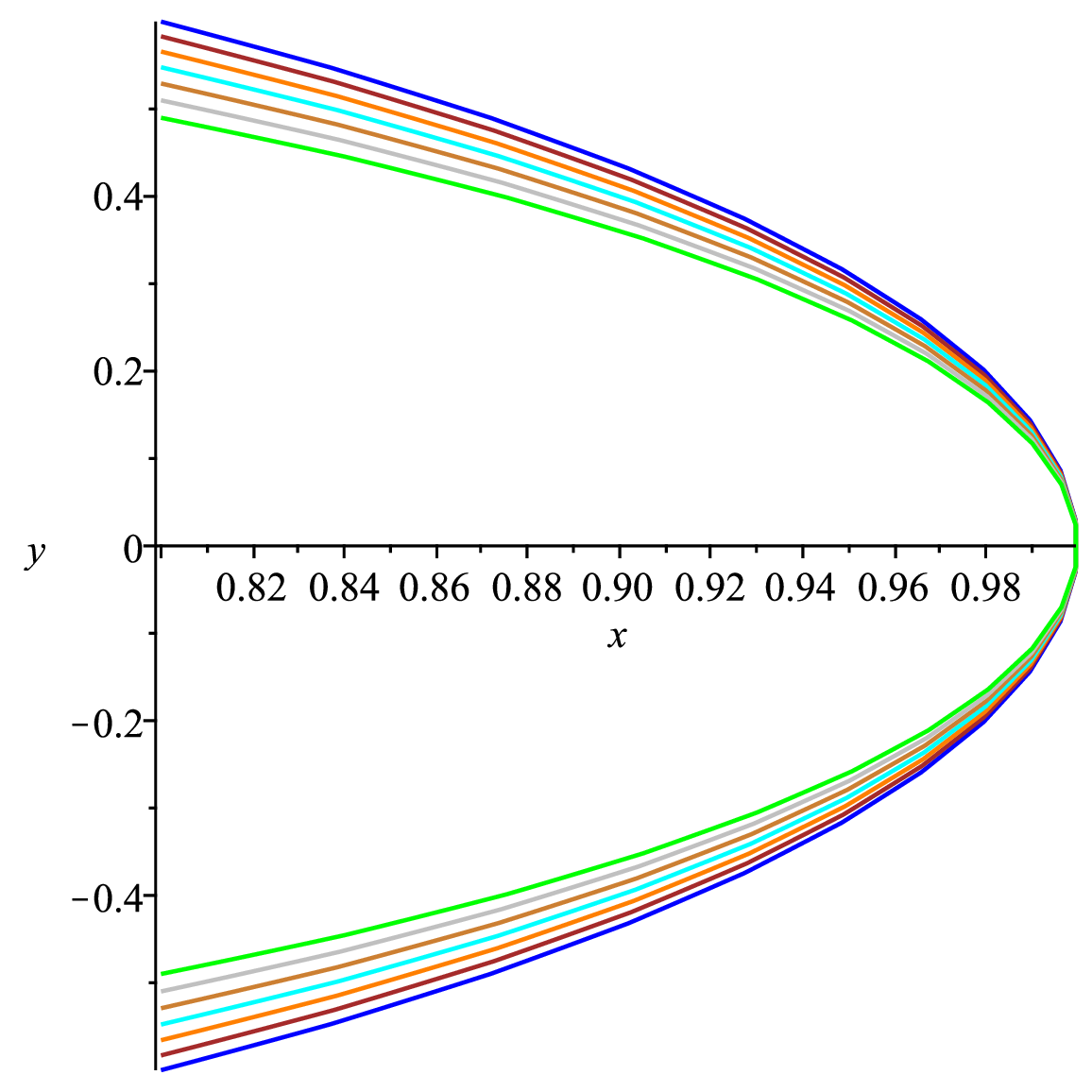}
\caption{The unit circle is in blue (the rightmost curve), and the other circles are tangent to it at $1$. These domains appear in the proof of item \ref{it4:lem:new:add} of Lemma \ref{lem:new:add}.}
\label{fig:neigh1}
\end{figure}

\begin{proof}[Proof of Lemma \ref{lem:new:add}]
Item \ref{it1:lem:new:add} will be the topic of the separate Lemma \ref{lem:poly_pos}, so we start with the proof of \ref{it2:lem:new:add}. It readily follows from the analytic implicit function theorem applied to the function $(x,y)\mapsto Q(x,y)$ in the neighborhood of $(1,1)$, noticing that $Q(1,1)=0$ and $\partial_2Q(1,1)=\sum_{i,j}j\mu(i,j)<0$ (by assumption). Similarly, \ref{it3:lem:new:add} follows from the (real version of) the implicit function theorem, using once again that $\partial_2Q(1,1)<0$.

\medskip

Before starting the proof of \ref{it4:lem:new:add}, we need the preliminary series expansion \eqref{eq:expansion_circles_0} below, which describes the behaviour of the modulus of $Y_{0}(x)$ as $x$ lies on a circle tangent to the unit circle, but with a smaller radius (see Figure \ref{fig:neigh1}). First of all, take the notation $a=Y_{0}'(1)$ and $2b=Y_{0}''(1)$. Then obviously
\begin{equation}
\label{eq:expansion_Y_0}
     Y_{0}(x) =1 + a(x-1)+b(x-1)^2+o(x-1)^2, \quad x \to 1.
\end{equation}
A few standard computations yield that for any $\epsilon\in[0,1]$, one has
\begin{equation}
\label{eq:expansion_circles_0}
     \vert Y_{0}(\epsilon +(1-\epsilon)e^{i\phi})\vert =1+\frac{1-\epsilon}{2}(a^2-a-2b +\epsilon( 2b-a^2))\phi^2+o(\phi^2),\quad \phi \to 0. 
\end{equation}
The values of $a$ and $b$ are computed below, in Lemma \ref{lem:Y_0'(1)_Y_0''(1)}, in terms of the first and second moments of the distribution $\mu$.

For later use, let us first show that $a^2-a-2b<0$. Using the explicit expressions for $a$ and $b$ given in Lemma \ref{lem:Y_0'(1)_Y_0''(1)},
\begin{equation}
\label{eq:comput_aab}
     a^2-a-2b=Y_{0}'(1)^2-Y_{0}'(1)-Y_{0}''(1)=\frac{(\mathbb EX)^2\mathbb E(Y^2)-2\mathbb EX\mathbb E(XY)\mathbb EY+\mathbb E(X^2)(\mathbb EY)^2}{(\mathbb EY)^3}.
\end{equation}
Since $\mathbb EY<0$, the denominator of \eqref{eq:comput_aab} is negative, and it is enough to prove that the numerator of \eqref{eq:comput_aab} is positive, namely,
\begin{equation}
     (\mathbb EX)^2\mathbb E(Y^2)-2\mathbb EX\mathbb EXY\mathbb EY+\mathbb E(X^2)(\mathbb EY)^2>0.
\end{equation}
As it turns out, the above inequality is a straightforward consequence of Cauchy-Schwarz inequality applied to the random variables $\frac{X}{\mathbb EX}$ and $\frac{Y}{\mathbb EY}$.

In order to construct the neighborhood $V$ in \eqref{eq:def_V}, we shall use the estimate \eqref{eq:expansion_circles_0}, as follows. Let us first observe that for $\epsilon\in[0,1]$ small enough (say $\epsilon\in[0,\epsilon_0]$), one has $a^2-a-2b +\epsilon( 2b-a^2)<0$ (indeed, $a^2-a-2b<0$, see above). One can also make the term $o(\phi^2)$ in \eqref{eq:expansion_circles_0} uniform in $\epsilon\in[0,\epsilon_0]$, as the point $x=1$ is regular for the function $Y_{0}(x)$ (and its derivatives). It follows that there exists a value $\phi_0>0$ such that for all $\phi\in(-\phi_0,\phi_0)\setminus\{0\}$ and all $\epsilon\in[0,\epsilon_0]$, $\vert Y_{0}(\epsilon +(1-\epsilon)e^{i\phi})\vert<1$. In conclusion, the neighborhood $V$ may be taken as the union of all these small arcs of circle as in \eqref{eq:def_V} (see also Figure~\ref{fig:neigh1}). Item \ref{it4:lem:new:add} is proved.

\medskip

We now prove \ref{it5:lem:new:add}.
We have $Q(1,y)=y^{k_0}(P(y)-1)$, where
\begin{equation}
\label{eq:def_P}
     P(y)=\sum_{j=-k_0}^\infty \mu(-,j)y^j,
\end{equation}
$\mu(-,j)$ denoting the second marginal of $\mu$.
By our main assumptions, the series $Q(1,y)$ has a radius of convergence $R\in(1,\infty]$. The function $P(y)$ in \eqref{eq:def_P} is well defined on $(0,R)$ and is strictly convex. Furthermore, one has $\lim_{y\to0+}P(y)=+\infty$. There exists a unique $\tau\in(0,R)$ such that $P$ is (strictly) decreasing on $(0,\tau)$ and (strictly) increasing on $(\tau,R)$; $\tau$ is called the structural constant, and $\rho=1/P(\tau)$ is called the structural radius (see \cite[Lem.~1]{BaFl-02}). One has $P'(\tau)=0$. In case of a negative drift, one has $\tau>1$, since 
\begin{equation*}
     P'(1) = \sum_{j=-k_0}^\infty j\mu(-,j)=\sum_{i,j=-k_0}^\infty j\mu(i,j)<0.
\end{equation*}
By Assumption \ref{as2pp}, 
one has $Q(1,R)\in(0,\infty]$, so that there exists $y_1 \in (\tau, R)$ such that $P(y_1)=0$. Furthermore, $P(y)\in(0,1)$ for any $y \in (1, y_1)$. 

A general fact about Laurent polynomials with non-negative coefficients enters the game: $\vert P(y)\vert \leq P(\vert y\vert)$. The inequalities $P(y)\in(0,1)$ for any $y \in (1, y_1)$ thus imply that $\vert P(y)\vert<1$ for all $1<\vert y \vert <y_1$.

\medskip 
  
We pursue by showing \ref{it6:lem:new:add}. We proceed as in the proof of \ref{it5:lem:new:add}. Using \eqref{eq:kernels}, we first rewrite the equality $Q(1+\epsilon,y)=0$ as
\begin{equation*}
     P_\epsilon(y)=1,\quad \text{where } P_\epsilon(y)=\sum_{j=-k_0}^\infty \left( \sum_{i=-k_0}^\infty \mu(i,j)(1+\epsilon)^i\right) y^j.
\end{equation*}
The polynomial $P_\epsilon$ above has non-negative coefficients and is strictly convex on $(0,R_\epsilon)$, where $R_\epsilon$ is the radius of convergence of $P_\epsilon$. For $\epsilon=0$, it is equal to $P$ as defined in \eqref{eq:def_P}. Using that $P'(1)<0$, we deduce that for $\epsilon>0$ small enough, $P_\epsilon'(Y_0(1+\epsilon))<0$. Then for $y\in(Y_0(1+\epsilon),Y_1(1+\epsilon))$, $P_\epsilon$ takes values in $(0,1)$. We conclude as in \ref{it5:lem:new:add}.

\medskip

In order to prove \ref{it7:lem:new:add}, we first write, using \eqref{eq:kernels}, that $Q(x,y)=x^{k_0}y^{k_0}\bigl(M(x,y)-1\bigr)$,
where
\begin{equation*}
     M(x,y)=\sum_{i,j\geq-k_0}\mu(i,j)x^i y^j.
\end{equation*}
Then, by positivity of the coefficients,
\begin{equation*}
     \vert M(x,y)\vert \leq M(\vert x\vert,\vert y\vert),
\end{equation*}
and we conclude using \ref{it6:lem:new:add}.

\medskip

{It remains to prove \ref{it8:lem:new:add}.} For $(x,y)\in(0,\infty)^2$, consider the function $g_t(x,y)= x y^{t} $ on $Q(x,y)=0$ and look for its extrema. Equivalently, look at the extrema of $(u,v)\mapsto e^{u +t v}$ on
\begin{equation}
\label{eq:dom_Q}
     \mathcal Q =\{(u,v)\in\mathbb R^2: \sum_{i,j} \mu(i,j) e^{iu +jv }=1\}=\{(u,v)\in\mathbb R^2: Q(e^u,e^v)=0\},
\end{equation}
see Figure \ref{fig:graphs_bis}. There are three particular points on the latter curve, namely, $(0,0)$, $(u_0,0)=(\log x_{1}, \log 1)$ and $(0, v_0)=(\log 1, \log y_{1})$.
\begin{figure}[ht!]
\includegraphics[width=0.4\textwidth]{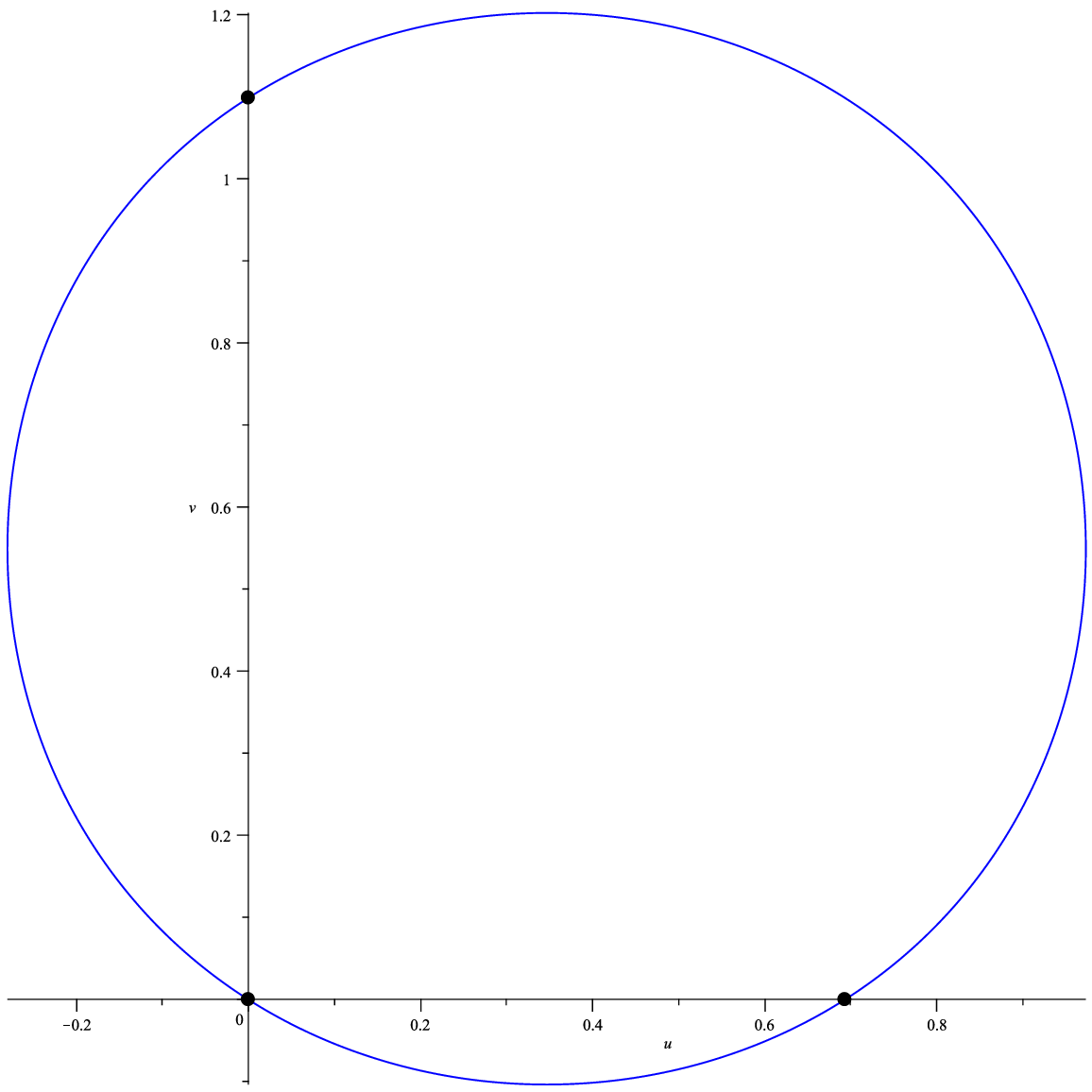}
\caption{An example of the domain $\mathcal Q$ as in \eqref{eq:dom_Q}, for the random walk with transition probabilities as in the caption of Figure~\ref{fig:examples_gamma0} (right display). The points $(0,0)$, $(u_0,0)=(\log 2,0)$ and $(0,v_0)=(0,\log 3)$ are 
visualized with bullets.}
\label{fig:graphs_bis}
\end{figure}

Let us show  that on 
\begin{equation*}
   {\mathcal P}=\mathcal Q\cap(0,\infty)^2,
\end{equation*}
which is the part of ${\mathcal Q}$ between $(u_0,0)$ and $(0,v_0)$ run counterclockwise, the function $u+tv$ is strictly bigger than its values at the boundary points:
\begin{equation}
\label{uveq}
     u+t v > \min\{u_0, t v_0\},\quad \forall (u,v) \in \mathcal{P},\ \forall t \in (0, \infty).
\end{equation}
Consider the critical points of $u+tv$ on ${\mathcal Q}$. A necessary condition is that $u'(v)=-t$. But $u'(v) =- \frac{\partial_v Q(e^u, e^v)}{\partial_u Q(e^u,e^v)}$, so that at critical points one must have
     $t=\frac{\partial_v Q(e^u, e^v)}{\partial_u Q(e^u,e^v)}$.
   On the other hand, it has been established by Hennequin \cite{He-63}
that the mapping
\begin{equation*}
     (u,v) \mapsto \frac{{\rm grad}\, Q(e^u,e^v)}{ \Vert{\rm grad}\, Q(e^u,e^v)     \Vert   }
\end{equation*}
is a diffeomorphism between ${\mathcal Q}$ and the unit circle. Then the
   critical points $(u,v)$ are the images of the points of the unit circle  such that the ratio of the second coordinate by the first coordinate equals $t$.
     There are exactly two points on the unit circle with this property,
       so that there exist exactly 
    two critical points of $u+tv$ on ${\mathcal Q}$. The function $u+tv$ being continuous on ${\mathcal Q}$, it reaches its maximum and minimum.
      Then one of these points must be its minimum on ${\mathcal Q}$  and cannot belong to ${\mathcal P}$, since the function is positive 
    on this part, while it vanishes at $(0,0) \in {\mathcal Q}$.
     The second critical point must be the maximum of $u+tv$ 
       on ${\mathcal Q}$ and may belong to ${\mathcal P}$ or not.
    Furthermore, the function $u+tv$ must be strictly monotonous on
     ${\mathcal  Q}$ between these two critical points. 
     Hence the estimate 
      \eqref{uveq} holds. It implies 
\begin{equation*}
     \log (1+\epsilon)+ t \log Y_{1} (1+\epsilon)> \min \{\log x_{1}, t \log y_{1}\},
\end{equation*}
which proves \ref{it8:lem:new:add}. 
\end{proof}

\begin{lem}
\label{lem:poly_pos}
Let $(\mu(i,j))_{i,j \geq  -k_0}$ be any family of non-negative real numbers summing to one, such that the semigroup of $\mathbb Z^2$ generated by the support $\{(i,j)\in\mathbb Z^2 : \mu(i,j)\neq0 \}$ is $\mathbb Z^2$ itself. If a pair $(x,y)\in\mathbb C^2$ with $\vert x\vert=\vert y\vert=1$ satisfies
\begin{equation*}
    \sum_{i,j\geq-k_0}\mu(i,j)x^i y^j=1,
\end{equation*}
then necessarily $x=y=1$.  
\end{lem}

Observe that the hypothesis on the semigroup is equivalent to the irreducibility of the random walk $Z_0$ on $\mathbb Z^2$ whose increment distribution is given by $\mu$. 
The proof of Lemma~\ref{lem:poly_pos} is very standard and will be omitted.

\begin{lem}
\label{lem:Y_0'(1)_Y_0''(1)}
Let $(X,Y)$ be a random vector with distribution $\mu$. One has
\begin{equation}
\label{eq:value_Y_0'(1)}
     Y_{0}'(1)=-\frac{\mathbb EX}{\mathbb EY}
\end{equation}
and
\begin{equation}
\label{eq:value_Y_0''(1)}
     Y_{0}''(1)=\frac{(\mathbb EX)^2\mathbb EY-(\mathbb EX)^2\mathbb E(Y^2)+2\mathbb EX\mathbb EXY\mathbb EY+\mathbb EX(\mathbb EY)^2-\mathbb E(X^2)(\mathbb EY)^2}{2(\mathbb EY)^3}.
\end{equation}
\end{lem}

Similarly, we could compute the derivatives of $Y_1$ at $1$ (see Lemma \ref{lem:new:add} \ref{it6:lem:new:add} for the definition of $Y_1$) in terms of the moments of a Doob transform of $(X,Y)$.

\begin{proof}[Proof of Lemma \ref{lem:Y_0'(1)_Y_0''(1)}]
Differentiating the identity $Q(x,Y_{0}(x))=0$, one obtains 
\begin{equation}
\label{eq:diff_Q_0}
     \partial_1Q(x,Y_{0}(x))+Y_{0}'(x)\partial_2Q(x,Y_{0}(x))=0,
\end{equation}
and in particular (using that $Y_{0}(1)=1$)
\begin{equation*}
     Y_{0}'(1)=-\frac{\partial_1Q(1,1)}{\partial_2Q(1,1)}=-\frac{\sum_{i,j} i\mu(i,j)}{\sum_{i,j} j \mu(i,j)} ,
\end{equation*}
which proves \eqref{eq:value_Y_0'(1)}. Differentiating now \eqref{eq:diff_Q_0}, we get
\begin{multline}
\label{eq:diff2_Q_0}
     \partial^2_{1,1}Q(x,Y_{0}(x))+2Y_{0}'(x)\partial^2_{1,2}Q(x,Y_{0}(x))+Y_{0}''(x)\partial_2Q (x,Y_{0}(x))\\+(Y_{0}'(x))^2\partial^2_{2,2}Q (x,Y_{0}(x)) 
     =0.
\end{multline}
Moreover, one easily computes
\begin{equation}
\label{eq:values_second_deriv}
\left\{\begin{array}{rcl}
     \partial^2_{1,1}Q(1,1)&=&(2k_0-1)\mathbb EX+\mathbb E(X^2),\\
     \partial^2_{2,2}Q(1,1)&=&(2k_0-1)\mathbb EY+\mathbb E(Y^2),\\
     \partial^2_{1,2}Q(1,1)&=&k_0\mathbb EX+k_0\mathbb EY+\mathbb E(XY).
     \end{array}\right.
\end{equation}
Plugging \eqref{eq:values_second_deriv} in \eqref{eq:diff2_Q_0} evaluated at $x=1$, we conclude that \eqref{eq:value_Y_0''(1)} holds.
%
\end{proof}

\subsection{One-dimensional stationary probabilities} 

In the forthcoming proof of Theorem~\ref{thm:main-3}, we need to identify the invariant measure of the stationary Markov chain $X_1$ defined in \eqref{local_trans_probabilities_1}, which is a one-dimensional reflected random walk on $\N$, whose transitions are 
given in Equation~\eqref{eq:distrib_p_1}.
Using our notation \eqref{eq:kernels}, the associated kernels are $Q(x,1)$ (in the regime when $k\geq k_0$) and $q''_k(x,1)$ (when $0\leq k<k_0$).

\begin{lem}
\label{lem:inv_mes}
The invariant measure $\{\pi_1(i)\}_{i\geq0}$ of $X_1$ can be computed as
\begin{equation}
\label{eq:exact_expr_1D}
     \pi_1(i)=\frac{1}{2\pi i} \int_{\vert x\vert =1-\epsilon} \frac{ \sum_{k=0}^{k_0-1} \pi_1(k) q''_k(x,1)  }{x^{i-k_0+1} Q(x,1)}dx=
      \frac{1}{2\pi i} \int_{\vert x\vert =1+\epsilon} \frac{ \sum_{k=0}^{k_0-1} \pi_1(k) q''_k(x,1)  }{x^{i-k_0+1} Q(x,1)}dx.
\end{equation}
As $i\to\infty$, it admits the asymptotics
\begin{equation}
\label{eq:asympt_expr_1D}
     \pi_1(i)  \sim  \frac{A_1}{x_1^i},
\end{equation}
where the constant $A_1$ is positive, and equal to
\begin{equation}
\label{eq:formulation_c}
     A_1=\frac{\sum_{k=0}^{k_0-1} \pi_1(k) q''_k(x_{1},1)  }{x_{1}^{-k_0+1} \partial_1 Q(x_{1},1)}.
\end{equation}
\end{lem}

Although Lemma \ref{lem:inv_mes} is classical in the probabilistic literature, we present some elements of proof below, in order to make our article self-contained. We thank Onno Boxma and Dmitry Korshunov for useful bibliographic advice.


\begin{proof}[Sketch of the proof of Lemma \ref{lem:inv_mes}]
Introduce the 
generating function
$
     \Pi_1(x)=\sum_{k=k_0}^\infty \pi_1(k)x^{k-k_0}.
$
Then the following functional equation holds (it is equivalent to the equilibrium equations):
\begin{equation*}
     Q(x,1)\Pi_1(x)= \sum_{k=0}^{k_0-1} \pi_1(k) q''_k(x,1).
\end{equation*}
The first integral expression in \eqref{eq:exact_expr_1D} (over $\vert x\vert=1-\epsilon$) immediately follows. Since $Q(x,1)=0$ and $\Pi_1(1)=1$, the right-hand side of the above identity is zero at $x=1$ and the function $\Pi_1$ is analytic in a neighborhood of $1$. We deduce the second integral representation in \eqref{eq:exact_expr_1D} (over $\vert x\vert=1+\epsilon$). Using the functional equation and the fact that $Q(x,1)$ has a simple pole at $x_1$ (see Lemma \ref{lem:new:add} \ref{it5:lem:new:add}), one immediately deduces the asymptotics \eqref{eq:asympt_expr_1D}, with the expression of the constant $A_1$ as in \eqref{eq:formulation_c}.

On the other hand, the (strict) positivity of the constant $A_1$ in \eqref{eq:asympt_expr_1D} is more difficult to establish (and is not clear at all from the algebraic expression of $A_1$ given in \eqref{eq:formulation_c}, as for $k\geq1$, $q_{k}''(x_1,1)$ may be negative). However, Theorem~2 in \cite{DeKoWa-19} shows the positivity of $A_1$ for a more general class of random walks; see also \cite{BoLo-96}.
\end{proof}

\section{Proof of Theorem \ref{thm:main-3}}
\label{sec:proofs_2}

Let us first summarize the proof Theorem~\ref{thm:main-3} in several important steps, to which we shall refer in the extended proof. First of all, it follows from the main functional equation \eqref{eq:main_func_eq} that for any $\epsilon>0$ small enough,
\begin{multline}
\label{eq:Green_function_double_integral}
 g\bigl((i_0,j_0) \to (i,j)\bigr)=\\
 \frac{1}{(2\pi i)^2} \iint\limits_{\vert x\vert =\vert y\vert=1-\epsilon} \frac{\sum_{\ell=0}^{k_0-1}q_\ell'(x,y) g_\ell(x)+ \sum_{k=0}^{k_0-1} q_k''(x,y) \widetilde g_k(y)+f_{i_0,j_0}(x,y)
  }{ x^{i+k_0-1}y^{j+k_0-1} Q(x,y) }dy dx.
\end{multline}
Then:
\begin{enumerate}[label=\arabic{*}.,ref=\arabic{*}]
     \item\label{it:proof_main_1}We shall apply the residue theorem to the inner integral above (in $y$), so as to split $g\bigl((i_0,j_0) \to (i,j)\bigr)$ into two terms, see \eqref{eq:Green_sum_terms}. 
     \item\label{it:proof_main_2}The first term in the decomposition \eqref{eq:Green_sum_terms} is a one-variable integral, to which we apply the classical residue theorem. Some technical work is needed here to prove that there is only one contributing pole, at $1$ (we use several properties proved in Lemma \ref{lem:new:add}).
     \item\label{it:proof_main_3}The second term in \eqref{eq:Green_sum_terms} is a double integral over $\vert x\vert =1-\epsilon$ and $\vert y\vert=1+\epsilon$. We will again apply the residue theorem to the inner integral and, in this way, obtain a further two-term decomposition and the expression \eqref{eq:Green_sum_terms2} for the Green function.
     \item\label{it:proof_main_4}The second term in the sum \eqref{eq:Green_sum_terms2} is studied via the residue theorem, in a similar way as the first term in the decomposition \eqref{eq:Green_sum_terms}. 
     \item\label{it:proof_main_5}The third term in \eqref{eq:Green_sum_terms2} is an integral on $\vert x\vert =\vert y\vert=1+\epsilon$ and is shown to be negligible.
     \item\label{it:proof_main_6}Conclusion.
\end{enumerate}

Before embarking in the proof, we state an equivalent, but more analytic version of Theorem~\ref{thm:main-2}. To that purpose, similarly to \eqref{eq:def_Green_MAP}, we introduce the following generating functions, for respectively fixed $k\in\N$ and $\ell\in\N$:
\begin{align}
\label{eq:GF_kcdot}
     {\cal G}_{(i,j)\to (k, \cdot)}(x) &= \sum_{\ell=0}^\infty g\bigl((i,j)\to(k,\ell)\bigr) x^\ell,\\
     {\cal G}_{(i,j)\to (\cdot,\ell)}(x) &= \sum_{k=0}^\infty g\bigl((i,j)\to(k,\ell)\bigr) x^k.\label{eq:GF_lcdot}
\end{align}

\begin{cor}
\label{cor:main-1-bis}
Under Assumptions \ref{as1}--\ref{as6} and \ref{as2pp}, there exists $\varepsilon > 0$ such that for any $(i,j)\in\N^2$ and any $k,\ell\in\N$, the generating functions ${\cal G}_{(i,j)\to (k, \cdot)}$ and ${\cal G}_{(i,j)\to (\cdot,\ell)}$ can be continued in a meromorphic manner in the disk $\{x\in\C: \vert x\vert < 1 + \varepsilon\}$, with a unique simple pole, which is located at the point $x=1$ and admits the residue
\begin{align}
     \Res_{1} \,{\cal G}_{(i,j)\to (k, \cdot)} &=  \pi_1(k) \P_{(i_0,j_0)} \bigl({\cal N}_1 < \infty\bigr)/V_1,\label{th1_e1}\\
     \Res_{1} \,{\cal G}_{(i,j)\to (\cdot,\ell)}    &=  \pi_2(\ell) \P_{(i_0,j_0)} \bigl({\cal N}_2 < \infty\bigr)/V_2.\nonumber
\end{align}
\end{cor} 

\begin{proof}[Proof of Theorem~\ref{thm:main-3}]
We start with \textit{Step \ref{it:proof_main_1}}. Let us fix $x$ on the circle $\vert x\vert =1-\epsilon$. Since the integrand of the inner integral in \eqref{eq:Green_function_double_integral} may be continued as a meromorphic function to the larger disc $\{\vert y\vert <1+\epsilon\}$, see Corollary \ref{cor:main-1-bis}, we write the Green function \eqref{eq:Green_function_double_integral} as
\begin{multline}
\label{eq:Green_sum_terms}
 \frac{1}{2\pi i} \int\limits_{\vert x\vert =1-\epsilon} \sum_{ y :1-\epsilon<\vert y\vert < 1+\epsilon} \Res  \frac{\sum_{\ell=0}^{k_0-1}q_\ell'(x,y) g_\ell(x)+ \sum_{k=0}^{k_0-1} q_k''(x,y) \widetilde g_k(y)+f_{i_0,j_0}(x,y)
  }{ x^{i+k_0-1}y^{j+k_0-1} Q(x,y) }dx \\+
 \frac{1}{(2\pi i)^2} \iint\limits_{\substack{\vert x\vert =1-\epsilon\\\vert y\vert =1+\e}} \frac{\sum_{\ell=0}^{k_0-1}q_\ell'(x,y) g_\ell(x)+ \sum_{k=0}^{k_0-1} q_k''(x,y) \widetilde g_k(y)+f_{i_0,j_0}(x,y)
  }{ x^{i+k_0-1}y^{j+k_0-1} Q(x,y) } dydx.
\end{multline}

\textit{Step \ref{it:proof_main_2}} consists in studying the first integral in \eqref{eq:Green_sum_terms}. Let us look at the residues appearing in the integrand. Obviously, the poles will be found among the zeros of $Q(x,y)$ and the poles of the numerator, $x$ being fixed on the circle $\vert x\vert =1-\epsilon$.

Let $O_\delta(1)$ and $U_\delta(1)$ be the neighborhoods of $1$ introduced in Lemma \ref{lem:new:add} \ref{it2:lem:new:add}. We are going to study successively three cases (recall that, in addition, we always have $\vert x\vert =1-\epsilon$ and $1-\epsilon<\vert y\vert<1+\epsilon$):
\begin{enumerate}[label=\arabic{*}.,ref=\arabic{*}]
   \item\label{it:neigh-case_1}$x\notin O_\delta(1)$;
   \item\label{it:neigh-case_2}$x\in O_\delta(1)$ and $y\notin U_\delta(1)$;
   \item\label{it:neigh-case_3}$x\in O_\delta(1)$ and $y\in U_\delta(1)$.
\end{enumerate}

We first consider case \ref{it:neigh-case_1} and prove that no point will contribute to the computation of the residues. By Lemma \ref{lem:new:add} \ref{it1:lem:new:add}, for any $\vert x\vert=\vert y\vert =1$ with $x\notin O_\delta(1)$, the continuous function $Q(x,y)$ is non-zero. By continuity, we also have $Q(x,y)\neq 0$ for any $1-\epsilon<\vert x\vert<1+\epsilon$, $1-\epsilon<\vert y\vert<1+\epsilon$ with $x\notin O_\delta(1)$. Case \ref{it:neigh-case_2} is handled symmetrically. 

In case \ref{it:neigh-case_3}, then using Lemma \ref{lem:new:add} \ref{it2:lem:new:add}, there is only one potential zero of $Q(x,y)$, namely $Y_{0}(x)$. We take $\delta$ sufficiently small to ensure that for all $\ell=0,\ldots, k_0-1$, $g_\ell(x)(1-x)$ and $\widetilde g_k(Y_{0}(x))(1-Y_{0}(x))$ are analytic in $O_\delta(1)$, see Corollary \ref{cor:main-1-bis} and Lemma \ref{lem:new:add} \ref{it2:lem:new:add}, and we show that $Y_{0}(x)$ is not a pole. Our key argument is that $y=Y_{0}(x)$ will also be a zero of the numerator, and so a removable singularity of the  integrand for any $x \in O_\delta(1)\setminus \{x=1\}$.

Let us introduce the domain $V$ as in \eqref{eq:def_V} (see Lemma \ref{lem:new:add} \ref{it4:lem:new:add}). Since $\vert x \vert <1$ and $\vert Y_{0}(x)\vert <1$  on this set, the main equation \eqref{eq:main_func_eq} implies 
\begin{equation}
\label{hh}
    \sum_{\ell=0}^{k_0-1}q_\ell'(x,Y_{0}(x)) g_\ell(x)+ \sum_{k=0}^{k_0-1} q_k''(x,Y_{0}(x)) \widetilde g_k(Y_{0}(x))+f_{i_0,j_0}(x,Y_{0}(x))=0, \quad \forall x \in V.
\end{equation}
Furthermore, the left-hand side of \eqref{hh} multiplied by the factor $(1-x)(1-Y_{0}(x))$ is an analytic function in $O_\delta(1)$, which equals zero in the domain $V\subset O_\delta(1)$. Then, by the principle of analytic continuation, the left-hand side of \eqref{hh} multiplied by  $(1-x)(1-Y_{0}(x))$ equals zero in the whole of  $O_\delta(1)$. Hence the left-hand side of \eqref{hh} is equal to zero in $O_\delta(1)\setminus\{x=1\}$.

In the first integral in \eqref{eq:Green_sum_terms}, it remains to compute the residues at the poles of the numerator. By Corollary \ref{cor:main-1-bis}, there exists only one pole of the numerator, namely, $y=1$, which is a pole of $\widetilde g_k(y)$ for all $k=0,\ldots, k_0-1$ (by \eqref{eq:relation_GG}, $\widetilde g_k$ and ${\cal G}_{(i,j)\to (k, \cdot)}$ have the same residue at $1$, namely $\pi_1(k) \P_{(i_0,j_0)} \bigl({\cal N}_1 < \infty\bigr)/V_1$). Thus we get 
\begin{multline*}
 g\bigl((i_0,j_0) \to (i,j)\bigr)
= \frac{\mathbb P_{(i_0,j_0)}\bigl(\mathcal N_1<\infty\bigl)}{V_1} \frac{1}{2\pi i } \int\limits_{\vert x\vert =1-\epsilon}\frac{ \sum_{k=0}^{k_0-1} \pi_1(k) q''_k(x,1)  }{x^{i-k_0+1} Q(x,1)}dx
\\
+ \frac{1}{(2\pi i)^2} \iint\limits_{\substack{\vert x\vert =1-\epsilon\\\vert y\vert =1+\e}} \frac{\sum_{\ell=0}^{k_0-1}q_\ell'(x,y) g_\ell(x)+ \sum_{k=0}^{k_0-1} q_k''(x,y) \widetilde g_k(y)+f_{i_0,j_0}(x,y)
  }{ x^{i+k_0-1}y^{j+k_0-1} Q(x,y) } dxdy,
\end{multline*}
where we inverted the order of integration in the second term. As proved in Lemma \ref{lem:inv_mes}, the first term is the integral of an analytic function in the annulus $\{1-\epsilon<\vert x\vert <1+\epsilon\}$, so that it equals the same integral over $\{\vert x\vert =1+\epsilon\}$, which is nothing else but the invariant measure announced in the theorem, see \eqref{eq:exact_expr_1D}. 

\medskip 
 
\textit{Step \ref{it:proof_main_3}.} We proceed with the second term of \eqref{eq:Green_sum_terms} as previously:
\begin{multline}
\label{eq:Green_sum_terms2}
 g\bigl((i_0,j_0) \to (i,j)\bigr)= \frac{\mathbb P_{(i_0,j_0)} \bigl(\mathcal N_1<\infty\bigr)}{V_1}\pi_1(i)\\
+ 
 \frac{1}{2\pi i} \int\limits_{\vert y\vert =1+\epsilon} \sum_{ x :1-\epsilon<\vert x\vert < 1+\epsilon} \Res \frac{\sum_{\ell=0}^{k_0-1}q_\ell'(x,y) g_\ell(x)+ \sum_{k=0}^{k_0-1} q_k''(x,y) \widetilde g_k(y)+f_{i_0,j_0}(x,y)
  }{ x^{i+k_0-1}y^{j+k_0-1} Q(x,y) }dy \\
+ \frac{1}{(2\pi i)^2} \iint\limits_{\vert x\vert =\vert y\vert=1+\epsilon} \frac{\sum_{\ell=0}^{k_0-1}q_\ell'(x,y) g_\ell(x)+ \sum_{k=0}^{k_0-1} q_k''(x,y) \widetilde g_k(y)+f_{i_0,j_0}(x,y)
  }{ x^{i+k_0-1}y^{j+k_0-1} Q(x,y)} dxdy. 
\end{multline}

\textit{Step \ref{it:proof_main_4}.} Using a symmetric reasoning as in Step \ref{it:proof_main_2}, we obtain
\begin{multline}
\label{eq:Green_sum_terms3}
    g\bigl((i_0,j_0) \to (i,j)\bigr)= \frac{\mathbb P_{(i_0,j_0)} \bigl(\mathcal N_1<\infty\bigr)}{V_1}\pi_1(i) +  \frac{\mathbb P_{(i_0,j_0)} \bigl(\mathcal N_2<\infty\bigr)}{V_2}\pi_2(j)\\
+ \frac{1}{(2\pi i)^2} \iint\limits_{\vert x\vert =\vert y\vert=1+\epsilon} \frac{\sum_{\ell=0}^{k_0-1}q_\ell'(x,y) g_\ell(x)+ \sum_{k=0}^{k_0-1} q_k''(x,y) \widetilde g_k(y)+f_{i_0,j_0}(x,y)}{ x^{i+k_0-1}y^{j+k_0-1} Q(x,y)} dxdy.
\end{multline}

\textit{Step \ref{it:proof_main_5}}. Let $i,j \to \infty$. We prove that the last integral above is  
$o(x_{1}^{-i} + y_{1}^{-j})$. Let us write the integral in the second line of \eqref{eq:Green_sum_terms3} as 
\begin{multline}
\label{eq:three_terms}
 \frac{1}{(2\pi i)^2} \sum_{\ell=0}^{k_0-1} \int\limits_{\vert x\vert =1+\epsilon} \frac{g_\ell(x)}{x^{i+k_0-1}}  \int\limits_{\vert y\vert =1+\e}
\frac{q_\ell'(x,y)}{y^{j+k_0-1}Q(x,y)}dy dx \\+ \frac{1}{(2\pi i)^2} \sum_{k=0}^{k_0-1} \int\limits_{\vert y\vert =1+\epsilon} \frac{\widetilde g_k(y)}{y^{j+k_0-1}}  \int\limits_{\vert x\vert =1+\e}
\frac{q_k''(x,y)}{x^{i+k_0-1}Q(x,y)}dx dy \\+  \frac{1}{(2\pi i)^2} \int\limits_{\vert y\vert =1+\epsilon} \frac{1}{y^{j+k_0-1}}  \int\limits_{\vert x\vert =1+\e}
\frac{f_{i_0, j_0}(x,y)}{x^{i+k_0-1}Q(x,y)}dx dy.
\end{multline}
Using Lemma \ref{lem:new:add} \ref{it7:lem:new:add}, we may rewrite the first term of \eqref{eq:three_terms} as 
\begin{equation}
\label{eq:first_three_terms}
     \frac{1}{(2\pi i)^2} \sum_{\ell=0}^{k_0-1} \int\limits_{\vert x\vert =1+\epsilon} \frac{g_\ell(x)}{x^{i+k_0-1}}  \int\limits_{\vert y\vert =Y_1(1+\epsilon)-\eta}
\frac{q_\ell'(x,y)}{y^{j+k_0-1}Q(x,y)}dy dx,
\end{equation}
for any $\eta>0$. The integral in \eqref{eq:first_three_terms} is bounded from above by (up to a multiplicative constant)
\begin{equation*}
     (1+\epsilon)^{-i}(Y_1(1+\epsilon)-\eta)^{-j}=o(x_{1}^{-i} + y_{1}^{-j}),
\end{equation*}
where the last equality is a consequence of Lemma \ref{lem:new:add} \ref{it8:lem:new:add}, since $\eta>0$ may be taken as small as we want. We conclude similarly with the second and third terms of \eqref{eq:three_terms}.

\medskip

\textit{Step \ref{it:proof_main_6}}. Using \eqref{eq:Green_sum_terms3} together with the computations just above, we deduce that
\begin{equation*}
     g\bigl((i_0,j_0) \to (i,j)\bigr)= \frac{\mathbb P_{(i_0,j_0)} \bigl(\mathcal N_1<\infty\bigr)}{V_1}\pi_1(i) +  \frac{\mathbb P_{(i_0,j_0)} \bigl(\mathcal N_2<\infty\bigr)}{V_2}\pi_2(j)+o(x_{1}^{-i} + y_{1}^{-j}).
\end{equation*}
Finally, we use the fact that in Lemma \ref{lem:inv_mes}, the constant $A_1$ in the stationary measure asymptotics is non-zero, so we obtain the proof of Equation \eqref{eq:mainres} of Theorem~\ref{thm:main-3}.
\end{proof}

\section{Glossary of the hitting times}
\label{sec:hitting_times}

Throughout the paper, we introduced and made use of the following hitting times:
\begin{equation*}
\left\{\begin{array}{rcll}
  \tau(k,\ell)  & = & \inf\{n>0: Z(n) = (k,\ell)\},&\quad(\text{see } \eqref{eq:def_tau(k,l)}),\\
  \tau & = &\inf\{n>0: Y(n) < k_0\},&\quad(\text{see } \eqref{eq:def_tau}),\\
  \tau^{\textnormal{loc}}(k,\ell) &=& \inf\{n>0: Z_1(n) = (k,\ell)\},&\quad(\text{see } \eqref{eq:def_tau_1(k,l)}),\\
  \tau^{\textnormal{loc}}_1 &=& \inf\{n>0: Y_1(n) < k_0\},&\quad(\text{see } \eqref{eq:def_tau_1}),\\
  t_1(i) &=& \inf\{n > 0 : {\cal M}(n) = i\},&\quad(\text{see } \eqref{eq:def_t_n}),\\
  t_{n+1}(i) &=& \inf\{p > t_n(i) : {\cal M}(p) = i\},&\quad(\text{see } \eqref{eq:def_t_n}),\\
  {\mathcal  T} &=&  \inf\{n \geq 0: {\cal A}(n) < k_0\},&\quad(\text{see } \eqref{eq:def_cal_T}),\\
   T(k) &=& \inf\bigl\{n\geq 1: {\cal M}(n)\in{\cal E}_0\cup\{k\}\bigr\},&\quad(\text{see } \eqref{eq:def_T(k)}),\\ 
{T}_{k} &=& \inf\{n > 0: X(n) \leq (k_0-1)\vee k\},& \quad(\text{see } \eqref{eq:T1_T2_k}), \\ 
{T}_{k}^{\textnormal{loc}} &=& \inf\{n > 0: X_1(n) \leq (k_0-1)\vee k\},& \quad(\text{see } \eqref{eq:def_T_k_loc} ).
\end{array}
\right.
\end{equation*}

\subsubsection*{Acknowledgments}
We thank Gerold Alsmeyer, Onno Boxma and Dmitry Korshunov for bibliographic suggestions. The last author would like to warmly thank Elisabetta Candellero, Steve Melczer and Wolfgang Woess for many discussions at the initial stage of the project. We thank the associate editor and the two anonymous referees for their very careful readings and their numerous suggestions.


\begin{thebibliography}{10}
\bibliographystyle{plain}

\bibitem{AlDo-01}
L. Alili and R. A. Doney (2001).
Martin boundaries associated with a killed random walk.
\textit{Ann. Inst. H. Poincar\'e Probab. Statist.} \textbf{37} 313--338

\bibitem{Al-94} 
G. Alsmeyer (1994).
On the Markov renewal theorem. 
\textit{Stochastic Process. Appl.} \textbf{50} 37--56

\bibitem{An-88}
A. Ancona (1988).
Positive harmonic functions and hyperbolicity. 
\textit{Potential theory---surveys and problems (Prague, 1987)}, 1--23,
Lecture Notes in Math., 1344, \textit{Springer, Berlin}

\bibitem{Asmussen} 
S. Asmussen (2003). 
\textit{Applied Probability and Queues.} 
Second edition. Springer-Verlag, New York

\bibitem{BaFl-02}
C. Banderier and P. Flajolet (2002).
Basic analytic combinatorics of directed lattice paths.
\textit{Comput. Sci.} {\bf281} 37--80

\bibitem{BoBMMe-21}
A. Bostan, M. Bousquet-M\'elou and S. Melczer (2021).
Counting walks with large steps in an orthant.
\textit{J. Eur. Math. Soc. (JEMS)} \textbf{23} 2221--2297

\bibitem{BMMi-10}
M. Bousquet-M\'elou and M. Mishna (2010).
Walks with small steps in the quarter plane.
\textit{Algorithmic probability and combinatorics}, 1--39, Contemp. Math., 520, \textit{Amer. Math. Soc., Providence, RI}

\bibitem{BoLo-96}
O. J. Boxma and V. I. Lotov (1996).
On a class of one-dimensional random walks.
\textit{Markov Process. Related Fields} {\bf2} 349--362

\bibitem{Ca-71}
P. Cartier (1971). 
Fonctions harmoniques sur un arbre.
\textit{Symposia Mathematica, Vol. IX} (Convegno di Calcolo delle Probabilit\`a, INDAM, Rome, 1971) 203--270

\bibitem{CoBo-83}
J. W. Cohen and O. J. Boxma (1983).
\textit{Boundary value problems in queueing system analysis.}
North-Holland Mathematics Studies, 79. North-Holland Publishing Co., Amsterdam

\bibitem{Co-92}
J. W. Cohen (1992).
\textit{Analysis of random walks.}
Studies in Probability, Optimization and Statistics, 2. IOS Press, Amsterdam


\bibitem{DeKoWa-19}
D. Denisov, D. Korshunov and V. Wachtel (2019).
Markov chains on $\mathbb Z^+$: analysis of stationary measure via harmonic functions approach.
\textit{Queueing Syst.} {\bf91} 26--295

\bibitem{DeWa-15}
D. Denisov and V. Wachtel (2015).
Random walks in cones.
\textit{Ann. Probab.} \textbf{43} 992--1044

\bibitem{Do-59}
J. L. Doob (1959).
Discrete potential theory and boundaries.
\textit{J. Math. Mech.} \textbf{8} 433--458

\bibitem{DuRaTaWa-22}
J. Duraj, K. Raschel, P. Tarrago and V. Wachtel (2022).
Martin boundary of random walks in convex cones.
\textit{Ann. H. Lebesgue} \textbf{5} 559--609 

\bibitem{FaIa-79}
G. Fayolle and R. Iasnogorodski (1979).
Two coupled processors: the reduction to a Riemann-Hilbert problem.
\textit{Z. Wahrsch. Verw. Gebiete} \textbf{47} 325--351

\bibitem{FaIaMa-17}
G. Fayolle, R. Iasnogorodski and V. Malyshev (2017).
\textit{Random walks in the quarter plane. Algebraic methods, boundary value problems, applications to queueing systems and analytic combinatorics.}
Second edition. Probability Theory and Stochastic Modelling, 40. Springer, Cham

\bibitem{FaMaMe-95}
G. Fayolle, V. Malyshev and M. Menshikov (1995).
\textit{Topics in the constructive theory of countable Markov chains.} 
Cambridge University Press, Cambridge

\bibitem{FaRa-15}
G. Fayolle and K. Raschel (2015).
About a possible analytic approach for walks in the quarter plane with arbitrary big jumps.
\textit{C. R. Math. Acad. Sci. Paris} {\bf353} 89--94


\bibitem{He-63}
P.-L. Hennequin (1963).
Processus de Markoff en cascade.
\textit{Ann. Inst. H. Poincar\'e} {\bf18} 109--195

\bibitem{IR-08}
I. Ignatiouk-Robert (2008).
Martin boundary of a killed random walk on a half-space.
\textit{J. Theoret. Probab.} \textbf{21} 35--68

\bibitem{IR-10a}
I. Ignatiouk-Robert (2010).
$t$-Martin boundary of reflected random walks on a half-space.
\textit{Electron. Commun. Probab.} {\bf15} 149--161

\bibitem{IR-10b}
I. Ignatiouk-Robert (2010).
Martin boundary of a reflected random walk on a half-space.
\textit{Probab. Theory Related Fields} {\bf148} 197--245

\bibitem{IR-20}
I. Ignatiouk-Robert (2020).
Martin boundary of a killed non-centered random walk in a general cone.
\textit{Preprint arXiv:2006.15870}

\bibitem{IRLo-10}
I. Ignatiouk-Robert and C. Loree (2010).
Martin boundary of a killed random walk on a quadrant.
\textit{Ann. Probab.} {\bf38} 1106--1142



\bibitem{KuMa-98}
I. Kurkova and V. Malyshev (1998).
Martin boundary and elliptic curves.
\textit{Markov Process. Related Fields} {\bf4} 203--272

\bibitem{KuRa-11}
I. Kurkova and K. Raschel (2011).
Random walks in $\mathbb Z_+^2$ with non-zero drift absorbed at the axes. 
\textit{Bull. Soc. Math. France} \textbf{139} 341--387

\bibitem{KuSu-03}
I. Kurkova and Y. M. Suhov (2003).
Malyshev's theory and JS-queues. Asymptotics of stationary probabilities.
\textit{Ann. Appl. Probab.} \textbf{13} 1313--1354

\bibitem{Ma-70}
V. A. Malyshev (1970).
\textit{Random walks. The Wiener-Hopf equation in a quadrant of the plane. Galois automorphisms} (Russian).
Izdat. Moskov. Univ., Moscow

\bibitem{Ma-72}
V. A. Malyshev (1972).
An analytic method in the theory of two-dimensional positive random walks (Russian).
\textit{Sibirsk. Mat. \v{Z}.} \textbf{13} 1314--1329, 1421

\bibitem{Ma-73}
V. A. Malyshev (1973).
Asymptotic behavior of the stationary probabilities for two-dimensional positive random walks (Russian).
\textit{Sibirsk. Mat. \v{Z}.} \textbf{14} 156--169, 238

\bibitem{Ma-41}
R. S. Martin (1941).
Minimal positive harmonic functions.
\textit{Trans. Amer. Math. Soc.} \textbf{49} 137--172

\bibitem{MeTw-93}
S. P. Meyn and R. L. Tweedie (1993).
\textit{Markov chains and stochastic stability}.
Communications and Control Engineering Series.
\textit{Springer-Verlag London, Ltd., London}

\bibitem{NeSp-66}
P. Ney and F. Spitzer (1966).
The Martin boundary for random walk.
\textit{Trans. Amer. Math. Soc.} \textbf{121} 116--132

\bibitem{PiWo-87}
M. Picardello and W. Woess (1987).
Martin boundaries of random walks: ends of trees and groups.
\textit{Trans. Amer. Math. Soc.} {\bf302} 185--205

\bibitem{PiWo-90}
M. Picardello and W. Woess (1990).
Examples of stable Martin boundaries of Markov chains.
In \textit{Potential theory (Nagoya, 1990)} 261--270, de Gruyter, Berlin, 1992

\bibitem{PiWo-92}
M. Picardello and W. Woess (1992).
Martin boundaries of Cartesian products of Markov chains.
\textit{Nagoya Math. J.} {\bf128} 153--169

\bibitem{Prabhu-Tang-Zhu} 
N. U. Prabhu, L.C. Tang and Y. Zhu (1991).
Some new results for the Markov random walk. 
\textit{J. Math. Phys. Sci.} {\bf25} 635--663

\bibitem{Se-81}
E. Seneta (1981).
\textit{Non-negative matrices and Markov chains.}
Second edition. Springer Series in Statistics. Springer-Verlag, New York

\bibitem{Wo-00}
W. Woess (2000).
\textit{Random walks on infinite graphs and groups.}
Cambridge Tracts in Mathematics, 138. 
\textit{Cambridge University Press}, Cambridge

\end{thebibliography}
\end{document}